\documentclass{article}

\usepackage{microtype}
\usepackage[a4paper, top=2.5cm,bottom=2.5cm,inner=2.5cm,outer=2.5cm,twoside]{geometry}

\usepackage{amsmath}
\usepackage{amssymb}
\usepackage{amsthm}
\usepackage{mathtools}

\newtheorem{thm}{Theorem}
\newtheorem{rem}{Remark}
\newtheorem{assump}{Assumption}
\newtheorem{lem}[thm]{Lemma}
\newtheorem{defn}[thm]{Definition}
\newtheorem{cor}[thm]{Corollary}

\usepackage[latin1]{inputenc}

\usepackage{hyperref}
\hypersetup{plainpages=false, linktocpage=true, colorlinks=true, breaklinks=true, linkcolor=black, menucolor=black, urlcolor=black, citecolor=black}

\usepackage{graphicx}
\usepackage{floatflt}
\usepackage{subfigure}
\graphicspath{{images/}}
\usepackage[all]{hypcap}

\usepackage{booktabs}
\usepackage{multirow}
\usepackage{placeins}

\newcommand{\R}{ \ensuremath{\mathbb{R}} }
\newcommand{\N}{ \ensuremath{\mathbb{N}} }
\newcommand{\Landau}{\ensuremath{\mathcal O}}

\newcommand{\jump}[1]{\ensuremath{{|\!\![} #1 {]\!\!|}}}
\newcommand{\bigjump}[1]{\ensuremath{\Big[\!\!\Big[ #1 \Big]\!\!\Big]}} 

\newcommand{\nach}[1]{\ensuremath{\,\mathrm{d} #1} }
\DeclareMathOperator{\spn}{span}
\DeclareMathOperator{\meas}{meas}

\DeclareMathOperator*{\argmin}{arg\,min}

\newcommand{\eps}{\varepsilon}
\newcommand{\E}{\mathrm{e}}
\newcommand{\D}{\mathrm{d}}

\title{Macro-element interpolation on tensor product meshes}
\author{Martin Schopf\thanks{Institut f\"ur Numerische Mathematik, Technische Universität Dresden, 01062 Dresden, Germany.\newline \textsf{martin.schopf@tu-dresden.de}}}

\date{\today}

\begin{document}

\maketitle

\begin{abstract} A general theory for obtaining anisotropic interpolation error estimates for macro-element interpolation is developed revealing general construction principles. We apply this theory to interpolation operators on a macro type of biquadratic $C^1$ finite elements on rectangle grids which can be viewed as a rectangular version of the $C^1$ Powell-Sabin element. This theory also shows how interpolation on the Bogner-Fox-Schmidt finite element space (or higher order generalizations) can be analyzed in a unified framework. Moreover we discuss a modification of Scott-Zhang type giving optimal error estimates under the regularity required without imposing quasi uniformity on the family of macro-element meshes used. We introduce and analyze an anisotropic macro-element interpolation operator, which is the tensor product of one-dimensional $C^1-P_2$ macro interpolation and $P_2$ Lagrange interpolation. These results are used to approximate the solution of a singularly perturbed reaction-diffusion problem on a Shishkin mesh that features highly anisotropic elements. Hereby we obtain an approximation whose normal derivative is continuous along certain edges of the mesh, enabling a more sophisticated analysis of a continuous interior penalty method in another paper.
\end{abstract}
{\it AMS subject classification (2010):}  65M60, 65N30\\
{\it Key words:} Anisotropic interpolation error estimates, differentiable finite elements, FEM, macro, Hermite interpolation, quasiinterpolation, Shishkin mesh

\section{Introduction}

There is a high interest in differentiable finite elements and their corresponding interpolation operators as these are used for instance in the construction and analysis of methods for higher order problems like the biharmonic equation. On a triangular mesh the fifth degree Argyris element and its reduced version --- the Bell element --- are most popular. However, they are rarely used as they introduce a large number of degrees of freedom. In fact, \v{Z}eni\v{z}ek \cite{Z70} showed that on a triangular element with polynomial shape functions at least 18 degrees of freedom are needed to grant the $C^1$ property. In this respect the Bell element can be considered optimal.

The desire for reducing the number of degrees of freedom used (and therefore the polynomial degree) lead to the construction of macro-elements in 1960s and 1970s. Let us mention the cubic Hsieh-Clough-Tocher macro-element \cite{Cia78} and the quadratic Powell-Sabin macro-element \cite{PS77}. In the latter, each base triangle is split into six sub-triangles that share an inner point (for instance the center of the inscribed circle) of the base triangle. The inner degrees of freedom are then eliminated by the $C^1$ property.

While there is a huge amount of literature for triangular macro-elements (see for instance the survey article \cite{N98} and the references therein), there appears to be only one publication \cite{HHZ2011} dealing with rectangular ones. Moreover, to the knowledge of the author, there appears to be no paper dealing with anisotropic interpolation error estimates for macro-element interpolation, i.e.~up to now macro-element interpolation has only been considered on quasi-uniform meshes. However, one can certainly improve the approximation quality by allowing elements with an arbitrarily high aspect ratio in certain cases. This benefit becomes obvious if the underlaying domain or the function to be approximated has anisotropic features (like layers).

In Section \ref{sec:univariate} of this paper we shall briefly introduce the concept of $C^1-P_2$ macro-interpolation in the 1D case and fix some notation.

The following Section \ref{sec:main_macro} starts by showing how the 1D $C^1-P_2$ macro-element extends to the 2D $C^1-Q_2$ macro-element on tensor product meshes. Then a general theory for obtaining anisotropic interpolation error estimates for macro-element interpolation is developed and general construction principles are revealed. This theory is then applied in order to analyze the $C^1-Q_2$ macro-element interpolation operator $\Pi$ as well as some reduced counterpart.

Thereafter we discuss a modification of $\Pi$ of Scott-Zhang \cite{SZ90} type in Subsection \ref{subsec:SZ_macro} giving optimal error estimates under the regularity required. The price to pay is that not all linear functionals that define this modified operator are local, i.e. in order to obtain the value of the quasi-interpolant on a base macro-element $M$ some averaging process of the data on a macro-element edge that does not necessarily belong to $M$ is needed. This causes some difficulties because quasi-interpolation operators of similar type are mostly studied on quasi-uniform meshes.

We summarize our results concerning $C^1$ (quasi-)interpolation in Subsection \ref{subsec:summary_macro} and cite some results of the literature.

In Section \ref{sec:anisomacro_macro} we introduce and analyze an anisotropic macro-element interpolation operator. Basically, this operator is the tensor product of one-dimensional $C^1-P_2$ macro-interpolation and $P_2$ Lagrange interpolation.

We conclude this paper with Section \ref{sec:Shishkin_macro} in which we apply the results of the (Sub-)Sections \ref{subsec:SZ_macro} and \ref{sec:anisomacro_macro} in order to approximate the solution of a singularly perturbed reaction-diffusion problem on a Shishkin mesh that features anisotropic elements, i.e.~elements with an unbounded aspect ratio for $\eps \rightarrow 0$. Hereby we obtain an approximation whose normal derivative is continuous along certain edges of the mesh, enabling a more sophisticated analysis of a continuous interior penalty method in the next chapter.

\section{\texorpdfstring{Univariate $C^1-P_2$ macro-element interpolation}{Univariate C1-P2 macro-element interpolation}}\label{sec:univariate}

Consider the 1D Hermite interpolation problem on the interval $[-1,1]$: Let $u$ be a real function over $[-1,1]$ such that $u(\pm 1),\, u'(\pm 1) \in \R$ can be defined. Find $s \in C^1[-1,1]$, such that
\begin{gather}\label{eq:hermiteProb}
	s(\pm 1) = u(\pm 1),\qquad s'(\pm 1) = u'(\pm 1).
\end{gather}

In 1983 Schumaker \cite{S83} observed that while the Hermite interpolation problem considered is only solvable for a quadratic polynomial $s \in P_2[-1,1]$ if and only if
\begin{gather*}
	u'(-1)+u'(1) = u(1)-u(-1),
\end{gather*}
there is always a solution in the space of quadratic splines with one simple knot. We may choose $x=0$ as this knot and introduce the spline space 
\begin{gather*}
	S^2 \coloneqq \big \{ v \in C^1[-1,1]\,:\, \left. v \right|_{T} \in P_2(T),\; T \in \{[-1,0],[0,-1]\} \big\}.
\end{gather*}
Of course other choices for the additional knot are possible. This parameter can be used to grant additional properties of the underlaying interpolation operator, see \cite{S83}.

A function $s$ that is a quadratic polynomial on each of the intervals $[-1,0]$ and $[0,1]$ can be characterized by six parameters of which two are determined by the $C^1$ property at zero. Hence, the remaining four parameters of a function $s \in S^2$ may be chosen in such a way that \eqref{eq:hermiteProb} is fulfilled. In fact, a simple calculation shows that
\begin{gather}\label{eq:lagrangeSpline}
	s(x) = \sum_{i = \pm 1} \left( u(i) \hat\varphi_{i}(x) + u'(i) \hat\psi_{i}(x) \right),\quad x\in [-1,1]
\end{gather}
is the unique solution of \eqref{eq:hermiteProb} in $S^2$. Here $\hat\varphi_{\pm 1}$ and $\hat\psi_{\pm 1} \in S^2$ denote the Lagrangian basis functions
\begin{gather}
\label{eq:basis}
\begin{alignedat}{2}
	\hat\varphi_{-1} (x) &= \frac{(x-1)^2}{2} - \left\{ \begin{alignedat}{2} &x^2,&\quad x &\in [-1,0], \\ &0,&\quad x &\in [0,1], \end{alignedat} \right.\qquad & 
	\hat\varphi_{1} (x) &= \frac{(x+1)^2}{2} - \left\{ \begin{alignedat}{2} &0,&\quad x &\in [-1,0], \\ &x^2,&\quad x &\in [0,1], \end{alignedat} \right.\\
	\hat\psi_{-1} (x) &= \frac{(x-1)^2}{4} - \left\{ \begin{alignedat}{2} &x^2,&\quad x &\in [-1,0], \\ &0,&\quad x &\in [0,1], \end{alignedat} \right.\qquad &
	\hat\psi_{1} (x) &= -\frac{(x+1)^2}{4} + \left\{ \begin{alignedat}{2} &0,&\quad x &\in [-1,0], \\ &x^2,&\quad x &\in [0,1], \end{alignedat} \right.
\end{alignedat}
\end{gather}
i.e.~these spline functions fulfill the conditions
\begin{align*}
	\hat\varphi_{-1}(-1) &= \hat\varphi_{1}(1) = \hat\psi_{-1}'(-1) = \hat\psi_{1}'(1) = 1,\\
	\hat\varphi_{1}(-1) = \hat\psi_{-1}(-1) &= \hat\psi_{1}(-1) = \hat\varphi_{-1}(1) = \hat\psi_{-1}(1) = \hat\psi_{1}(1) = 0,\\
	\hat\varphi_{-1}'(-1) = \hat\varphi_{1}'(-1) &= \hat\psi_{1}'(-1) = \hat\varphi_{-1}'(1) = \hat\varphi_{1}'(1) = \hat\psi_{-1}'(1) = 0.
\end{align*}
For a graphical representation of these functions, see Figure \ref{fig:splinebase}.

\begin{figure}
	\centering
	\includegraphics[width=.6\textwidth]{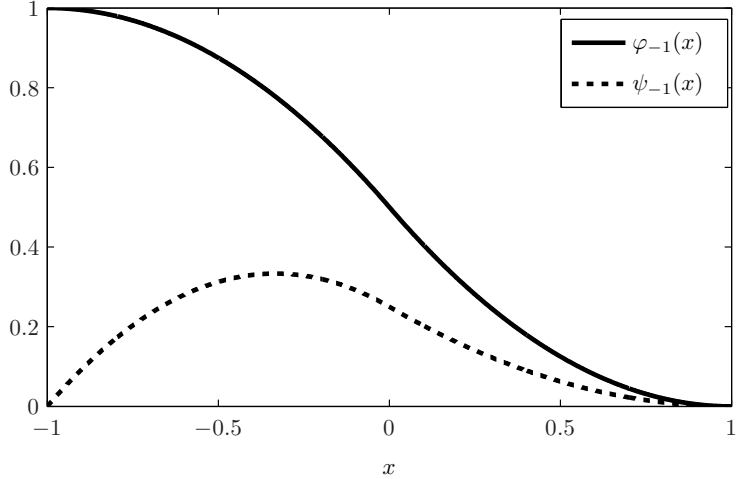}
	\caption{Lagrangian basis functions $\varphi_{-1}$ and $\psi_{-1}$}
	\label{fig:splinebase}
\end{figure}

Based on the symmetry of the subproblem defining the basis functions we observe 
\begin{gather*}
	\hat\varphi_{-1}(x) = \hat\varphi_1(-x)\qquad \text{and} \qquad	\hat\psi_{-1}(x) = -\hat\psi_1(-x)\quad \forall x \in [-1,1].
\end{gather*}
Moreover, $\hat\varphi_{\pm 1}'$ are even functions, i.e.
\begin{gather*}
	\hat\varphi_{\pm 1}'(x) = \hat\varphi_{\pm 1}'(-x) \quad \forall x \in [-1,1].
\end{gather*}

From these properties we can deduce that $\hat\varphi_1'(x) = -\hat\varphi_{-1}'(-x) = \hat\varphi_{-1}'(x)$ for all $x \in [-1,1]$. Hence, similar to a cubic polynomial the derivative $s'$ of a spline $s \in S^2$ is an element of a three dimensional vector space. Since the second derivative of the spline considered is piecewise constant, it belongs to a two dimensional space.

This fact can nicely be seen if we switch from the Lagrangian representation \eqref{eq:lagrangeSpline} of the solution of \eqref{eq:hermiteProb} to its Newtonian one. Based on $\hat\psi_1(\pm 1) = \hat\psi_1'(-1) = 0$ we observe, that
\begin{gather}\label{eq:newtonSpline}
	s(x) = u[-1] + u[-1,-1](x+1) + u[-1,-1,1](x+1)^2 + u[-1,-1,1,1] 4 \hat\psi_1(x).
\end{gather}
Here $u[x_0,\dots,x_N]$ are the well known divided differences of order $N$ of $u$ with possibly coincident knots $x_0 \le x_1 \le \dots \le x_N$, recursively defined by
\begin{gather*}
	u[x_i] \coloneqq u(x_i)\quad \text{and}\quad u[x_0,\dots,x_N] \coloneqq \left\{ \begin{alignedat}{2}& \frac{1}{N!} u^{(N)}(x_0),\quad & &\text{if } x_0 = \dots = x_N, \\
	& \frac{u[x_1,\dots,x_N]-u[x_0,\dots,x_{N-1}]}{x_N - x_0},\quad & &\text{else.} \end{alignedat} \right.
\end{gather*}

A simple calculation shows that
\begin{gather} \label{eq:dividiff}
	\begin{gathered}
	u[-1] = u(-1), \quad u[-1,-1] = u'(-1),\quad u[-1,-1,1] = \frac{1}{4}\big( u(1) - u(-1) \big) - \frac{1}{2} u'(-1),\\
	u[-1,-1,1,1] = \frac{1}{4} \big( u(-1) - u(1) + u'(-1) + u'(1) \big).
	\end{gathered}
\end{gather}
If we substitute the expressions from \eqref{eq:dividiff} into \eqref{eq:newtonSpline} and expand in terms of $u(\pm 1)$ and $u'(\pm 1)$ we re-obtain the Lagrangian representation \eqref{eq:lagrangeSpline} of $s$. However, the Newtonian form  \eqref{eq:newtonSpline} of $s$ will prove to be very useful in the derivation of anisotropic interpolation error estimates.

\section[\texorpdfstring{$C^1-Q_2$ macro-element interpolation}{C1-Q2 macro-element interpolation}]{\texorpdfstring{$C^1-Q_2$ macro-element interpolation on tensor product meshes}{C1-Q2 macro-element interpolation on tensor product meshes}}\label{sec:main_macro}

One can easily solve the Hermite interpolation problem \eqref{eq:hermiteProb} for a cubic polynomial $s$. Hence, similar to \eqref{eq:basis} a Lagrangian basis for a cubic $C^1$ spline can be obtained associated with the values of the function and its first derivative in the endpoints of the interval considered. It is well-known that the tensor product of this basis of the cubic $C^1$ splines leads to the Bogner-Fox-Schmidt element, which is in fact a $C^1$ element. Here the 16 degrees of freedom are associated with the values $v(V_i)$, the first derivatives $v_x(V_i)$, $v_y(V_i)$ and the mixed derivative $v_{xy}(V_i)$ of a function $v \in Q_3(T)$ at the four vertices $V_i$, $i=1,\dots,4$ of a rectangle $T$, see Figure \ref{fig:bogner-macro}. Note that the restriction of the generated finite element space to any element $T$ is $Q_3(T)$, where $T$ is a rectangle of the underlaying triangulation with sides aligned to the coordinate axes.

\begin{figure}
	\centering
	\includegraphics[width=.9\textwidth]{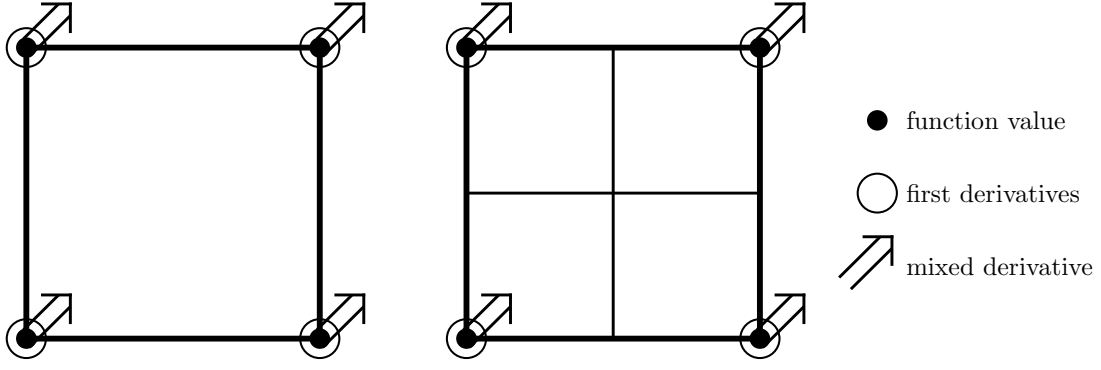}
	\caption{The Bogner-Fox-Schmidt $Q_3$ element (left) and its $Q_2$ analogue on a macro of four elements (right)}
	\label{fig:bogner-macro}
\end{figure}

By analogy with the Bogner-Fox-Schmidt element the tensor product of the basis functions \eqref{eq:basis} generates a $C^1$ macro-element, as well. One obtains 16 basis functions that are piecewise biquadratic:
\begin{gather}\label{eq:basis2d}
	\begin{alignedat}{2}
	\hat\varphi_{i,j}(x,y) &\coloneqq \hat\varphi_i(x) \hat\varphi_j(y),&\quad \hat\phi_{i,j}(x,y) &\coloneqq \hat\psi_i(x) \hat\varphi_j(y),\\
	\hat\chi_{i,j}(x,y) &\coloneqq \hat\varphi_i(x) \hat\psi_j(y),&\quad \hat\psi_{i,j}(x,y) &\coloneqq \hat\psi_i(x) \hat\psi_j(y),
	\end{alignedat}\quad i,j\in\{-1,1\}.
\end{gather}
Whenever definitions are tied to a reference (macro-)element we shall continue to use a hat symbol to emphasize this fact. With the dual functionals
\begin{gather*}
	\begin{alignedat}{2}
		F^{\hat\varphi}_{i,j}(v) &\coloneqq v(i,j),& \quad F^{\hat\phi}_{i,j}(v) &\coloneqq v_x(i,j),\\
		F^{\hat\chi}_{i,j}(v) &\coloneqq v_y(i,j),& \quad F^{\hat\psi}_{i,j}(v) &\coloneqq v_{xy}(i,j),
	\end{alignedat}\quad i,j\in\{-1,1\}.
\end{gather*}
the basis functions obey the Lagrange relation
\begin{gather*}
	F^v_{i,j}(w_{k,\ell}) = \delta_{vw} \delta_{ik} \delta_{j\ell}
\end{gather*}
for $v,w \in \{\hat\varphi, \hat\phi, \hat\chi, \hat\psi \}$ and $i,j,k,\ell \in \{-1,1\}$. We denote by $\hat M$  the reference macro-element which is given as the triangulation of the reference domain $\Lambda \coloneqq [-1,1]^2$ induced by the coordinate axes. On $\hat M$ the four basis functions for $i=j=-1$ associated with the point $(-1,-1)$ are depicted in Figure \ref{fig:basis}.

\begin{figure}
	\centering
	\includegraphics[width=.8\textwidth]{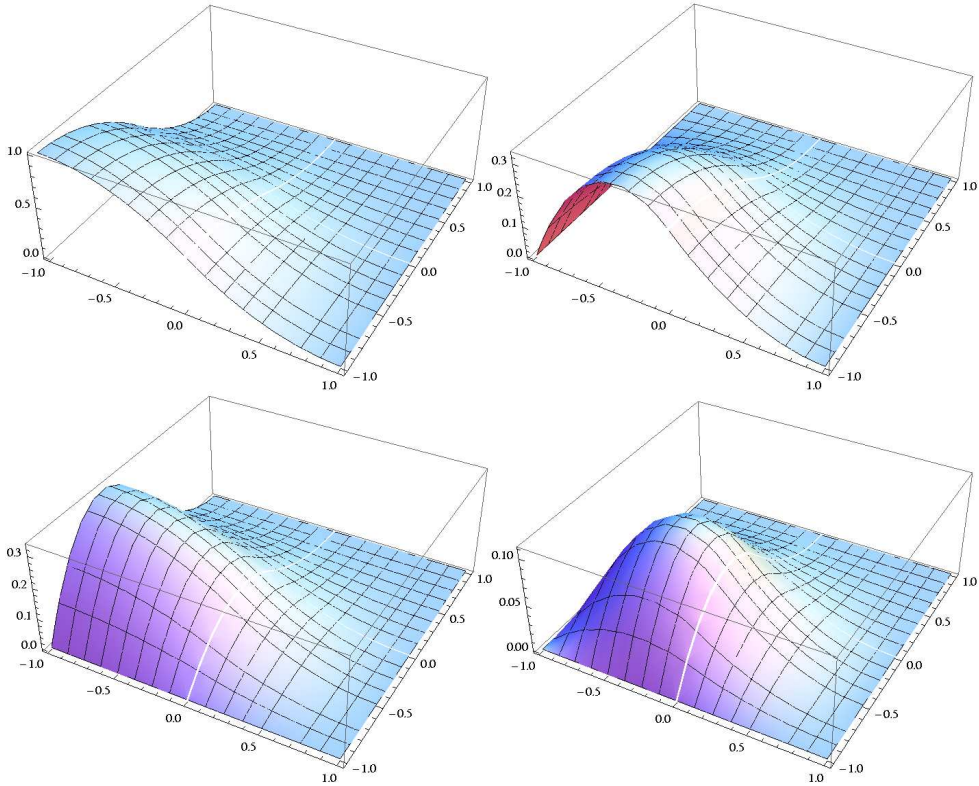}
	\caption{The basis functions $\hat\varphi_{-1,-1}$, $\hat\phi_{-1,-1}$, $\hat\chi_{-1,-1}$, $\hat\psi_{-1,-1}$ on the reference macro-element.}
	\label{fig:basis}
\end{figure}

In a natural way, a biquadratic interpolant $\hat \Pi v \in C^1(\Lambda)$ of a function $v \in C^2(\Lambda)$ is defined by
\begin{gather}\label{eq:interpolant}
	\hat\Pi v = \sum_{i,j\in\{-1,1\}} F^{\hat\varphi}_{i,j}(v) \hat\varphi_{i,j} + F^{\hat\phi}_{i,j}(v) \hat\phi_{i,j} + F^{\hat\chi}_{i,j}(v) \hat\chi_{i,j} + F^{\hat\psi}_{i,j}(v) \hat\psi_{i,j}.
\end{gather}
By affine equivalence, it suffices to define the interpolation operator $\hat\Pi$ on the reference macro-element $\hat M$. Given a rectangular macro-element mesh $\mathcal{M}$ of tensor product type, the value of the interpolant $\Pi v$ of a function $v \in C^2(\overline \Omega)$ in a certain point $(x,y) \in \overline \Omega$ of the physical domain can be obtained by identifying a macro-element $M$ such that $(x,y) \in M$ and performing an affine transformation.

After making an independent construction an excessive search of the literature available showed that the $C^1-Q_2$ macro-element is not new. In fact, it can be traced back to the PhD thesis \cite{A89}. In the work \cite{H97} the thesis \cite{A89} is cited and optimal interpolation error estimates 
\begin{gather*}
	|u - \Pi u|_m \le C h^{3-m} |u|_3
\end{gather*}
for $m=0,1,2$ are proven for $u\in H^3(\Omega) \cap C^2(\overline \Omega)$ and a tensor product triangulation which is required to be quasi-uniform.

Strangely, this idea appears to be unpublished until 2011. In \cite{HHZ2011} the $C^1$ property of the finite element space $V_h$ introduced by the $C^1-Q_2$ macro-element on a tensor product triangulation $\mathcal{T}_h$ of a domain $\Omega$ is shown. Moreover, it is established that $V_h$ coincides with the full $C^1-Q_2$ space, i.e.:
\begin{gather}\label{eq:full}
	V_h = \{v_h \in C^1(\Omega)\,:\, \left. v_h \right|_T \in Q_2(T)\; \forall T \in \mathcal{T}_h\}.
\end{gather}
This appears to be of high interest in certain applications. Finally, optimal interpolation error estimates are derived for an extension of the Girault-Scott operator into the $C^1-Q_2$ finite element space, i.e.~a modification $\tilde \Pi $ of the operator $\Pi$ (defined via an affine transformation as $\hat \Pi$ on the reference macro-element $\hat M$ in \eqref{eq:interpolant}) is obtained in such a way that a function $v \in H^2(\Omega)$ can be interpolated and
\begin{gather*}
	\|v-\tilde \Pi v\|_0 + h |v-\tilde \Pi v |_1 + h^2 |v-\tilde \Pi v |_2 \le C h^2 |v|_2. 
\end{gather*}

However, the analysis in \cite{HHZ2011} of the interpolation error also requires quasi-uniformity of the triangulation $\mathcal{T}_h$, i.e.~it is assumed that there is a positive constant $C>0$ such that for all axis-aligned mesh rectangles $T \in \mathcal{T}_h$ the edge lengths $h_x(T)$ and $h_y(T)$ in $x$- and $y$-direction are equivalent to a global discretization parameter $h$, i.e.
\begin{gather}
	C h \le h_x(T),h_y(T) \le h\quad \forall T \in \mathcal{T}_h.
\end{gather}
On the other hand there are problems that can be treated efficiently if elements with very high aspect ratios are permitted within the triangulation or if edge lengths of neighboring elements are allowed to vary unbounded. As examples, let us mention the approximation of a smooth function over a long and thin domain $\Omega$ or solutions of partial differential equations with anisotropic behavior like layers. Wherefore we ask the question:
{\em Is it possible to prove anisotropic interpolation error estimates for the operator $\Pi$ from \eqref{eq:interpolant} or a modification of it?}

It turns out that the wonderful theory of \cite{AD92,A99} is incapable to handle the analysis of macro-element interpolation. In the following we shall therefore develop a slight modification of it.

\section{A theory on anisotropic macro-element interpolation}\label{sec:macrotheory}

We first introduce some notation, partly adopted from \cite{AD92}. 

Let $\hat M \coloneqq \big\{\hat T_i\big\}_{i=1}^\ell$ be our reference macro-element, i.e.~a triangulation of some reference domain $\Lambda$. For a set of multi-indices $\boldsymbol{P}$ we denote by 
\begin{gather}\label{eq:polyspace}
\boldsymbol{P}(\Lambda) \coloneqq \spn\{\boldsymbol{X} \mapsto \boldsymbol{X}^{\boldsymbol{\alpha}}\,:\, \boldsymbol{\alpha} \in \boldsymbol{P}\} \subset C^\infty(\Lambda)
\end{gather}
the corresponding polynomial function space over $\Lambda$ that is spanned by the monomials $\boldsymbol{X}^{\boldsymbol{\alpha}}$ ($\boldsymbol{\alpha} \in \boldsymbol{P}$).

Here we used standard multi-index notation:
\begin{gather*}
	\boldsymbol{\alpha} = (\alpha_1,\alpha_2),\quad |\boldsymbol{\alpha}| = \alpha_1+\alpha_2,\quad \boldsymbol{X}^{\boldsymbol{\alpha}} = x^{\alpha_1}y^{\alpha_2},\quad \boldsymbol{h}^{\boldsymbol{\alpha}} = h_x^{\alpha_1} h_y^{\alpha_2},\quad
	\boldsymbol{D}^{\boldsymbol{\alpha}} = \frac{\partial^{\alpha_1}}{\partial x^{\alpha_1}} \frac{\partial^{\alpha_2}}{\partial y^{\alpha_2}}.
\end{gather*}

The hull $\boldsymbol{\overline P}$ of $\boldsymbol{P}$ is the set
\begin{gather*}
	\boldsymbol{\overline P} \coloneqq \boldsymbol{P} \cup \{\boldsymbol{\alpha} + \boldsymbol{e}_i\;:\; \boldsymbol{\alpha} \in \boldsymbol{P},\; i=1,2\},
\end{gather*}
where $\{\boldsymbol{e}_1,\boldsymbol{e}_2\}$ denotes the canonical basis of $\R^2$.

Associated with a set of multi-indices $\boldsymbol{P}$ with $\boldsymbol{0} \coloneqq (0,0) \in \boldsymbol{P}$ and $1 \le p \le \infty$ we introduce a norm and a semi-norm on the reference domain $\Lambda$:
\begin{gather*}
	\|v\|_{\boldsymbol{P},p}^p \coloneqq \sum_{\boldsymbol{\alpha} \in \boldsymbol{P}} \|\boldsymbol{D}^{\boldsymbol{\alpha}} v\|_{L_p(\Lambda)}^p,\qquad |v|_{\boldsymbol{\overline P},p}^p \coloneqq \sum_{\boldsymbol{\alpha} \in \boldsymbol{\overline P} \setminus \boldsymbol{P}} \|\boldsymbol{D}^{\boldsymbol{\alpha}} v\|_{L_p(\Lambda)}^2,
\end{gather*}
with obvious modifications for $p=\infty$.
Furthermore, let $H^{\boldsymbol{P}}_p(\Lambda)$ denote the function space
\begin{gather}\label{eq:HPp}
	H^{\boldsymbol{P}}_p(\Lambda) \coloneqq \{ v \in L^1(\Lambda)\; :\; \|v\|_{\boldsymbol{P},p} < \infty \}
\end{gather}
and let $S(\hat M)$ be a spline space such that for $v \in S(\hat M)$ the restrictions $\left.v\right|_{\hat T_i}$ are polynomials, $i = 1,\dots,\ell$.

The following two Lemmas are taken from \cite{AD92}.

\begin{lem}\label{lem:lemma1}
	Let $\boldsymbol{P}$ be a set of multi-indices. To each $v\in H^{\boldsymbol{P}}_p(\Lambda)$ there exists a unique $q \in \boldsymbol{P}(\Lambda)$ with
	\begin{gather*}
		\int_{\Lambda} \boldsymbol{D}^{\boldsymbol{\alpha}} (v-q) \,\mathrm{d} \boldsymbol{X} = 0\qquad \forall \boldsymbol{\alpha} \in \boldsymbol{P}.
	\end{gather*}
\end{lem}
For a short and elegant proof see \cite[Lemma 1]{AD92}. The argument is a slight extension from the well-known  Bramble-Hilbert theory.

\begin{lem}\label{lem:lemma2}
	Let $\boldsymbol{P}$ be a set of multi-indices with $\boldsymbol{0} \in \boldsymbol{P}$. Then there exists a constant $C$ independent of $v$ such that	
	\begin{gather*}
		\|v\|_{\boldsymbol{\overline P},p} \le C |v|_{\boldsymbol{\overline P},p}
	\end{gather*}
	for all $v \in H^{\boldsymbol{\overline P}}(\Lambda)$ with $\int_{\Lambda} \boldsymbol{D}^{\boldsymbol{\alpha}} v \,\mathrm{d}\boldsymbol{X}=0$ for $\boldsymbol{\alpha} \in \boldsymbol{P}$.
\end{lem}
An indirect proof can be found in \cite[Lemma 2]{AD92}. It relies on the compactness of a certain embedding, extending a similar result from Bramble and Hilbert.

The next Lemma is an adaptation of \cite[Lemma 3]{AD92} to our patchwise setting.
\begin{lem}\label{lem:lemma3} Let $\boldsymbol{\gamma}$ be a multi-index, $I:C^\mu(\Lambda) \rightarrow S(\hat M) \subset H^{\boldsymbol{P}+\boldsymbol{\gamma}}_p(\Lambda)$, $\mu \in \N$ be a linear operator and let $\boldsymbol{Q}$ be a set of multi-indices with $\boldsymbol{0} \in \boldsymbol{Q}$ and $\boldsymbol{P} \subset \boldsymbol{\overline Q}$. Assume that there are linear functionals $F_i \in \left( H^{\boldsymbol{\overline Q}}_p(\Lambda)\right)'$, $i=1,\dots,j=\dim \boldsymbol{D}^{\boldsymbol{\gamma}} S(\hat M)$, with the properties
	\begin{gather}\label{eq:functionals}
		\left \{ \begin{aligned}
		F_i(\boldsymbol{D}^{\boldsymbol{\gamma}} I u) &= F_i(\boldsymbol{D}^{\boldsymbol{\gamma}} u),\qquad i=1,\dots,j,\qquad \forall u \in C^\mu(\Lambda) \cap H^{\boldsymbol{\overline Q} + \boldsymbol{\gamma}}_p(\Lambda),\\
		\big(F_i(\boldsymbol{D}^{\boldsymbol{\gamma}} s) &= 0\quad \text{for $i=1,\dots,j$} \big) \quad \Rightarrow \quad \boldsymbol{D}^{\boldsymbol{\gamma}} s = 0\qquad \forall s \in S(\hat M).
		\end{aligned}\right.
	\end{gather}
	Then there exists a constant $C$ independent of $u$ such that
	\begin{gather}\label{eq:interpolestim}
		\|u-I u\|_{\boldsymbol{P} + \boldsymbol{\gamma},p} \le C \left( |u|_{\boldsymbol{\overline Q}+\boldsymbol{\gamma},p} + \|q -  I q\|_{\boldsymbol{P}+\boldsymbol{\gamma},p} \right) \qquad \forall u \in C^\mu(\Lambda) \cap H^{\boldsymbol{\overline Q} + \boldsymbol{\gamma}}_p(\Lambda),
	\end{gather}
	where the polynomial $q \in (\boldsymbol{Q}+\boldsymbol{\gamma})(\Lambda)$ is uniquely determined by
	\begin{gather}\label{eq:polyq}
		\int_{\Lambda} \boldsymbol{D}^{\boldsymbol{\alpha}+\boldsymbol{\gamma}} (u-q) \,\mathrm{d} \boldsymbol{X} = 0\qquad \forall \boldsymbol{\alpha} \in \boldsymbol{Q}.
	\end{gather}
\end{lem}
\begin{proof}
	By Lemma \ref{lem:lemma1} the polynomial $q \in (\boldsymbol{Q}+\boldsymbol{\gamma})(\Lambda)$ satisfying \eqref{eq:polyq} is indeed unique.	
	The triangle inequality gives
	\begin{gather}\label{eq:triangle}
		\|u -  I u\|_{\boldsymbol{P}+\boldsymbol{\gamma},p} \le \|u -  q\|_{\boldsymbol{\overline Q}+\boldsymbol{\gamma},p} + \|q -  I q\|_{\boldsymbol{P}+\boldsymbol{\gamma},p} + \|I(q -  u)\|_{\boldsymbol{P}+\boldsymbol{\gamma},p}.
	\end{gather}
	Based on \eqref{eq:functionals} we observe that $\sum_{i=1}^j |F_i(\cdot)|$ is a norm on $\boldsymbol{D}^{\boldsymbol{\gamma}} S(\hat M)$. Hence, norm equivalence in finite dimensional spaces yields for the last term
	\begin{gather}\label{eq:discrete}
	\begin{aligned}
		\|I(q -  u)\|_{\boldsymbol{P}+\boldsymbol{\gamma},p} &= \|\boldsymbol{D}^{\boldsymbol{\gamma}} I(q-u)\|_{\boldsymbol{P},p} \le C \sum_{i=1}^j \big|F_i\big(\boldsymbol{D}^{\boldsymbol{\gamma}} I(q-u)\big)\big|\\
		&= C \sum_{i=1}^j \big|F_i\big(\boldsymbol{D}^{\boldsymbol{\gamma}} (q-u)\big)\big| \le C \|u-q\|_{\boldsymbol{\overline Q}+\boldsymbol{\gamma},p}.
	\end{aligned}
	\end{gather}
With \eqref{eq:polyq} an application of Lemma \ref{lem:lemma2} gives
	\begin{gather}\label{eq:abschseminorm}
		\|u-q\|_{\boldsymbol{\overline Q}+\boldsymbol{\gamma},p} = \big\| \boldsymbol{D}^{\boldsymbol{\gamma}} (u-q) \big\|_{\boldsymbol{\overline Q},p} \le C |\boldsymbol{D}^{\boldsymbol{\gamma}}u|_{\boldsymbol{\overline Q},p} = C |u|_{\boldsymbol{\overline Q}+\boldsymbol{\gamma},p}.
	\end{gather}
Collecting \eqref{eq:triangle}, \eqref{eq:discrete} and \eqref{eq:abschseminorm} the result follows.
\end{proof}

\begin{rem}
	The estimate \eqref{eq:interpolestim} shows that a macro-element interpolation operator should be designed in such a way that on the macro-element polynomials with a degree as high as possible are reproduced. Ideally, $q=Iq$ for all $q \in (\boldsymbol{Q}+\boldsymbol{\gamma})(\Lambda)$ which leads to the estimate $\|u-Iu\|_{\boldsymbol{P} + \boldsymbol{\gamma}} \le C |u|_{\boldsymbol{\overline Q}+\boldsymbol{\gamma}}$ for all $u \in C^\mu(\Lambda) \cap H^{\boldsymbol{\overline Q} + \boldsymbol{\gamma}}_p(\Lambda)$. Otherwise an additional error component arises due to the inability to reproduce certain polynomials. This is the only difference in comparison with the theory of \cite{AD92} caused by a triangle inequality with $Iq$ in \eqref{eq:triangle}. Such an amendment becomes necessary because in general the polynomial $q \not \in S(\hat M)$ does not lie within the spline space.
\end{rem}

\begin{defn}
	Since the interpolation operator is usually defined by linear functionals we follow the nomenclature of \cite{AD92} and will call the $F_i$ from \eqref{eq:functionals} \emph{associated functionals (with respect to $\boldsymbol{D}^{\boldsymbol{\gamma}}$)}.
\end{defn}


\section[\texorpdfstring{$C^1$ macro-interpolation on anisotropic meshes}{C1 macro-interpolation on anisotropic meshes}]{\texorpdfstring{$C^1$ macro-interpolation on anisotropic tensor product meshes}{C1 macro-interpolation on anisotropic tensor product meshes}}

Before we turn our attention to a rigorous analysis of $\Pi$ from \eqref{eq:interpolant} we want to consider a simpler \emph{reduced} operator. By doing so we demonstrate the developed techniques without getting bogged down in details. Moreover, the insight gained into this reduced interpolation operator will prove to be very useful in the analysis of a quasi-interpolation operator.

\subsection{A reduced macro-element interpolation operator}\label{subsec:reduced}

Let us consider the reference domain $\Lambda\coloneqq[-1,1]^2$ decomposed into the reference macro-element $\hat M \coloneqq\{\hat T_i\}_{i=1,\dots,4}$, where $\hat T_i$ is the intersection of $\Lambda$ with the $i$th quadrant, $i=1,\dots,4$. 
With the basis functions from \eqref{eq:basis2d} we introduce the following \emph{reduced macro-element interpolation operator} $\hat \Pi^r: C^1(\Lambda) \rightarrow S(\hat M)$ with $S(\hat M) \subset \{v \in C^1(\Lambda)\, :\, \left.v\right|_{\hat T_i} \in Q_2,\; i=1,\dots,4\}$,
\begin{gather}\label{eq:redinterpolant}
	(\hat \Pi^r v)(x,y) = \sum_{i,j\in\{-1,1\}} v(i,j) \hat \varphi_{i,j}(x,y) + v_x(i,j) \hat \phi_{i,j}(x,y) + v_y(i,j) \hat \chi_{i,j}(x,y),\quad (x,y) \in \Lambda.
\end{gather}
In comparison to $\Pi$ from \eqref{eq:interpolant} we discard the basis functions associated with the mixed derivative.
Since $\hat \Pi$ maps a sufficiently smooth function into $C^1(\Lambda)$, as was shown in \cite{HHZ2011}, we observe for $v \in C^1(\Lambda)$ that
\begin{gather*}
	\hat \Pi^r v = \hat \Pi \left(\hat \Pi^r v\right) \in C^1(\Lambda).
\end{gather*}
Hence, indeed $S(\hat M) \subset \{v \in C^1(\Lambda)\, :\, \left.v\right|_{\hat T_i} \in Q_2,\; i=1,\dots,4\} \subset W_{2,p}(\Lambda)$. Let us fix $\boldsymbol{\gamma} = (1,0)$. If we seek to apply Lemma \ref{lem:lemma3} to this setting we need to find eight associated functionals $F_i$, $i=1,\dots,8$, since
\begin{gather}\label{eq:basisdx}
\begin{aligned}
	\boldsymbol{D}^{(1,0)} S(\hat M) = \spn \{&\hat\varphi_{-1}'(x)\hat\varphi_{-1}(y),\,\hat\varphi_{-1}'(x)\hat\varphi_{1}(y),\,\hat\psi_{-1}'(x)\hat\varphi_{-1}(y),\,\hat\psi_{1}'(x)\hat\varphi_{-1}(y),\\
&\hat\psi_{-1}'(x)\hat\varphi_{1}(y),\,\hat\psi_{1}'(x)\hat\varphi_{1}(y),\,\hat\varphi_{-1}'(x)\hat\psi_{-1}(y),\,\hat\varphi_{-1}'(x)\hat\psi_{1}(y)\}
\end{aligned}
\end{gather}
is an eight-dimensional space.
Setting $\boldsymbol{P} = \boldsymbol{Q} \coloneqq \{(0,0),(0,1),(1,0)\}$ these functionals must be members of $\left(W_{2,p}(\Lambda)\right)'$.
For $i=1,\dots,4$ let $V_i$ denote the four vertices of $\Lambda$. Then for $v \in W_{2,p}(\Lambda)$ we find
\begin{gather*}
	F_i(v) \coloneqq v(V_i),\quad i=1,\dots,4
\end{gather*}
with $|F_i(v)|\le C \|v\|_{W_{2,p}(\Lambda)}$, $i=1,\dots,4$ due to the well known Sobolev embedding $W_{2,p}(\Lambda) \hookrightarrow C(\Lambda)$ in two dimensions. Moreover, for $u \in H^{\boldsymbol{\overline Q} + \boldsymbol{\gamma}}_p(\Lambda)$, i.e.~$u_x\in W_{2,p}(\Lambda)$ one has
\begin{gather*}
	F_i(\boldsymbol{D}^{\boldsymbol{\gamma}}u) = u_x(V_i) = F_i(\boldsymbol{D}^{\boldsymbol{\gamma}}\hat \Pi^r u),\quad i=1,\dots,4.
\end{gather*}
The other four associated functionals are defined on the edges $E_1 \coloneqq \{(x,-1)\,:\, |x|\le 1\}$ and $E_2 \coloneqq \{(x,1)\,:\, |x|\le 1\}$ of $\Lambda$ which are parallel to the $x$-axis. In fact, they are the mean value and the mean value of the normal derivative:
\begin{gather*}
	F_{4+i}(v) \coloneqq \frac{1}{2} \int_{E_i} v(s) \D s,\qquad
	F_{6+i}(v) \coloneqq \frac{1}{2} \int_{E_i} v_y(s) \D s,\quad i=1,2.
\end{gather*}
By well known trace theorems $|F_i(v)|\le C \|v\|_{W_{2,p}(\Lambda)}$, $i=5,\dots,8$ (see e.g.~\cite{A99} and the references cited in Section 1.3) and
\begin{gather*}
	F_{5}(\boldsymbol{D}^{\boldsymbol{\gamma}}u) = \frac{1}{2} \int_{E_1} u_x(s) \D s = \frac{1}{2}\big(u(V_2) - u(V_1)\big) = \frac{1}{2} \int_{E_1} \boldsymbol{D}^{\boldsymbol{\gamma}} \left(\hat \Pi^r u\right) (s) \D s = F_{5}(\boldsymbol{D}^{\boldsymbol{\gamma}} \hat\Pi^r u),\\
	F_{7}(\boldsymbol{D}^{\boldsymbol{\gamma}}u) = \frac{1}{2} \int_{E_1} u_{xy}(s) \D s = \frac{1}{2} \big(u_y(V_2) - u_y(V_1)\big) =  \frac{1}{2} \int_{E_1} \boldsymbol{D}^{\boldsymbol{\gamma}} \left(\hat \Pi^r u\right)_y(s) \D s = F_{7}(\boldsymbol{D}^{\boldsymbol{\gamma}} \hat \Pi^r u).
\end{gather*}
Similarly, these identities can be shown to hold true for $F_6$ and $F_8$.

Next we show that the functionals $F_i$ define a norm in $\boldsymbol{D}^{\boldsymbol{\gamma}} S(\hat M)$. For this purpose let $u \in \boldsymbol{D}^{\boldsymbol{\gamma}} S(\hat M)$ with $F_i(u)=0$ for $i=1,\dots,8$. Based on the relations
\begin{gather*}
	\hat\psi_i'(k)\hat\varphi_j(\ell) = \delta_{ik} \delta_{j\ell},\qquad i,j,k,\ell\in\{-1,1\}
\end{gather*}
and $v(\pm 1,\pm 1)=0$ for all other basis functions $v$ of $\boldsymbol{D}^{\boldsymbol{\gamma}} S(\hat M)$ in \eqref{eq:basisdx} we find that
\begin{gather*}
	u \in \spn \{ \hat\varphi_{-1}'(x) \hat\varphi_{-1}(y),\, \hat\varphi_{-1}'(x) \hat\varphi_{1}(y),\, \hat\varphi_{-1}'(x) \hat\psi_{-1}(y),\, \hat\varphi_{-1}'(x) \hat\psi_{1}(y) \}.
\end{gather*}
Out of these remaining four basis functions only $\hat\varphi_{-1}'(x) \hat\varphi_{-1}(y)$ is non-trivial on the edge $E_1$. Similarly, only $\hat\varphi_{-1}'(x) \hat\varphi_{1}(y)$ has values different from zero on $E_2$. Moreover, these values are all not positive. Since the mean values $F_5(u) = F_6(u) = 0$ of $u$ vanishes on these edges we conclude that
\begin{gather*}
	u \in \spn \{ \hat\varphi_{-1}'(x) \hat\psi_{-1}(y), \hat\varphi_{-1}'(x) \hat\psi_{1}(y) \}.
\end{gather*}
The remaining two basis functions are treated in the same way: while $\hat\varphi_{-1}'(x) \hat\psi_{-1}(y)$ has a non-trivial and non-positive normal derivative on the Edge $E_1$ we find $\hat\varphi_{-1}'(x) \hat\psi_{1}'(y) \equiv 0$ on $E_1$.
On the edge $E_2$ the relations are exactly the other way round. Hence, $u \equiv 0$.

An application of Lemma \ref{lem:lemma3} yields
\begin{gather}\label{eq:errred1}
\|(u-\hat \Pi^r u)_x\|_{W_{1,p}(\Lambda)} \le C \left( |u_x|_{W_{2,p}(\Lambda)} + \|(q -  \hat \Pi^r q)_x\|_{W_{1,p}(\Lambda)} \right),
\end{gather}
for all $u \in C^1(\Lambda) \cap H^{\overline Q + \boldsymbol{\gamma}}_p(\Lambda)$. The latter means that $u_x \in W_{2,p}(\Lambda)$. The polynomial $q$ is determined by \eqref{eq:polyq} and we want to estimate the second error component of \eqref{eq:errred1} containing it. Obviously, $q \in \big(\boldsymbol{Q}+\boldsymbol{\gamma}\big)(\Lambda)$ has a representation of the form
\begin{gather*}
	q(x,y) = q_1 x + q_2 x y + q_3 x^2\qquad (x,y)\in \Lambda.
\end{gather*}
Here the coefficients $q_i \in \R$, $i=1,2,3$ are determined by $u$. A direct calculation shows that the function $(x,y) \mapsto x$ is invariant under interpolation:
\begin{align}\label{eq:invarprop}
	\hat \Pi^r x &= \big(\hat \varphi_{1,-1}(x,y) + \hat \varphi_{1,1}(x,y) \big) - \big(\hat \varphi_{-1,-1}(x,y) + \hat \varphi_{-1,1}(x,y) \big) + \sum_{i,j\in\{-1,1\}} \hat \phi_{i,j}(x,y) \\
	&= \big(\hat\varphi_{1}(x) - \hat\varphi_{-1}(x) + \hat\psi_{-1}(x) + \hat\psi_{1}(x)\big) \big(\hat\varphi_{-1}(y) + \hat\varphi_{1}(y) \big) = x.
\end{align}
Similarly, the function $(x,y) \mapsto x^2$ is preserved by the interpolation operator on the macro-element, i.e. $\hat \Pi^r(x^2) = x^2$. From \eqref{eq:polyq} with $\boldsymbol{\alpha} = (0,1)$ we determine $q_2 = \frac{1}{4} \int_{\Lambda} u_{xy}(x,y) \,\D x\D y$, hence  
\begin{gather}\label{eq:errred2}
	\|(q -  \hat\Pi^r q)_x\|_{W_{1,p}(\Lambda)} = |q_2|\, \|( x y -  \hat\Pi^r x y)_x\|_{W_{1,p}(\Lambda)} \le C \left| \int_{\Lambda} u_{xy}(x,y) \,\D x \D y \right|.
\end{gather}
Collecting \eqref{eq:errred1} and \eqref{eq:errred2} we arrive at
\begin{gather}\label{eq:respirref}
	\|(u-\hat\Pi^r u)_x\|_{W_{1,p}(\Lambda)} \le C \left( | u_x|_{W_{2,p}(\Lambda)} + \left| \int_{\Lambda} u_{xy}(x,y) \,\D x \D y \right| \right),
\end{gather}
for all $\forall u \in C^1(\Lambda) \cap H^{\boldsymbol{\overline Q} + \boldsymbol{\gamma}}_p(\Lambda)$.

\begin{rem}\label{rem:red11}
	For $\boldsymbol{\gamma} = (1,1)$ one can choose
	\begin{gather*}
		F_i(v) \coloneqq \int_{E_i} v \nach{s}\quad i=1,\dots,4\qquad \text{and} \qquad F_5(v) = \int_{\Lambda} v \nach{s}
	\end{gather*}
	as associated functionals. Here $E_i$ denotes the $i$th edge of $\Lambda$, $i=1,\dots,4$. In fact, it is easy to show that $\sum_{i=1}^5 |F_i(\cdot)|$ is a norm on
	\begin{gather*}
		\boldsymbol{D}^{(1,1)} S(\hat M) = \spn\{\hat\varphi_{-1}'(x)\hat\varphi_{-1}'(y),\,\!\hat\psi_{-1}'(x)\hat\varphi_{-1}'(y),\,\!\hat\psi_{1}'(x)\hat\varphi_{-1}'(y),\,\!\hat\varphi_{-1}'(x)\hat\psi_{-1}'(y),\,\!\hat\varphi_{-1}'(x)\hat\psi_{1}'(y)\}
	\end{gather*}
	and that $F_i(\boldsymbol{D}^{(1,1)} \hat\Pi^r u) = F_i(\boldsymbol{D}^{(1,1)} u)$ for $i=1,\dots,5$. Moreover, $F_i \in \big(W_{1,p}(\Lambda)\big)'$:
	\begin{align*}
		\big|F_i(v)\big| &=\left|\int_{E_i} v \nach{s} \right| \le \|v\|_{L_1(E_i)} \le C \|v\|_{W_{1,1}(\Lambda)} \le C \|v\|_{W_{1,p}(\Lambda)} \quad i=1,\dots,4,\\
		\big|F_5(v)\big| &= \left|\int_{\Lambda} v \,\D x\D y \right| \le C \|v\|_{L_p(\Lambda)} \le C \|v\|_{W_{1,p}(\Lambda)},
	\end{align*}
	based on Sobolev embeddings and H\"older's inequality. Hence,
	\begin{gather}\label{eq:dxypirref}
	\|(u-\hat\Pi^r u)_{xy}\|_{L_p(\Lambda)} \le C \left( |u_{xy}|_{W_{1,p}(\Lambda)} + \left| \int_{\Lambda} u_{xy}(x,y) \,\D x \D y \right| \right)
	\end{gather}
\end{rem}

Now, let $\mathcal{M}$ be a tensor product mesh of $\Omega$. We shall refer to $\mathcal{M}$ as the macro-element mesh and do not require it to be quasi-uniform, i.e.~there are no restrictions on the element sizes of the underlaying 1D-triangulations $\mathcal{M}_x$ and $\mathcal{M}_y$ of the macro-element mesh.
We obtain the element mesh $\mathcal{T}$ as the tensor product mesh of the two 1D-triangulations that are generated by subdividing every element of $\mathcal{M}_x$ and $\mathcal{M}_y$ uniformly into two elements of equal size. 
The choice of the midpoint as transition point of a macro-element $M$ is not significant. The theory can handle any subdivision such that the elements within one macro are comparable in size. However, it simplifies the presentation.
See Figure \ref{fig:macro-element} for a graphical representation of $\mathcal{M}$ and $\mathcal{T}$.

\begin{figure}
	\centering
	\includegraphics[width=.6\textwidth]{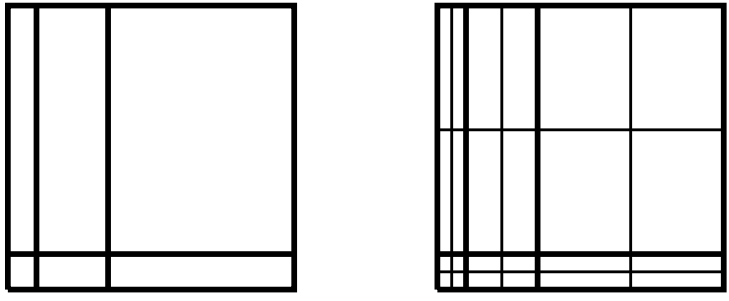}
	\caption{macro-element mesh $\mathcal{M}$ (left) and element mesh $\mathcal{T}$ (right).}
	\label{fig:macro-element}
\end{figure}

Let $M \in \mathcal{M}$ be the macro-element $M = [x_0-h_1,x_0+h_1] \times [y_0-h_2,y_0+h_2]$. Note that $M$ consists out of the four elements of $\mathcal{T}$ that share the vertex $(x_0,y_0)$. Introducing the reference mapping $F_M$ from $[-1,1]^2$ to $M$ by
\begin{gather}\label{eq:mapping}
		x = x_0 + h_1 \hat x,\qquad y = y_0 + h_2 \hat y,
\end{gather}
we obtain anisotropic error estimates for the macro-interpolation operator $\Pi^r u \coloneqq \hat \Pi^r \hat u \circ F_M^{-1}$ with $\hat u \coloneqq u \circ F_M$ on $M$.
\begin{thm}\label{thm:Pired}
	Associated with the shape of the macro-element $M$ let $\boldsymbol{h}\coloneqq(h_1,h_2)$. For $u \in C^1(M)$ with $u_x \in W_{2,p}(M)$ we have the estimate
	\begin{gather}\label{eq:estimred}
		\|\left( u - \Pi^r u \right)_x\|_{L_p(M)} \le C \left( \sum_{|\boldsymbol{\alpha} |=2} \boldsymbol{h}^{\boldsymbol{\alpha}} \|\boldsymbol{D}^{\boldsymbol{\alpha}} u_x\|_{L_p(M)} + h_2 \left| \int_{M} u_{xy}(x,y) \,\D x \D y \right| \right).
	\end{gather}
\end{thm}
\begin{proof}
	The proof uses change of variables, the result \eqref{eq:respirref} on the reference macro-element and the relation $\boldsymbol{D}^{\boldsymbol{\alpha}} = \boldsymbol{h}^{-\boldsymbol{\alpha}} \boldsymbol{\hat D}^{\boldsymbol{\alpha}}$:
	\begin{gather*}
		\|\boldsymbol{D}^{(1,0)}\left( u - \Pi^r u \right)\|_{L_p(M)}^p = h_1^{-p} \left \|\boldsymbol{\hat D}^{(1,0)}\left( \hat u - \hat \Pi^r \hat u \right) \right \|_{L_p(\Lambda)}^p h_1 h_2\\
		\le C h_1^{-p} \left( \left| \boldsymbol{\hat D}^{(1,0)} \hat u \right|_{W_{2,p}(\Lambda)}^p + \left| \int_{\Lambda} \hat u_{xy}(\hat x,\hat y) \,\D \hat x \D \hat y \right|^p \right) h_1 h_2\\
		\le C h_1^{-p} \left( \sum_{|\boldsymbol{\alpha} |=2} h_1^p \boldsymbol{h}^{p \boldsymbol{\alpha}} \left\| \boldsymbol{D}^{\boldsymbol{\alpha}} \boldsymbol{D}^{(1,0)} u \right\|_{L_p(M)}^p + h_1^p h_2^p \left| \int_{M} \boldsymbol{D}^{(1,1)} u(x,y) \,\D x \D y \right|^p \right).
	\end{gather*}
	Which is the desired estimate. In the case $p=\infty$ some minor modifications are needed.
\end{proof}

\begin{rem}
	The diagonal form of the affine reference mapping $F_M$ according to \eqref{eq:mapping} is needed for affine equivalence of the interpolation operator. Note that only in this case $Q_k$ elements are affine equivalent.
\end{rem} 
 
\begin{rem}
	While for functions $u \in C^1(M)$ with $u_x \in W_{2,p}(M)$ the reduced interpolation operator $\Pi^r$ is not of second order in the $W_{1,p}$ semi-norm it is of optimal second order if additionally the mean value of the mixed derivative $u_{xy}$ vanishes on $M$. Clearly, this reduction in approximation ability corresponds to discarding the basis functions $\hat\psi_{\pm 1,\pm 1}$ in \eqref{eq:redinterpolant}.
\end{rem} 

\begin{rem}
Similarly, one can can deduce from the result in Remark \ref{rem:red11} that
\begin{gather*}
	\|(u-\Pi^r u)_{xy}\|_{L_p(M)} \le C \left( \sum_{|\boldsymbol{\alpha} |=1} \boldsymbol{h}^{\boldsymbol{\alpha}} \|\boldsymbol{D}^{\boldsymbol{\alpha}} u_{xy}\|_{L_p(M)} + \left| \int_{M} u_{xy}(x,y) \,\D x \D y \right| \right).
\end{gather*}
Clearly, this result is in general unsatisfactory. The inability to yield anisotropic interpolation error estimates for second order derivatives of the approximation error is caused by discarding the basis functions corresponding to the mixed derivative.
\end{rem}

\subsection{\texorpdfstring{The full $C^1-Q_2$ interpolation operator}{The full C1-Q2 interpolation operator}}\label{subsec:c1q2}

As a second example we want to consider the interpolation operator $\hat \Pi$ of \eqref{eq:interpolant}. We refer to it as \emph{full} not only to contrast it from the reduced operator in the previous subsection but also to underline the property \eqref{eq:full} of its underlaying macro-element space. Since this operator is closely related to interpolation on the bicubic $C^1$ Bogner-Fox-Schmidt element, we shall first give a result from the literature. To the knowledge of the author there exists only one paper dealing with anisotropic interpolation error estimates for this element. In \cite{CYM2009} the authors derive the result
\begin{gather}\label{eq:resBFS}
	\left \|\boldsymbol{\hat D}^{\boldsymbol{\gamma}} \left( \hat u - \hat I_{12} \hat u\right) \right \|_{0,\hat K} \le C \left| \boldsymbol{\hat D}^{\boldsymbol{\gamma}} \hat u\right|_{4-|\boldsymbol{\gamma}|,\hat K},
\end{gather}
for $|\boldsymbol{\gamma}|\le 2$ and $\hat u \in H^4(\hat K)$ on the reference element $\hat K \coloneqq [0,1]^2$. Here $I_{12}$ is the analogue of $\hat\Pi$ in the space of bicubic polynomials, i.e.~the Lagrangian basis functions in \eqref{eq:interpolant} have to be replaced by bicubic polynomials satisfying the same (duality and Kronecker) relations. Using affine transformation this result can be extended to
\begin{gather}\label{eq:estimBFS}
	\left \| \boldsymbol{D}^{\boldsymbol{\gamma}} \left(  u - I_{12} u\right) \right \|_{0, K} \le C \sum_{|\boldsymbol{\alpha}|=4-|\boldsymbol{\gamma}|} \boldsymbol{h}^{\boldsymbol{\alpha}} \left\| \boldsymbol{D}^{\boldsymbol{\alpha}} \boldsymbol{D}^{\boldsymbol{\gamma}} u \right\|_{0,K},
\end{gather}
for $|\boldsymbol{\gamma}|\le 2$ and $u \in H^4(K)$ on a rectangular element $K$ with sides aligned to the coordinate axes and $\boldsymbol{h}=(h_1,h_2)$ with edge lengths $h_i$, $i=1,2$. 

However, in \cite{CYM2009} the theory of Apel \cite{AD92,A99} is not used to obtain this result. Instead a new interpolation operator $\hat L_1$ is introduced such that $\hat L_1(\boldsymbol{\hat D}^{\boldsymbol{\gamma}} \hat u) = \boldsymbol{\hat D}^{\boldsymbol{\gamma}} \hat I_{12} \hat u$ and standard interpolation theory is applied to obtain a bound for the interpolation error of $\hat L_1$. Since we are dealing with only piecewise polynomials this path is blocked for us. A spin-off of our discussion will be how the results \eqref{eq:resBFS} and \eqref{eq:estimBFS} can be obtained using Apel's theory. The key is to recognize that divided differences can be used as associated functionals. Since we need certain Sobolev embeddings, we focus on the case $p=2$, which also appears to be the most important one with respect to applications.

Inspired by \cite{CYM2009} we generalize the Newtonian representation \eqref{eq:newtonSpline} to two dimensions obtaining
\begin{gather}\label{eq:newton2D}
	\begin{gathered}
	\big(\hat \Pi u\big)(x,y) = \sum_{i=1}^3 \sum_{j=1}^3 F_{i,j}(u) (x+1)^{i-1} (y+1)^{j-1} + 4 \sum_{j=1}^3 F_{4,j}(u) \hat\psi_1(x) (y+1)^{j-1}\\
	+ 4 \sum_{i=1}^3 F_{i,4}(u) (x+1)^{i-1} \hat\psi_1(y) + 16 F_{4,4}(u) \hat\psi_1(x) \hat\psi_1(y).
	\end{gathered}
\end{gather}
Here the 16 functionals $F_{i,j}$, $i,j=1,\dots,4$ are two-dimensional divided differences with multiple knots, see e.g.~\cite{PB2009}. If we define a sorted node sequence by 
\begin{gather*}
	n_i = \left \{
	\begin{alignedat}{2}
	&{-1}&\quad &\text{for $i=1$,}\\
	&{-1},{-1}&\quad &\text{for $i=2$,}\\
	&{-1},{-1},1&\quad &\text{for $i=3$,}\\
	&{-1},{-1},1,1&\quad &\text{for $i=4$,}\\
	\end{alignedat}\right.
\end{gather*}
then $F_{i,j}(u) \coloneqq u[n_i;n_j]$ is the divided difference of order $i-1$ to $x$ and order $j-1$ to $y$:
\begin{defn}
	For a fixed $y \in [-1,1]$ let 
	\begin{gather*}
		u_{n_i}(y) \coloneqq u(\cdot,y)[n_i]
	\end{gather*}
	denote the parametrized one dimensional divided difference (with respect to $x$ and the node sequence $n_i$). Then the two dimensional divided difference $u[n_i;n_j]$ is defined by
	\begin{gather*}
		u[n_i;n_j] \coloneqq u_{n_i}[n_j].
	\end{gather*}
\end{defn}

\begin{rem}
	Because of $u[n_i;n_j] = \big(u(x,\cdot)[n_j]\big)[n_i]$ one can start with the evaluation in $y$, as well.
\end{rem}

We find that
\begin{gather}\label{eq:valderiv}
	F_{i,j}(u) = u[\overbrace{{-1},\dots,{-1}}^{i \text{ times}} ; \overbrace{{-1},\dots,{-1}}^{j \text{ times}}] =  \boldsymbol{D}^{(i-1,j-1)} u(-1,-1), \qquad i,j=1,2.
\end{gather}
Moreover, using \eqref{eq:dividiff} for instance
\begin{align*}
	F_{3,1}&=u[{-1},{-1},1;{-1}] = \frac{1}{4} u(1,-1) - \frac{1}{4} u(-1,-1) - \frac{1}{2} u_x(-1,-1),\\
	F_{4,1}&= u[{-1},{-1},1,1;{-1}]=\frac{1}{4} u(-1,-1)-\frac{1}{4}u(1,-1)+\frac{1}{4} u_x(-1,-1)+\frac{1}{4} u_x(1,-1),\\
	F_{3,2}&=u[{-1},{-1},1;{-1},{-1}] = \frac{1}{4}u_y(1,-1)-\frac{1}{4} u_y(-1,-1)-\frac{1}{2} u_{xy}(-1,-1),\\
	F_{4,2}&= u[{-1},{-1},1,1;{-1},{-1}] = \frac{1}{4} u_y(-1,-1)-\frac{1}{4}u_y(1,-1)+\frac{1}{4} u_{xy}(-1,-1)+\frac{1}{4} u_{xy}(1,-1).
\end{align*}
Similarly, the other divided differences can be calculated, e.g.
\begin{align*}
	F_{3,3} = u[{-1},{-1},1;{-1},{-1},1] &= \frac{1}{4}\left( \frac{1}{4} u(1,1) - \frac{1}{4} u(-1,1) - \frac{1}{2} u_x(-1,1)\right)\\ 
	&\quad- \frac{1}{4}\left( \frac{1}{4} u(1,-1) - \frac{1}{4} u(-1,-1) - \frac{1}{2} u_x(-1,-1)\right)\\
	&\quad- \frac{1}{2}\left( \frac{1}{4} u_y(1,-1) - \frac{1}{4} u_y(-1,-1) - \frac{1}{2} u_{xy}(-1,-1)\right).
\end{align*}
Obviously, all divided differences $F_{i,j}$, $i,j=1,\dots,4$ can be expressed as linear combinations of the interpolation data $\{u(\pm 1,\pm 1),u_x(\pm 1,\pm 1),u_y(\pm 1,\pm 1),u_{xy}(\pm 1,\pm 1)\}$.

In contrast to Subsection \ref{subsec:reduced} we want to consider an arbitrary multi-index $\boldsymbol{\gamma}$ with $|\boldsymbol{\gamma}|\le 2$ here. Consequently, certain sets and functionals depend on the specific choice of $\boldsymbol{\gamma}$ and we emphasize this by using the additional subscript or superscript $\boldsymbol{\gamma}$. By applying the differential operator $\boldsymbol{D}^{\boldsymbol{\gamma}}$ to the representation \eqref{eq:newton2D} we observe that the space $\boldsymbol{D}^{\boldsymbol{\gamma}} S(\hat M)$ can be normed by $\sum_{(i,j) \in J_{\boldsymbol{\gamma}}} |F_{i,j}(\cdot)|$. Here
\begin{gather*}
	J_{(\gamma_1,\gamma_2)} = \{(i,j)\, : \, i=\gamma_1+1,\dots,4 \text{ and } j=\gamma_2+1,\dots,4\}.
\end{gather*}
Note also that by construction $\dim \boldsymbol{D}^{\boldsymbol{\gamma}} S(\hat M) = |J_{\boldsymbol{\gamma}}|$. We want to apply Lemma \ref{lem:lemma3} with $\boldsymbol{Q} \coloneqq \{\boldsymbol{\alpha} \, : \, |\boldsymbol{\alpha}| \le 3 - |\boldsymbol{\gamma}|\}$. In order to establish that $F_{i,j}$ with $(i,j) \in J_{\boldsymbol{\gamma}}$ are associated functionals according to \eqref{eq:functionals} we have to show that the divided differences $F_{i,j}$ can be interpreted as the application of a linear functional $F_{i,j}^{\boldsymbol{\gamma}} \in \left(H^{4-|\boldsymbol{\gamma}|}(\Lambda)\right)'$ on the derivative $\boldsymbol{D}^{\boldsymbol{\gamma}} u$ such that 
\begin{gather*}
	F_{i,j}^{\boldsymbol{\gamma}} (\boldsymbol{D}^{\boldsymbol{\gamma}} \hat\Pi u) = F_{i,j}^{\boldsymbol{\gamma}} (\boldsymbol{D}^{\boldsymbol{\gamma}} u),\qquad (i,j)\in J_{\boldsymbol{\gamma}},\qquad \forall u \in \{ v \in C^2(\Lambda)\, :\, \boldsymbol{D}^{\boldsymbol{\gamma}} v \in H^{4-|\boldsymbol{\gamma}|}(\Lambda) \}.
\end{gather*}

Following \cite{CYM2009} we reinterpret \eqref{eq:valderiv} in the form 
\begin{gather*}
	F_{i,j}(u) = \boldsymbol{D}^{(i-1-\gamma_1,j-1-\gamma_2)} \boldsymbol{D}^{\boldsymbol{\gamma}} u(-1,-1) =: F_{i,j}^{\boldsymbol{\gamma}} (\boldsymbol{D}^{\boldsymbol{\gamma}} u), \qquad (i,j)\in J_{\boldsymbol{\gamma}},\, i,j=1,2.
\end{gather*}
Since all these $F_{i,j}$ can be expressed as linear combinations of the interpolation data we have
\begin{gather}\label{eq:interpolinvar}
	F_{i,j}^{\boldsymbol{\gamma}} (\boldsymbol{D}^{\boldsymbol{\gamma}} u) = F_{i,j}(u) = F_{i,j}(\hat\Pi u) = F_{i,j}^{\boldsymbol{\gamma}} (\boldsymbol{D}^{\boldsymbol{\gamma}} \hat\Pi u)
\end{gather} 
for $(i,j)\in J_{\boldsymbol{\gamma}}$, $i,j=1,2$.
Moreover, from a standard Sobolev embedding $H^{4-|\boldsymbol{\gamma}|} \hookrightarrow C^{2-|\boldsymbol{\gamma}|}$ for $v \in H^{4-|\boldsymbol{\gamma}|}(\Lambda)$
\begin{gather*}
	\left| F_{i,j}^{\boldsymbol{\gamma}} (v) \right| \le \left | \boldsymbol{D}^{(i-1-\gamma_1,j-1-\gamma_2)} v(-1,-1) \right| \le C \|v\|_{4-|\boldsymbol{\gamma}|}, \qquad (i,j)\in J_{\boldsymbol{\gamma}},\, i,j=1,2.
\end{gather*}
 
For the other divided differences we need some kind of Peano form which was developed in \cite{CYM2009}. In fact, replacing $u(1)$ and $u'(1)$ in \eqref{eq:dividiff} by the Taylor expansions
\begin{gather*}
	u(1) = u(-1) + 2 u'(-1) + \int_{-1}^1 (1-x) u''(x) \,\D x\qquad\text{and}\qquad u'(1) = u'(-1) + \int_{-1}^1 u''(x) \,\D x
\end{gather*}
one obtains
\begin{gather}\label{eq:peanoform}
	u[{-1},{-1},\overbrace{1\dots 1}^{i \text{ times}}] = \int_{-1}^1 s_i(x) u''(x) \D x\qquad \text{with}\qquad s_i(x)=\left \{
	\begin{alignedat}{2}
		& (1-x)/4,&\quad &\text{for $i=1$,}\\
		& x/4,&\quad &\text{for $i=2$.}\\
	\end{alignedat}\right.
\end{gather}
With respect to \eqref{eq:interpolinvar} it is important to note that this identity does not only hold for $C^2([-1,1])$ functions but also for the quadratic $C^1$ splines considered as can be checked by examining all the basis functions $\hat\varphi_{\pm 1}$ and $\hat\psi_{\pm 1}$, for instance
\begin{gather*}
	\hat\varphi_{-1}[-1,-1,1] = -\frac{1}{4} = \int_{-1}^1 \frac{(1-x)}{4} \hat\varphi_{-1}''(x) \D x.
\end{gather*}
Clearly, $\left| s_i(x) \right| \le \frac{1}{2}$ for $x\in[-1,1]$ and $i=1,2$. Hence, one can conclude that
\begin{align*}
	F_{i,j}(u) &= u[n_{i};n_j] = \int_{-1}^1 s_{i-2}(x) \boldsymbol{D}^{(2,j-1)} u(x,-1) \,\D x \\
	&= \int_{-1}^1 s_{i-2}(x) \boldsymbol{D}^{(2-\gamma_1,j-1-\gamma_2)} \boldsymbol{D}^{\boldsymbol{\gamma}} u(x,-1) \,\D x =: F_{i,j}^{\boldsymbol{\gamma}}(\boldsymbol{D}^{\boldsymbol{\gamma}} u),
\end{align*}
for $(i,j)\in J_{\boldsymbol{\gamma}},\, i=3,4,\,j=1,2$ and with
\begin{gather*}
	\left| F_{i,j}^{\boldsymbol{\gamma}}(v)\right| \le \frac{1}{2} \int_{-1}^1 \left| \boldsymbol{D}^{(2-\gamma_1,j-1-\gamma_2)} v(x,-1) \right| \,\D x \le C \|v\|_{4-|\boldsymbol{\gamma}|}, \qquad (i,j)\in J_{\boldsymbol{\gamma}},\, i=3,4,\,j=1,2.
\end{gather*}
for $v\in H^{4-|\boldsymbol{\gamma}|}(\Lambda)$ by a trace theorem (c.p.~\cite{CYM2009}). Note that the identities in \eqref{eq:interpolinvar} hold true for $(i,j)\in J_{\boldsymbol{\gamma}},\, i=3,4,\,j=1,2$, as well.

In exactly the same manner we treat the functionals $F_{i,j}$ with $(i,j)\in J_{\boldsymbol{\gamma}},\, i=1,2,\,j=3,4$:
\begin{align*}
	F_{i,j}(u) &= u[n_{i};n_j] = \int_{-1}^1 s_{j-2}(y) \boldsymbol{D}^{(i-1,2)} u(-1,y) \,\D x \\
	&= \int_{-1}^1 s_{j-2}(y) \boldsymbol{D}^{(i-1-\gamma_1,2-\gamma_2)} \boldsymbol{D}^{\boldsymbol{\gamma}} u(-1,y) \,\D y =: F_{i,j}^{\boldsymbol{\gamma}}(\boldsymbol{D}^{\boldsymbol{\gamma}} u).
\end{align*}
Using the same argument as before it is easy to obtain $\left| F_{i,j}^{\boldsymbol{\gamma}}(v)\right| \le \|v\|_{4-|\boldsymbol{\gamma}|}$ and \eqref{eq:interpolinvar} for
$(i,j)\in J_{\boldsymbol{\gamma}},\, i=1,2,\,j=3,4$.

Finally, we consider $F_{i,j}$ for $i,j=3,4$.
\begin{align*}
	F_{i,j}(u) &= u[n_{i};n_j] = \int_{\Lambda} s_{i-2}(x) s_{j-2}(y) \boldsymbol{D}^{(2,2)} u(x,y) \,\D x\D y \\
	&= \int_{\Lambda} s_{i-2}(x) s_{j-2}(y) \boldsymbol{D}^{(2-\gamma_1,2-\gamma_2)} \boldsymbol{D}^{\boldsymbol{\gamma}} u(x,y) \,\D x\D y =: F_{i,j}^{\boldsymbol{\gamma}}(\boldsymbol{D}^{\boldsymbol{\gamma}} u), \qquad i,j=3,4.
\end{align*}
Again the functionals can be shown to be bounded. Using the Cauchy Schwarz inequality
\begin{align*}
	\left| F_{i,j}^{\boldsymbol{\gamma}} (v) \right| \le \frac{1}{4} \int_{\Lambda} \left| \boldsymbol{D}^{(2-\gamma_1,2-\gamma_2)} v(x,y) \right| \,\D x\D y \le C \left\| \boldsymbol{D}^{(2-\gamma_1,2-\gamma_2)} v \right\|_0 \le C \|v\|_{4-|\boldsymbol{\gamma}|}
\end{align*}
for $i,j=3,4$. Additionally, \eqref{eq:interpolinvar} holds true for $i,j=3,4$.

Hence, for a given differential operator $\boldsymbol{D}^{\boldsymbol{\gamma}}$ with $|\boldsymbol{\gamma}| \le 2$ we can use $F_{i,j}^{\boldsymbol{\gamma}}$, $(i,j)\in J_{\boldsymbol{\gamma}}$ as associated functionals and Lemma \ref{lem:lemma3} yields for $u\in C^2(\Lambda)$ with $\boldsymbol{D}^{\boldsymbol{\gamma}} u \in H^{4-|\boldsymbol{\gamma}|}(\Lambda)$ that
\begin{gather}\label{eq:estimPi}
	\left\| \boldsymbol{D}^{\boldsymbol{\gamma}} \left( u- \hat\Pi u \right) \right\|_{2-|\boldsymbol{\gamma}|} \le C \left( \left| \boldsymbol{D}^{\boldsymbol{\gamma}} u \right|_{4-|\boldsymbol{\gamma}|} + \left\| \boldsymbol{D}^{\boldsymbol{\gamma}} \left( q - \hat\Pi q \right) \right\|_{2-|\boldsymbol{\gamma|}} \right).
\end{gather}
Here the polynomial $q \in (\boldsymbol{Q}+\boldsymbol{\gamma})(\Lambda)$ is defined by \eqref{eq:polyq} where $\boldsymbol{Q}\coloneqq \{\boldsymbol{\alpha} \, : \, |\boldsymbol{\alpha}| \le 3 - |\boldsymbol{\gamma}|\}$. Hence, with $\boldsymbol{X}=(x,y)$,
\begin{gather*}
	q(x,y) = \sum_{\boldsymbol{\alpha} \in \boldsymbol{Q}} q_{\boldsymbol{\alpha} +\boldsymbol{\gamma}} \boldsymbol{X}^{\boldsymbol{\alpha} +\boldsymbol{\gamma}} = \sum_{\substack{\boldsymbol{\alpha} \in \boldsymbol{Q} \\ |\boldsymbol{\alpha} + \boldsymbol{\gamma}| < 3}} q_{\boldsymbol{\alpha}+\boldsymbol{\gamma}} \boldsymbol{X}^{\boldsymbol{\alpha}+\boldsymbol{\gamma}} + \sum_{\substack{\boldsymbol{\alpha} \in \boldsymbol{Q} \\ |\boldsymbol{\alpha} + \boldsymbol{\gamma}| = 3}} q_{\boldsymbol{\alpha}+\boldsymbol{\gamma}} \boldsymbol{X}^{\boldsymbol{\alpha}+\boldsymbol{\gamma}},
\end{gather*}
with coefficients $q_{\boldsymbol{\alpha}+\boldsymbol{\gamma}} = q_{\boldsymbol{\alpha}+\boldsymbol{\gamma}}(u) \in \R$ for $\boldsymbol{\alpha} \in \boldsymbol{Q}$.
A simple calculation shows that $\hat\Pi(x^{\boldsymbol{\beta}}) = x^{\boldsymbol{\beta}}$ on $\Lambda$ for all $\boldsymbol{\beta}$ satisfying $|\boldsymbol{\beta}| < 3$. Therefore we obtain with a triangle inequality
\begin{gather*}
	\left\| \boldsymbol{D}^{\boldsymbol{\gamma}} \left( q - \hat\Pi q \right) \right\|_{2-|\boldsymbol{\gamma}|} \le \sum_{\substack{\boldsymbol{\alpha} \in \boldsymbol{Q} \\ |\boldsymbol{\alpha} + \boldsymbol{\gamma}| = 3}} |q_{\boldsymbol{\alpha}+\boldsymbol{\gamma}}|\, \left\| \boldsymbol{D}^{\boldsymbol{\gamma}} \big(\boldsymbol{X}^{\boldsymbol{\alpha}+\boldsymbol{\gamma}} - \hat\Pi (\boldsymbol{X}^{\boldsymbol{\alpha}+\boldsymbol{\gamma}})\big) \right\|_{2-|\boldsymbol{\gamma}|} \le C \sum_{\substack{\boldsymbol{\alpha} \in \boldsymbol{Q} \\ |\boldsymbol{\alpha}+\boldsymbol{\gamma}| = 3}} |q_{\boldsymbol{\alpha}+\boldsymbol{\gamma}}|.
\end{gather*}
Note that the last summation is carried out over multi-indices of highest order for which we observe by \eqref{eq:polyq}
\begin{align*}
	\int_{\Lambda} \boldsymbol{D}^{\boldsymbol{\alpha} +\boldsymbol{\gamma}} q(x,y)\,\D x\D y &= 4 (\alpha_1+\gamma_1)!(\alpha_2+\gamma_2)!\, q_{\boldsymbol{\alpha}+\boldsymbol{\gamma}}\\ 
	&= \int_{\Lambda} \boldsymbol{D}^{\boldsymbol{\alpha} +\boldsymbol{\gamma}} u(x,y) \,\D x\D y,\qquad \boldsymbol{\alpha} \in \boldsymbol{Q},\, |\boldsymbol{\alpha}+\boldsymbol{\gamma}| = 3.
\end{align*}
Hence,
\begin{gather}\label{eq:polypartPi}
	\left\| \boldsymbol{D}^{\boldsymbol{\gamma}} \left( q - \hat\Pi q \right) \right\|_{2-|\boldsymbol{\gamma}|} \le C \sum_{|\boldsymbol{\alpha}|=3-|\boldsymbol{\gamma}|} \left| \int_{\Lambda} \boldsymbol{D}^{\boldsymbol{\alpha}} \boldsymbol{D}^{\boldsymbol{\gamma}} u(x,y) \,\D x\D y \right|.
\end{gather}
Collecting \eqref{eq:estimPi} and \eqref{eq:polypartPi} we obtain another main result.

\begin{thm}\label{thm:approxPi} 
	Let $\boldsymbol{\gamma}$ be a multi-index with $|\boldsymbol{\gamma}| \le 2$ and let $\hat \Pi$ denote the full $C^1-Q_2$ interpolation operator defined in \eqref{eq:interpolant} on the reference macro-element $\hat M$.
	For $u \in H^4(\Lambda)$ we have the estimate
	\begin{gather*}
		\big\|\boldsymbol{D}^{\boldsymbol{\gamma}} \big( u - \hat\Pi u \big)\big\|_0 \le C \left( \left| \boldsymbol{D}^{\boldsymbol{\gamma}} u \right|_{4-|\boldsymbol{\gamma}|} + \sum_{|\boldsymbol{\alpha}|=3-|\boldsymbol{\gamma}|} \left| \int_{\Lambda} \boldsymbol{D}^{\boldsymbol{\alpha}} \boldsymbol{D}^{\boldsymbol{\gamma}} u(x,y) \,\D x\D y \right| \right).
	\end{gather*}
\end{thm}

Using an affine mapping we can define $\Pi$ on a macro-element $M \in \mathcal{M}$ and extend the result like in the proof of Theorem \ref{thm:Pired}.

\begin{cor}\label{cor:respi} Let $M = [x_0-h_1,x_0+h_1]\times[y_0-h_2,y_0+h_2]$ be the axis-aligned macro-element that contains the four elements sharing the vertex $(x_0,y_0)$. With the reference mapping $F_M$ from $\Lambda \coloneqq[-1,1]^2$ to $M$ defined in \eqref{eq:mapping} one can introduce the full $C^1-Q_2$ interpolation operator $\Pi$ on $M$ by $\Pi u \coloneqq \hat \Pi \hat u\circ F_M^{-1}$ with $\hat u \coloneqq u \circ F_M$.
	Let $\boldsymbol{\gamma}$ be a multi-index with $|\boldsymbol{\gamma}| \le 2$ and $\boldsymbol{h}=(h_1,h_2)$. Then for $u \in H^4(M)$ we have the estimate
	\begin{gather}\label{eq:estimPiM}
		\|\boldsymbol{D}^{\boldsymbol{\gamma}} \left( u - \Pi u \right)\|_{0,M} \le C \left( \sum_{|\boldsymbol{\alpha}| = 4-|\boldsymbol{\gamma}|} \!\!\boldsymbol{h}^{\boldsymbol{\alpha}} \left\| \boldsymbol{D}^{\boldsymbol{\alpha}} \boldsymbol{D}^{\boldsymbol{\gamma}} u \right\|_{0,M} +\! \sum_{|\boldsymbol{\alpha}|=3-|\boldsymbol{\gamma}|} \!\!\boldsymbol{h}^{\boldsymbol{\alpha}} \left| \int_{M} \boldsymbol{D}^{\boldsymbol{\alpha}} \boldsymbol{D}^{\boldsymbol{\gamma}} u(x,y) \,\D x\D y \right| \right).
	\end{gather}
\end{cor}

\begin{cor}\label{cor:bicubic}
	Let $I_{12}$ be the analogue of $\Pi$ in the space of bicubic polynomials, i.e.~the Lagrangian basis functions in \eqref{eq:interpolant} are replaced by bicubic polynomials satisfying the same duality and Kronecker relations. Let $\boldsymbol{\gamma}$ be a multi-index with $|\boldsymbol{\gamma}| \le 2$. Then the estimate \eqref{eq:estimBFS} holds true for all $u \in C^2(M)$ with $\boldsymbol{D}^{\boldsymbol{\gamma}} u \in H^{4-\boldsymbol{\gamma}}(M)$.
\end{cor}
\begin{proof}
	Replace \eqref{eq:newton2D} by
	\begin{gather*}
		\begin{gathered}
			\big(I_{12} u\big)(x,y) = \sum_{i=1}^3 \sum_{j=1}^3 F_{i,j}(u) (x+1)^{i-1} (y+1)^{j-1} + \sum_{j=1}^3 F_{4,j}(u) (x+1)^2 (x-1) (y+1)^{j-1}\\
	+ \sum_{i=1}^3 F_{i,4}(u) (x+1)^{i-1} (y+1)^2 (y-1) + F_{4,4}(u) (x+1)^2(x-1)(y+1)^2(y-1).
		\end{gathered}
	\end{gather*}
	Now all arguments carry over to $I_{12}$. Observe that for this interpolation operator on the reference element we find
	$I_{12}(\boldsymbol{X}^{\boldsymbol{\beta}}) = \boldsymbol{X}^{\boldsymbol{\beta}}$ for all $\boldsymbol{\beta}$ satisfying $|\boldsymbol{\beta}| \le 3$. Hence, the additional error component containing $q$ vanishes.
\end{proof}

\begin{rem}
	It is possible to extend this result to Hermite interpolation by polynomials of higher degree as was done in \cite{CYM2009}. However, in that paper a different technique is used. By identifying possible associate functionals we enable the analysis of these operators using the unified theory of Apel and Dobrowolski, see \cite{AD92}.
\end{rem}

\begin{rem}
	Comparing the estimates \eqref{eq:estimBFS} and \eqref{eq:estimPiM} we see that the macro-interpolation attains in general a lower order than the corresponding element interpolation. This is due to the inability of the macro-interpolation operator to reproduce cubic polynomials.
\end{rem}

\begin{rem}\label{rem:choicepsi}
	The reduced macro-interpolation operator $\Pi^r$ is of even lower order compared to $\Pi$ and it appears doubtful to obtain anisotropic estimates for second order derivatives of the interpolation error of $\Pi^r$. However, it does not rely on so much regularity of the function $u$ to be interpolated. Note that the only difference of $\hat\Pi^r$ and $\hat\Pi$ is the choice of the functional determining the coefficient of the basis functions $\hat\psi_{i,j}$, $i,j=-1,1$, see also Table \ref{tab:operators}.
\end{rem}

\begin{table}
\label{tab:operators}
\centering
\begin{tabular}{lccc}
\toprule
& $\Pi u$ & $\Pi^r u$ & $\tilde \Pi u$\\
\midrule
\parbox{5.5cm}{coefficient of the basis-function\\
corresponding to mixed derivative} & $u_{xy}(x_i,y_j)$ & 0 & \parbox{3cm}{\centering{}$a_{ij}$, see \eqref{eq:aij},\\ non-local}\\
\midrule
required regularity of $u$ & $H^4(\Omega)$ & $H^3(\Omega)$ & $H^3(\Omega)$\\
\midrule
\parbox{5.5cm}{formal order of the first derivative\\
of the approximation error in $L_2$} & 2 & 1 & 2\\
\midrule
\parbox{6cm}{best possible order of the first deriva-\\
tive of the approximation error in $L_2$} & 3 & 2 & 2\\
\midrule
\parbox{6cm}{anisotropic estimates for the deriva-\\
tives $\boldsymbol{D}^{\boldsymbol{\gamma}}$ of the approximation error} & $|\boldsymbol{\gamma}| \le 2$ & $|\boldsymbol{\gamma}| \le 1$ & $|\boldsymbol{\gamma}| \le 1$\\
\bottomrule
\end{tabular}
\caption{Comparison of the (quasi-)interpolation operators $\Pi^r$ (Subsection \ref{subsec:reduced}), $\Pi$ (Subsection \ref{subsec:c1q2}) and $\tilde \Pi$ (Subsection \ref{subsec:SZ_macro}) on tensor product meshes.}
\end{table}

\subsection[\texorpdfstring{$C^1$ macro-element quasi-interpolation}{C1 macro-element quasi-interpolation}]{\texorpdfstring{A $C^1-Q_2$ macro-element quasi-interpolation operator of Scott-Zhang-type on tensor product meshes}{A C1-Q2 macro-element quasi-interpolation operator of Scott-Zhang-type on tensor product meshes}}\label{subsec:SZ_macro}

Let us start this subsection by recalling the definition of the Scott-Zhang quasi-interpolation operator $Z_h$. This operator was designed in order to obtain approximations to functions $u$ that are not sufficiently regular for nodal interpolation, see \cite{SZ90}. For instance, one might wish to approximate non-smooth functions. The basic idea is to use local $L_2$ projections on certain element edges to specify the coefficients of the approximating finite element function $Z_h u$. In contrast to the well-known Cl\'ement quasi-interpolant this approach can grant the projection property and the ability to preserve homogeneous boundary conditions.

Since we only want to demonstrate the basic ideas and fix some notation here we shall only consider the function space $V_h$ of continuous piecewise linears induced by a quasi-uniform partition $\Omega^h$ of the polygonal domain $\Omega \subset \R^2$ into triangles. For a more extensive presentation we refer the interested reader to \cite[Section 3.2]{A99}.

Let $\varphi_i$, $i \in I$ denote the nodal basis functions of $V_h$, i.e.~for any grid node $\boldsymbol{X}_j$, $j \in I$ the piecewise linear function $\varphi_i \in V_h$ satisfies
\begin{gather}\label{eq:nodalbasis}
	\varphi_i(\boldsymbol{X}_j) = \delta_{ij}.
\end{gather}
Next, for each node $\boldsymbol{X}_i$, $i \in I$ of the mesh we pick an edge $\sigma_i$ of a mesh triangle such that $\boldsymbol{X}_i \in \sigma_i$. If $\boldsymbol{X}_i \in \partial \Omega$ belongs to the boundary then we further restrict the choice of these edges by demanding $\sigma_i \subset \partial \Omega$. This is essential if one wishes to preserve homogeneous boundary conditions. Now the Scott-Zhang operator is defined by
\begin{gather}\label{eq:defZh}
	Z_h u(x,y) = \sum_{i \in I} \big(\Pi_{\sigma_i} u\big)(\boldsymbol{X}_i) \varphi_i(x,y),
\end{gather}
where $\Pi_{\sigma_i}:L_2(\sigma_i) \to P_1(\sigma_i)$, $i \in I$ is the local $L_2$-projection operator. It is easy to see that $Z_h$ inherits the property of being a \emph{projector}; actually, $Z_h v_h = v_h$ for all $v_h \in V_h$.

In order to provide an equivalent but more useful definition of the Scott-Zhang quasi-interpolant $Z_h u$ to $u$ let us assume that $\sigma_i$ is the straight line connecting the nodes $\boldsymbol{X}_i$ and $\boldsymbol{X}_j$ for some $j \in I$. On $\sigma_i$ let $\psi_i^d \in P_1(\sigma_i)$ denote a \emph{dual basis function}, uniquely determined by
\begin{gather}\label{eq:dualityrel}
	\int_{\sigma_i} \psi_i^d \varphi_i \,\D s = 1\qquad \text{and} \qquad  \int_{\sigma_i} \psi_i^d \varphi_j \,\D s = 0.
\end{gather}
Obviously $\Pi_{\sigma_i} u \in P_1(\sigma_i)$ can be represented as a linear combination of the restrictions of $\varphi_i$ and $\varphi_j$ to $\sigma_i$, i.e.
\begin{gather*}
	\Pi_{\sigma_i} u = b_i \varphi_i\big|_{\sigma_i} + b_j \varphi_j\big|_{\sigma_i}.
\end{gather*}
with real numbers $b_i$ and $b_j$ still to be specified. Hence, by \eqref{eq:dualityrel} and the definition of $\Pi_{\sigma_i}$ one finds that
\begin{gather}\label{eq:defbi}
	b_i = b_i \int_{\sigma_i} \psi_i^d \varphi_i \,\D s = \int_{\sigma_i} (\Pi_{\sigma_i} u) \psi_i^d\,\D s = \int_{\sigma_i} u \psi_i^d\,\D s.
\end{gather}
Finally, by the Kronecker relation \eqref{eq:nodalbasis} it is clear that $\big(\Pi_{\sigma_i} u\big)(\boldsymbol{X}_i) = b_i$. Consequently, with \eqref{eq:defbi} and \eqref{eq:defZh} one obtains
\begin{gather}\label{eq:defSZ}
	Z_h u(x,y) = \sum_{i \in I} \int_{\sigma_i} u \psi_i^d\,\D s\, \varphi_i(x,y).
\end{gather}

From its representation \eqref{eq:defSZ} it can be seen that the coefficients of the Scott-Zhang interpolant $Z_h u$ to $u$ are \emph{weighted local averages of $u$} over $\sigma_i$. In fact, the dual basis function $\psi_i^d$ can be interpreted as some weighting function since
\begin{gather*}
	\int_{\sigma_i} \psi_i^d \,\D s = \int_{\sigma_i} \psi_i^d (\varphi_i + \varphi_j) \,\D s = \int_{\sigma_i} \psi_i^d \varphi_i \,\D s = 1,
\end{gather*}
because of \eqref{eq:dualityrel} and the fact that $\{\varphi_i,\varphi_j\}$ is a partition of unity on $\sigma_i$.
In this light it is clear that stability and error estimates for $Z_h u$ over an element $T$ will be based on the values of derivatives of $u$ on an entire patch $\omega_T$ of elements around $T$. More precisely, a mesh triangle $T_j$ is a subset of $\omega_T$ iff $T$ has a vertex $\boldsymbol{X}_i$ such that $\boldsymbol{X}_i \in \sigma_i \subset T_j$.

Moreover, \eqref{eq:defSZ} extends the domain of definition. Naturally one would demand that for the function $u$ to be approximated it holds $u \in L_2(\sigma_i)$. However, since one has $\psi_i^d \in L_\infty(\sigma_i)$ for the polynomial dual basis functions $\psi_i^d$, $i\in I$, it is possible to apply $Z_h$ to any function $u$ such that its trace satisfies $u \in L_1(\sigma_i)$.

Under the assumption of a quasi-uniform mesh $\Omega^h$ and for $u\in W_{j,p}(\omega_T)$ the stability estimate
\begin{gather*}
	|Z_h u|_{W_{k,p}(T)} \le C h^{-k} \sum_{j=0}^\ell h^j |u|_{W_{j,p}(\omega_T)},\qquad \text{for $p\in [1,\infty]$, $0 \le k\le \ell \le 2$, $\ell \ge 1$},
\end{gather*}
can be found for instance in \cite{Dscript}. Next, standard arguments can be used to obtain the error estimate
\begin{gather*}
	|u - Z_h u|_{W_{k,p}(T)} \le C h^{\ell-k} |u|_{W_{\ell,p}(\omega_T)},\qquad \text{for $p\in [1,\infty]$, $0 \le k\le \ell \le 2$, $\ell \ge 1$}.
\end{gather*}

In \cite{A99} the Scott-Zhang operator is studied over anisotropic meshes of tensor product type. It is shown in Theorem 3.1 of that book that for $p\in[1,\infty]$ and some rectangular axis-aligned element $T$ this operator grants a stability estimate and an anisotropic quasi-interpolation error estimate for $\|Z_h u\|_{L_p(T)}$ and $\|u-Z_h u\|_{L_p(T)}$, respectively. Moreover, in \cite{A99} one finds a counterexample showing that in general the original Scott-Zhang operator does not provide such an estimate for derivatives of the approximation error. Therefore the original operator is modified in several ways in the Sections 3.3, 3.4 and 3.4 of \cite{A99} and anisotropic quasi-interpolation error estimates for the resulting operators are obtained. However, in the entire third chapter of that book it is assumed that there is no abrupt change in the element sizes. This means that while elements are allowed to have an arbitrary aspect ratio $h_x/h_y$ the edge length $h_x$ and $h_y$ have to vary gradually when moving from one element to a neighboring one, see \cite[(3.4) on page 100]{A99}. Clearly, this assumption is quite restrictive. For instance, the frequently used Shishkin-type meshes do not meet this requirement.

The paper \cite{AR08} deals with the possibility of applying the Scott-Zhang operator on Shishkin meshes $\Omega^N$ of tensor product type. The authors suggest to choose the element edges $\sigma_i$ for every mesh node $\boldsymbol{X}_i \in \sigma_i$, $i =1,\dots,N^2$ in a special way:
\begin{itemize}
 \item Certain edges $\sigma_i$ on the boundary may be chosen arbitrarily but the rest has to be parallel to one coordinate axis, say the $x$-axis.
 \item The ratio of the size of the patch $\omega_T$ to the size of the element $T$ must have an $\eps$-uniform upper bound in both coordinate directions. Consequently, for instance an element $T$ with a small side in the $x$-direction must be associated with a patch $\omega_T$ with the same property.
\end{itemize}
This modified Scott-Zhang operator $Q_N$ can be applied on a Shishkin mesh. Unfortunately the authors needed more  regularity of the regular solution component $S \in W_{2,\infty}(\Omega)$ of a convection-diffusion problem to prove optimal quasi-interpolation error estimates. Still, this result shows that the Scott-Zhang operator is quite flexible and that it can be tailored to suit an application on meshes with abrupt changes in the mesh sizes.

Note that the original Scott-Zhang operator and its modifications sketched so far were introduced for elements of Lagrange-type, i.e.~the linear functionals associated with the element are function evaluations in certain points.
The $C^1-Q_2$ macro-element however features also the point evaluation of derivatives. We want to apply the basic ideas of the Scott-Zhang operator to the components of the $C^1-Q_2$ macro-element space that are associated with the evaluation of the mixed second derivative. We do so with the aim of reducing the regularity required to prove anisotropic quasi-interpolation error estimates. In view of Remark \ref{rem:choicepsi} we study the question, whether it is possible to define a new $C^1$ interpolation operator $\tilde \Pi$ by introducing the right functional corresponding to the mixed derivative in such a way that estimates like \eqref{eq:estimPiM} are possible assuming only some $W_{3,p}$ regularity of $u$.

The quasi-interpolant $\tilde \Pi u$ to $u$ over some macro-element $M$ will be governed on an macro-element neighbourhood or macro-element patch around $M$. More precisely, the coefficients of the basis functions that correspond to the mixed derivative are calculated by some weighted averaging process of the mixed derivative of $u$ over macro-element edges that do not necessarily belong to $M$. Because of this non-local character of $\tilde \Pi$ we have to be very careful when a reference mapping to some reference domain is used to prevent imposing very restrictive conditions on the geometry of the macro-element patch. Instead we shall use some ideas of \cite{A99} and estimate directly on the world domain.

Let $\boldsymbol{X}_{ij}\coloneqq (x_i,y_j)$, $(i,j) \in I$ denote the nodes of a rectangular tensor product mesh $\mathcal{M}_{\boldsymbol{h}}$, generated by the two arbitrary one-dimensional triangulations $\{x_i\}_{i=0}^n$ and $\{y_j\}_{j=0}^m$. We shall refer to $\mathcal{M}_{\boldsymbol{h}}$ as \emph{macro-element mesh}. We use
\begin{gather*}
h_i \coloneqq \frac{1}{2}(x_i-x_{i-1}),\quad i=1,\dots,n\qquad\text{and} \qquad k_j \coloneqq \frac{1}{2}(y_j-y_{j-1}),\quad j=1,\dots,m,
\end{gather*}
to denote the local step sizes in $x$- and $y$-direction. Each macro-element $M \in \mathcal{M}_{\boldsymbol{h}}$ is subdivided into four congruent elements introducing new mesh nodes with subscript $\frac{1}{2}, \frac{3}{2}, \frac{5}{2}, \dots$. The generated \emph{element mesh} is denoted by $\mathcal{T}_{\boldsymbol{h}}$, see Figure \ref{fig:macro-element}. Note that one may chose a different refinement of the macro-element mesh such that the elements within one macro-element remain comparable in size. We choose the presented uniform one in order to simplify the presentation. Now each macro-element $M_{ij} \coloneqq [x_{i+1/2}-h_i,x_{i+1/2}+h_i]\times[y_{j+1/2}-k_j,y_{j+1/2}+k_j] \in \mathcal{M}_{\boldsymbol{h}}$ is centered around $(x_{i+1/2},y_{j+1/2})$ and consists of four elements of size $\boldsymbol{h}_{ij} \coloneqq (h_i,k_j)$.
Moreover, we denote by $I_{ij} \coloneqq I_{M_{ij}} \coloneqq \{(i,j),(i,j+1),(i+1,j),(i+1,j+1)\}$ the set of the four node indices that are vertices of $M_{ij}$.

Let $V_{\boldsymbol{h}}$ denote the space of $C^1-Q_2$ finite element functions over the tensor product mesh $\mathcal{T}_{\boldsymbol{h}}$. Using the reference mapping $F_{ij}:[-1,1]^2 \to M_{ij} \in \mathcal{M}_{\boldsymbol{h}}$ with
\begin{gather*}
	x = x_{i+1/2} + h_i \hat x,\qquad \text{and} \qquad y = y_{j+1/2} + k_j \hat y
\end{gather*}
we can specify basis functions of $V_{\boldsymbol{h}}$ in the world domain using \eqref{eq:basis2d}. Consider for instance the lower right vertex $\boldsymbol{X}_{i+1,j}$ of the macro-element $M_{ij}$. Then the basis function $\psi_{i+1,j}$ associated with the mixed derivative in $\boldsymbol{X}_{i+1,j}$ admits the representation
\begin{gather*}
	\psi_{i+1,j}\big|_{M_{ij}} = h_i k_j \hat \psi_{1,-1} \circ F_{ij}^{-1},
\end{gather*}
where $\hat \psi_{1,-1}$ was defined in \eqref{eq:basis2d}. Similarly,
\begin{gather*}
	\psi_{i,j}\big|_{M_{ij}} = h_i k_j \hat \psi_{-1,-1} \circ F_{ij}^{-1},\quad \!\psi_{i,j+1}\big|_{M_{ij}} = h_i k_j \hat \psi_{-1,1} \circ F_{ij}^{-1},\quad \!\psi_{i+1,j+1}\big|_{M_{ij}} = h_i k_j \hat \psi_{1,1} \circ F_{ij}^{-1}\!.
\end{gather*}

Let us now define a (quasi-)interpolation operator $\tilde \Pi$ by
\begin{gather}\label{eq:tildepi}
	\tilde \Pi u \big|_{M} \coloneqq \Pi^r \left(u\big|_{M} \right) + \sum_{(k,\ell) \in I_{M}} a_{k,\ell} \,\psi_{k,\ell},
\end{gather}
with the reduced interpolation operator $\Pi^r$ from Subsection \ref{subsec:reduced} and real numbers $a_{k,\ell}$ still to be determined. Note that the choice of $a_{k,\ell}$ does not alter the ability of $\tilde \Pi$ to reproduce inhomogeneous Dirichlet boundary conditions $g$ (if $g \in \left.V_h\right|_{\partial \Omega}, i.e.~g\in C^1(\partial \Omega)$ and piecewise quadratic), because $\psi_{k,\ell}$, $(k,\ell) \in I_{ij}$ vanishes on the boundary of $M_{ij}$, see Figure \ref{fig:basis}.

The local choices $a_{k,\ell} = 0$ and $a_{k,\ell} = u_{xy}(\boldsymbol{X}_{k\ell})$ correspond to $\tilde \Pi = \Pi^r$ and $\tilde \Pi = \Pi$, respectively. Next, we want to follow the approach of Scott and Zhang \cite{SZ90} and define the coefficients $a_{k,\ell}$ using certain mean values of $u_{xy}$ along macro-element edges $\sigma_{k,\ell}$. Hence, as already mentioned, the interpolation operator is of non-local character and the theory developed in Section \ref{sec:macrotheory} can not be applied to $\tilde \Pi$. However, the definition of $\tilde \Pi u$ on a macro-element $M$ is not global but shall be based on the values of $u_{xy}$ on the \emph{macro-element neighbourhood} $S_M$ of $M$:
\begin{gather}\label{eq:macroelemSM}
	S_M = \bigcup \{ M'\,:\, M' \in \mathcal{M}_{\boldsymbol{h}},\, M' \cap M \neq \emptyset \}.
\end{gather}
More precisely, we associate every node $\boldsymbol{X}_{ij}$ of the macro-element mesh $\mathcal{M}_{\boldsymbol{h}}$ with a macro-element edge $\sigma_{i,j} \subset S_M$ such that $\boldsymbol{X}_{ij} \in \sigma_{i,j}$, see Figure \ref{fig:neighbor-sigma} for some illustration.

\begin{figure}
	\centering
	\includegraphics[height=6.5cm]{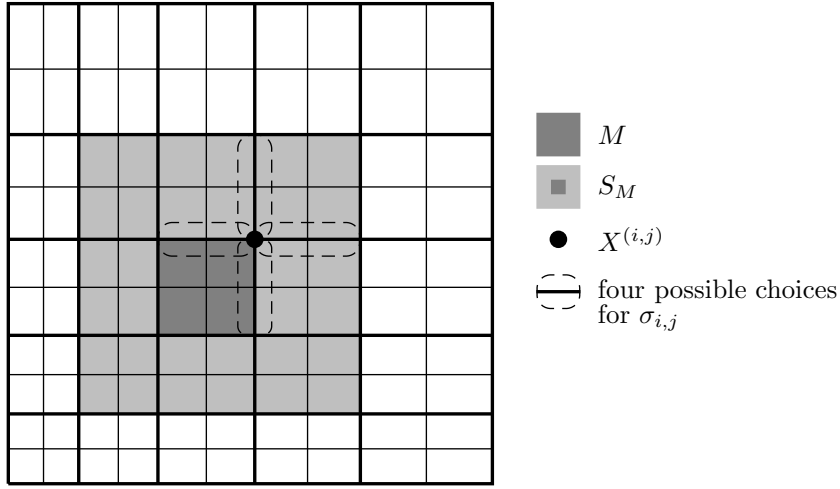}
	\caption{Definitions on the macro-element mesh.}
	\label{fig:neighbor-sigma}
\end{figure}

Once the edges $\sigma_{i,j} \in S_M$ for $(i,j) \in I_M$ are chosen we can define the \emph{associated macro-element patch} $\omega_M$ around $M$. Another patch neighbourhood of $M$ is needed because the value of our quasi-interpolation operator will be based on values of its interpolant on $M$ and $\sigma_{i,j}$. Hence, if the approximation error is estimated norms of the interpolant over a patch of macro-elements will appear on the right hand side of the estimate. On the other hand estimates that use the full neighbourhood $S_M$ might be too crude.

\begin{defn}\label{defn:rm}
The smallest (in area) rectangular patch of macro-elements that contains the convex hull of $\{\sigma_{i,j}\,:\, (i,j)\in I_M\}$ is called the \emph{associated macro-element patch} $\omega_M$ around $M$.
\end{defn}

Note that $M \subset \omega_M \subset S_M$. If for instance at each node $\boldsymbol{X}_{ij}$ of the tensor product macro-element mesh the set $\sigma_{i,j}$ is chosen to be the edge to the left of that point, then the associated macro-element patch $\omega_M$ around $M$ is defined as the union of $M$ and its left macro-element neighbour.

We plan to set 
\begin{gather}\label{eq:aij}
	a_{i,j} = \frac{\partial^2}{\partial x\partial y} \left( \Pi_{\sigma_{i,j}} u \right)(\boldsymbol{X}_{ij}),
\end{gather}
with a suitable projector $\Pi_{\sigma_{i,j}}$. Assuming that $\sigma_{i,j}$ is the horizontal macro-element edge $(x_i,x_{i+1})\times \{y_j\}$ that connects the macro-element vertices $\boldsymbol{X}_{ij}$ and $\boldsymbol{X}_{i+1,j}$ we set 
\begin{gather}\label{eq:ansatzpisigma}
	\Pi_{\sigma_{i,j}} u = b_{i,j} \psi_{i,j} + b_{i+1,j} \psi_{i+1,j}.
\end{gather}
We determine the real coefficients $b_{i,j}$ and $b_{i+1,j}$ by
\begin{gather}\label{eq:pisigma}
	\int_{\sigma_{i,j}} \frac{\partial^2}{\partial x\partial y} \left( \Pi_{\sigma_{i,j}} u \right) (x,y_j) v(x) \,\D x =
	\int_{\sigma_{i,j}} \frac{\partial^2 u(x,y_j)}{\partial x\partial y} v(x) \,\D x\qquad \text{for all $v \in \mathbb{V}_{i+1/2}$},
\end{gather}
with $\mathbb{V}_{i+1/2} = \spn \{ \psi_i+\psi_{i+1}, \theta_{i+1/2}\}$ and $\theta_{i+1/2}(x) = \left(\frac{x-x_{i+1}}{h_{i}}\right)^2-\frac{1}{6}$.
Here $\psi_i$ and $\psi_{i+1}$ are the one dimensional spline basis functions from \eqref{eq:basis} scaled to $\sigma_{i,j}$, i.e.
\begin{gather}\label{eq:worldbasis}
\begin{alignedat}{2}
	\psi_{i} (x) &= \frac{h_{i}}{4} - \frac{x-x_{i+1/2}}{2} \left\{ \begin{alignedat}{2} &-\frac{3(x-x_{i+1/2})^2}{4 h_{i}},&\quad x_i &\le x \le x_{i+1/2}, \\ &+\frac{(x-x_{i+1/2})^2}{4 h_{i}},&\quad x_{i+1/2} &\le x \le x_{i+1}, \end{alignedat} \right.\qquad &\\
	\psi_{i+1} (x) &= -\frac{h_{i}}{4} - \frac{x-x_{i+1/2}}{2} \left\{ \begin{alignedat}{2} &-\frac{(x-x_{i+1/2})^2}{4 h_{i}},&\quad x_i &\le x \le x_{i+1/2}, \\ &+\frac{3(x-x_{i+1/2})^2}{4 h_{i}},&\quad x_{i+1/2} &\le x \le x_{i+1} \end{alignedat} \right.
\end{alignedat}
\end{gather}
and $h_{i} = x_{i+1}-x_{i+1/2} = x_{i+1/2}-x_{i}$.
Note that these functions have $\Landau(h_{i})$ scalings while their first derivatives have $\Landau(1)$ scalings on $(x_i,x_{i+1})$. Moreover, we would like to recall $\psi_{i,j}(x,y) = \psi_i(x) \psi_j(y)$ and point out that the ansatz \eqref{eq:ansatzpisigma} is justified by the fact that $\frac{\partial^2}{\partial x\partial y} \psi_{k,\ell}$ vanishes on $\sigma_{i,j}$ if $\boldsymbol{X}_{k,\ell} \notin \sigma_{i,j}$, i.e.~only adjacent basis functions contribute to the integral over $\sigma_{i,j}$ on the left hand side of \eqref{eq:pisigma}. The choice of $\mathbb{V}_{i+1}$ will become clear in the next Lemma. Basically, we need that this space is $L_2$-orthogonal to certain functions to prove that discrete functions are left invariant, see Lemma \ref{lem:vhpreservation}.

Next, we want to find a more suitable representation of $a_{i,j}$ according to \eqref{eq:aij}. For this purpose let us define the dual basis function $\psi_{i}^d \in \mathbb{V}_{i+1/2}$ by
\begin{gather}\label{eq:dualbasis}
	\int_{\sigma_{i,j}} \frac{\partial^2 \psi_{k,j}(x,y_j)}{\partial x\partial y} \psi_{i}^d(x) \,\D x = \delta_{k,i}.
\end{gather}
This system yields with \eqref{eq:ansatzpisigma}
\begin{gather*}
	b_{i,j} = \sum_{k=i}^{i+1} b_{k,j} \int_{\sigma_{i,j}} \frac{\partial^2 \psi_{k,j}(x,y_j)}{\partial x\partial y} \psi_{i}^d(x) \,\D x = \int_{\sigma_{i,j}} \frac{\partial^2}{\partial x\partial y} \left( \Pi_{\sigma_{i,j}} u \right) (x,y_j) \psi_{i}^d(x) \,\D x.
\end{gather*}
An application of \eqref{eq:pisigma} then gives
\begin{gather*}
	b_{i,j} = \int_{\sigma_{i,j}} \frac{\partial^2 u (x,y_j)}{\partial x\partial y} \psi_{i}^d(x) \,\D x.
\end{gather*}
Finally, we use the Lagrange relation $\frac{\partial^2 \psi_{k,\ell}(x_i,y_j)}{\partial x\partial y} = \delta_{k,i} \delta_{\ell,j}$ to obtain
\begin{gather}\label{eq:formulaaij}
	a_{i,j} = \frac{\partial^2}{\partial x\partial y} \left( \Pi_{\sigma_{i,j}} u \right)(\boldsymbol{X}_{ij}) = b_{i,j} = \int_{\sigma_{i,j}} \frac{\partial^2 u (x,y_j)}{\partial x\partial y} \psi_{i}^d(x) \,\D x.
\end{gather}
Hence, $a_{i,j}$ is indeed a weighted mean value of $u_{xy}$ on the macro-element edge $\sigma_{i,j}$. For the weighting function we solve \eqref{eq:dualbasis} to find 
\begin{gather}\label{eq:duallocal}
\begin{alignedat}{2}
	\psi_{i}^d (x) &= - \frac{h_{i}^2 + 12 h_{i} (x-x_{i+1/2})}{2 h_{i}^3} \left\{ \begin{alignedat}{2} & - \frac{3 (x-x_{i+1/2})^2}{h_{i}^3},&\quad x_i &\le x \le x_{i+1/2}, \\ & + \frac{9 (x-x_{i+1/2})^2}{h_{i}^3},&\quad x_{i+1/2} &\le x \le x_{i+1}, \end{alignedat} \right.\qquad &\\
	\psi_{i+1}^d (x) &= - \frac{h_{i}^2 -12 h_{i} (x-x_{i+1/2})}{2 h_{i}^3} \left\{ \begin{alignedat}{2} & +\frac {9 (x-x_{i+1/2})^2}{h_{i}^3},&\quad x_i &\le x \le x_{i+1/2}, \\ & -\frac{3 (x-x_{i+1/2})^2}{h_{i}^3},&\quad x_{i+1/2} &\le x \le x_{i+1}. \end{alignedat} \right.
\end{alignedat}
\end{gather}
Here $\psi_{i+1}^d \in \mathbb{V}_{i+1/2}$ is the other dual basis function on $\sigma_{i,j}$, satisfying \eqref{eq:dualbasis} with $i$ replaced by $i+1$. Note that $\psi_i^d,\psi_{i+1}^d \in C^1\big(x_i,x_{i+1}\big)$ and that $\|\psi_i^d\|_{L_\infty(\sigma_{i,j})} \le C h_{i}^{-1}$ with a similar bound for $\|\psi_{i+1}^d\|_{L_\infty(\sigma_{i,j})}$. A simple calculation shows the important property
\begin{gather}\label{eq:proppsid}
	\int_{\sigma_{i,j}} \psi_i^d(x) \,\D x = 1,
\end{gather}
which again underlines the role of $\psi_i^d$ as a weighting function.
\begin{rem}
	In Section 4 of \cite{HHZ2011} a similar macro-element edge based approach is used to reduce the regularity demanded of the function to be interpolated. There, the Girault-Scott operator is extended to the $C^1-Q_2$ macro-element. In \cite{HHZ2011} integration by parts is applied to an identity similar to \eqref{eq:formulaaij} which results in a different system defining the dual basis functions. However, this approach appears to be not suitable for anisotropic quasi-interpolation error estimates. Another difference to that paper is that here we mix local and non-local functionals for the definition of our quasi-interpolation operator which is reflected in the sophisticated choice of $\mathbb{V}_{i+1/2}$.
\end{rem}
\begin{lem}\label{lem:vhpreservation}
	$\tilde \Pi$ preserves $V_{\boldsymbol{h}}$ functions, i.e.
	\begin{gather}\label{eq:vhpreservation}
		\tilde \Pi v_{\boldsymbol{h}} = v_{\boldsymbol{h}}\qquad \text{for all $v_{\boldsymbol{h}} \in V_{\boldsymbol{h}}$.}
	\end{gather}
\end{lem}
\begin{proof}
	Since every function $v_{\boldsymbol{h}} \in V_{\boldsymbol{h}}$ is uniquely determined by the nodal values
	\begin{gather*}
		v_{\boldsymbol{h}}(\boldsymbol{X}_{ij}),\qquad \frac{\partial v_{\boldsymbol{h}}}{\partial x}(\boldsymbol{X}_{ij}),\qquad \frac{\partial v_{\boldsymbol{h}}}{\partial y}(\boldsymbol{X}_{ij}),\qquad \frac{\partial^2 v_{\boldsymbol{h}}}{\partial x\partial y}(\boldsymbol{X}_{ij}),
	\end{gather*}
	in the macro-element vertices $\boldsymbol{X}_{ij}$ with $(i,j) \in I$, it remains to prove that these functionals are invariant to the application of the quasi-interpolation operator $\tilde \Pi$. Let us prove the identity of the last functional involving the mixed derivative as the other ones are trivial. We observe with \eqref{eq:formulaaij} that
	\begin{gather}\label{eq:valxy}
		\frac{\partial^2 \tilde \Pi v_{\boldsymbol{h}}}{\partial x\partial y} (\boldsymbol{X}_{ij}) = a_{i,j} = \int_{\sigma_{i,j}} \frac{\partial^2 v_{\boldsymbol{h}} (x,y_j)}{\partial x\partial y} \psi_{i}^d(x) \,\D x.
	\end{gather}
	Since $v_{\boldsymbol{h}} \in V_{\boldsymbol{h}}$ it can be expanded on the macro-element $M$ considered in terms of the basis functions $\varphi_{i,j}$, $\phi_{i,j}$, $\chi_{i,j}$ and $\psi_{i,j}$, $(i,j) \in I_M$ according to \eqref{eq:basis2d}. For the mixed derivative on $\sigma_{i,j}$ we find
	\begin{align*}
		\left.\frac{\partial^2 v_{\boldsymbol{h}}}{\partial x\partial y}\right|_{\sigma_{i,j}} &= \sum_{\ell =i}^{i+1} \frac{\partial v_{\boldsymbol{h}}}{\partial y}(\boldsymbol{X}_{\ell{}j}) \frac{\partial^2}{\partial x\partial y} \chi_{\ell,j} + \frac{\partial^2 v_{\boldsymbol{h}}}{\partial x \partial y}(\boldsymbol{X}_{\ell{}j}) \frac{\partial^2}{\partial x\partial y} \psi_{\ell,j} \\
		&= \sum_{\ell = i}^{i+1} \frac{\partial v_{\boldsymbol{h}}}{\partial y}(\boldsymbol{X}_{\ell{}j}) \varphi_\ell'(x) \underbrace{\psi_j'(y_j)}_{=1} + \frac{\partial^2 v_{\boldsymbol{h}}}{\partial x \partial y}(\boldsymbol{X}_{\ell{}j}) \frac{\partial^2}{\partial x\partial y} \psi_{\ell,j},
	\end{align*}
	since the mixed derivative of the other basis functions vanishes on $\sigma_{i,j}$. The functions $\varphi_\ell'$, $\ell=i,i+1$ are continuous, piecewise linear and vanish in the endpoints of the interval $(x_i,x_{i+1})$. Hence, the odd function $\psi_i+\psi_{i+1}$ is $L_2(\sigma_{i,j})$ orthogonal to them. A direct calculation shows the same orthogonality relation for $\theta_{i+1}$, i.e.~ $\mathbb{V}_{i+1/2} \perp_{L_2(\sigma_{i,j})} \varphi_\ell'$, $\ell= i,i+1$. 
Using this orthogonality and \eqref{eq:dualbasis} in \eqref{eq:valxy} we see that
	\begin{gather*}
		\frac{\partial^2 \tilde \Pi v_{\boldsymbol{h}}}{\partial x\partial y} (\boldsymbol{X}_{ij}) = \int_{\sigma_{i,j}} \frac{\partial^2 v_{\boldsymbol{h}} (x,y_j)}{\partial x\partial y} \psi_{i}^d(x) \,\D x = \frac{\partial^2 v_{\boldsymbol{h}}}{\partial x\partial y} (\boldsymbol{X}_{ij}).
	\end{gather*}
	From which the assertion follows.
\end{proof}

\begin{rem}
	With \eqref{eq:formulaaij} the quasi-interpolation operator $\tilde \Pi$ from \eqref{eq:tildepi} is a projector due to \eqref{eq:vhpreservation}.
\end{rem}

\begin{lem}\label{lem:polypreserve}
	For some macro-element $M\in \mathcal{M}_{\boldsymbol{h}}$ let $v \in Q_2(\omega_M)$ i.e.~$v$ is biquadratic on the associated macro-element patch $\omega_M$ around the macro-element $M$, then
	\begin{gather}
		\left.\tilde \Pi v\right|_M = \left.v\right|_M.
	\end{gather}
\end{lem}
\begin{proof}
	We use the $Q_2$-preservation of the interpolation operator $\Pi$ and Lemma \ref{lem:vhpreservation}:
	\begin{gather*}
		\left. v\right|_M = \left. (\Pi v) \right|_M = \left.\big(\tilde \Pi(\Pi v) \big)\right|_M = \left. (\tilde \Pi v)\right|_M.
	\end{gather*}
	In the second identity we applied Lemma \ref{lem:vhpreservation} and need $v \in Q_2(\omega_M)$ because of the non-local character of $\tilde \Pi$.
\end{proof}

The following lemma is taken from \cite[Theorem 1.1]{DL04}.
\begin{lem}\label{lem:DL04}
	Let $\Omega \subset \R^n$ be convex with diameter $d$ and let $g \in W_{\nu,p}(\Omega)$, $\nu \in \N$, $p \in [1,\infty]$. Then there exists a polynomial $p_\nu^g \in P_{\nu-1}$ for which
	\begin{gather*}
		|g-p_\nu^g|_{W_{k,p}(\Omega)} \le C(n,\nu) d^{\nu-k} |g|_{W_{\nu,p}(\Omega)},\quad k=0,1,\dots,\nu.
	\end{gather*}
\end{lem}
Here the polynomial
\begin{gather*}
	p_\nu^g(x) = Q^\nu\big( g(A\cdot)\big) \big(A^{-1} x\big)
\end{gather*}
is constructed using the averaged Taylor polynomial $Q^\nu$ over the ball $B(0,1) \subset \R^n$ and $A$ is John's optimal affine transform with respect to $\Omega$, cp.~\cite{DL04}. The basic idea of this paper is the usage of ellipsoids in contrast to balls which is more suitable for anisotropic elements to which we want to apply this result. Yet, we shall first give a small modification of it.
\begin{lem}\label{lem:lemmaDLS}
	Let $\Omega \subset \R^n$ be convex with diameter $d$, $\boldsymbol{\gamma}$ be a multi-index with $|\boldsymbol{\gamma}|=m \in \N$, and let $v \in W_{\ell,p}(\Omega)$, $m \le \ell \in \N$, $p\in[1,\infty]$. Then there exists a polynomial $p_\ell^v \in P_{\ell-1}$ for which
	\begin{gather*}
		|\boldsymbol{D}^{\boldsymbol{\gamma}} (v-p_\ell^v)|_{W_{k,p}(\Omega)} \le C(n,\ell-m) d^{\ell-m-k} |\boldsymbol{D}^{\boldsymbol{\gamma}} v|_{W_{\ell-m,p}(\Omega)},\quad k=0,1,\dots,\ell-m.
	\end{gather*}
\end{lem}
\begin{proof}
	First assume that $v \in C^\ell(\Omega)$. Applying Lemma \ref{lem:DL04} with $\nu = \ell-m$ and $g = \boldsymbol{D}^{\boldsymbol{\gamma}} v$ yields the existence of a polynomial $p_{\ell-m}^{\boldsymbol{D}^{\boldsymbol{\gamma}} v} \in P_{\ell-m-1}$ such that
	\begin{gather*}
		|\boldsymbol{D}^{\boldsymbol{\gamma}} v - p_{\ell-m}^{\boldsymbol{D}^{\boldsymbol{\gamma}} v}|_{k,p} \le C(n,\ell-m) d^{\ell-m-k} |\boldsymbol{D}^{\boldsymbol{\gamma}} v|_{\ell-m,p},\quad k=0,1,\dots,\ell-m.
	\end{gather*}
	Next one finds that for $1\le m\le \ell-1$
	\begin{gather*}
		p_{\ell-m}^{\boldsymbol{D}^{\boldsymbol{\gamma}} v}(x) = Q^{\ell-m}\big((\boldsymbol{D}^{\boldsymbol{\gamma}} v)(A\cdot)\big)\big(A^{-1}x\big) = \boldsymbol{D}^{\boldsymbol{\gamma}} \Big( Q^\ell \big(v (A\cdot ) \big)\big(A^{-1}x\big) \Big) = \boldsymbol{D}^{\boldsymbol{\gamma}} p_\ell^v,
	\end{gather*}
	i.e.~the averaged Taylor polynomial and differentiation commute in some sense \cite[Corollary 3.4]{DL04}. Now the case $v \in W_{\ell,p}(\Omega)$ follows by standard arguments based on the density of $C^\infty(\Omega)$ in $W_{\ell,p}(\Omega)$.
	For $m=0$ the assertion of the Lemma is given by Lemma \ref{lem:DL04} and for $m=\ell$ the assertion is trivial.
\end{proof}

\begin{rem}
	A slightly more general result is given in \cite[Lemma 2.1]{A99}. However, there the dependencies of the constant of geometrical properties of the domain considered is not stated explicitly.
\end{rem}

\begin{assump}\label{assump:ass1}
	Let for each node $\boldsymbol{X}_{ij}$ of a macro-element $M$ the macro-element edges $\sigma_{i,j}$ be chosen in such a way that with the associated macro-element patch $\omega_M$ around $M$ it holds
	\begin{gather}\label{eq:ass1}
		h_k(\omega_M) \le C h_k(M)\quad k=1,2.
	\end{gather}
	Here and in the following $h_k(T)$ denotes the size of an axis-aligned rectangle $T$ in $x_k$-direction, $k=1,2$. Moreover, we set $\boldsymbol{h}_M = (h_1(M),h_2(M))$ for any macro-element $M$.
\end{assump}

\begin{lem}\label{lem:lemmapoly}
	Based on Assumption \ref{assump:ass1} for any $u\in W_{\ell,p}(\omega_M)$ there is a polynomial $q \in P_{\ell-1}(\omega_M)$ with
	\begin{gather*}
		\sum_{|\boldsymbol{\alpha}|\le \ell-m} \boldsymbol{h}_M^{\boldsymbol{\alpha}} \left| \boldsymbol{D}^{\boldsymbol{\alpha}} (u-q)\right|_{W_{m,p}(\omega_M)} \le C \sum_{|\boldsymbol{\alpha}|= \ell-m} \boldsymbol{h}_M^{\boldsymbol{\alpha}} \left| \boldsymbol{D}^{\boldsymbol{\alpha}} u\right|_{W_{m,p}(\omega_M)},
	\end{gather*}
	for all $m=0,\dots,\ell$.
\end{lem}
\begin{proof}
	Using an affine transformation we can map the macro-element $M$ to the reference macro-element $[-1,1]^2$. This transformation maps $\omega_M$ to $\hat \omega_M$. Based on \eqref{eq:ass1} we see that the diameter of the rectangle $\hat \omega_M$ can be bounded by a constant. Hence, we can apply Lemma \ref{lem:lemmaDLS} in the transformed domain. Scaling back to $\omega_M$ we obtain due to $\boldsymbol{h}_M^{\boldsymbol{\alpha}} \boldsymbol{D}^{\boldsymbol{\alpha}} = \boldsymbol{\hat D}^{\boldsymbol{\alpha}}$ that
	\begin{gather*}
		\sum_{|\boldsymbol{\alpha}| \le \ell-m} \boldsymbol{h}_M^{\boldsymbol{\alpha}} \|\boldsymbol{D}^{\boldsymbol{\alpha} + \boldsymbol{\gamma}} (u-q)\|_{L_p(\omega_M)} \le C \sum_{|\boldsymbol{\alpha}| = \ell-m} \boldsymbol{h}_M^{\boldsymbol{\alpha}} \|\boldsymbol{D}^{\boldsymbol{\alpha} + \boldsymbol{\gamma}} u\|_{L_p(\omega_M)},
	\end{gather*}
	for a multi-index $\boldsymbol{\gamma}$ with $|\boldsymbol{\gamma}|=m$. The assertion follows by summing up over all of these multi-indices.
\end{proof}

\begin{rem}
	A similar lemma is given in \cite[Lemma 3.1]{A99}. However, there the mesh is required to have no abrupt changes in the element sizes. Clearly, Assumption \ref{assump:ass1} can be dropped then. Note that \eqref{eq:ass1} can also be found in the paper \cite{AR08}.
\end{rem}

\begin{lem}[Stability of $\tilde \Pi$]\label{lem:stab}
	Under Assumption \ref{assump:ass1} the quasi-interpolation operator $\tilde \Pi$ satisfies the stability estimate
	\begin{gather*}
		\big| \tilde \Pi u \big|_{W_{1,p}(M)} \le C C_{M,p} \sum_{|\boldsymbol{\alpha}|\le 2} \boldsymbol{h}_M^{\boldsymbol{\alpha}} |\boldsymbol{D}^{\boldsymbol{\alpha}} u|_{W_{1,p}(\omega_M)}
	\end{gather*}
	with
	\begin{gather*}
		C_{M,p} \coloneqq \left(\frac{\meas M}{\min_{M' \in \mathcal{M}_{\boldsymbol{h}},\,M' \subset \omega_M} \meas M'}\right)^{1/p} \ge 1,
	\end{gather*}
	provided that $u \in W_{3,p}(\omega_M) \cap C^1(M)$ with $p\in[1,\infty]$. 
\end{lem}
\begin{proof}
	Let $M\in \mathcal{M}_{\boldsymbol{h}}$ be a macro-element and set $\boldsymbol{h}_M = (h_1(M),h_2(M))$. We consider a first derivative in $x$-direction. Using the definition of $\tilde \Pi$ and a triangle inequality we find that
	\begin{gather}\label{eq:stabtriangle}
		\big\|\big(\tilde \Pi u\big)_x\big\|_{L_p(M)} \le \|(\Pi^r u)_x\|_{L_p(M)} + \left \| \sum_{(i,j)\in I_M}  a_{i,j} \frac{\partial \psi_{i,j}}{\partial x} \right\|_{L_p(M)},
	\end{gather}
	with coefficients $a_{i,j}$ depending on the direction of $\sigma_{i,j}$ given by
	\begin{gather}\label{eq:aijfinal}
	 a_{i,j} = \left\{\begin{alignedat}{2}
		&\int_{\sigma_{i,j}} \frac{\partial^2 u (x,y_j)}{\partial x\partial y} \psi_{i}^d(x) \,\D x &\quad &\text{if $\sigma_{i,j}$ is horizontal,}\\
		&\int_{\sigma_{i,j}} \frac{\partial^2 u (x_i,y)}{\partial x\partial y} \psi_{j}^d(y) \,\D y &\quad &\text{if $\sigma_{i,j}$ is vertical.}\\
		\end{alignedat}\right.
	 \end{gather}
	 We estimate the first term on the right hand side of \eqref{eq:stabtriangle} using Theorem \ref{thm:Pired}
	 \begin{gather}
	 \begin{aligned}\label{eq:stabpir}
		\big\|\big(\Pi^r u\big)_x\big\|_{L_p(M)} &\le \|u_x\|_{L_p(M)} + \|(u - \Pi^r u)_x\|_{L_p(M)}\\
		&\le C \sum_{|\boldsymbol{\alpha}|\le 2} \boldsymbol{h}_M^{\boldsymbol{\alpha}} \|\boldsymbol{D}^{\boldsymbol{\alpha}} u_x\|_{L_p(M)}.
	 \end{aligned}
	 \end{gather}
	 For the other term we use $\left\|\frac{\partial }{\partial x} \psi_{i,j} \right\|_{L_\infty(M)} \le C h_2(M)$ which yields
	 \begin{gather}\label{eq:stabhol1}
		\left \| \sum_{(i,j)\in I_M}  a_{i,j} \frac{\partial \psi_{i,j}}{\partial x} \right\|_{L_p(M)} \le C (\meas M)^{1/p} h_2(M) \max_{(i,j) \in I_M} |a_{i,j}|.
	 \end{gather}
	 Next we use $\|\psi^d_k\|_{\infty,\sigma_{i,j}} \le C \meas(\sigma_{i,j})^{-1}$ for $k=i,j$ and obtain with a H\"older inequality
	 \begin{gather}\label{eq:stabhol2}
		|a_{i,j}| \le C \meas(\sigma_{i,j})^{-1} \left\|\frac{\partial^2 u}{\partial x\partial y}\right\|_{L_1(\sigma_{i,j})}\quad \text{for $(i,j) \in I_M$}.
	 \end{gather}
	 Set
	 \begin{gather*}
		M' \coloneqq \argmin_{\substack{\tilde M \in \mathcal{M}_{\boldsymbol{h}} \\ \tilde M \subset \omega_M}} (\meas \tilde M),
	 \end{gather*}
	 i.e.~the macro-element $M' \in \mathcal{M}_{\boldsymbol{h}}$ belongs to the associated macro-element patch $\omega_M$ around $M$ and realizes the smallest surface measure.
	 Using the embeddings $W_{1,p} (\hat \omega_M) \hookrightarrow W_{1,p}(\hat M') \hookrightarrow L_1(\hat\sigma_{i,j})$ in a transformed domain $\hat \omega_M$ and scaling back to the original one, we see that
	 \begin{gather}\label{eq:stabtrace}
		\|v\|_{L_1(\sigma_{i,j})} \le \meas(\sigma_{i,j}) \meas(M')^{-1/p} \sum_{|\boldsymbol{\alpha}| \le 1} \boldsymbol{h}_M^{\boldsymbol{\alpha}} \|\boldsymbol{D}^{\boldsymbol{\alpha}} v\|_{L_p(\omega_M)}
	 \end{gather}
	 for $v\in W_{1,p}(\omega_M)$. Here we also used Assumption \ref{assump:ass1}. Collecting \eqref{eq:stabhol1}, \eqref{eq:stabhol2} and \eqref{eq:stabtrace} with $v=\frac{\partial^2 u}{\partial x\partial y}$ we obtain 
	 \begin{gather}\label{eq:stabprecise}
		\left \| \sum_{(i,j)\in I_M}  a_{i,j} \frac{\partial \psi_{i,j}}{\partial x} \right\|_{L_p(M)} \le C \frac{(\meas M)^{1/p}}{(\meas M')^{1/p}} \sum_{|\boldsymbol{\alpha}| \le 1} \boldsymbol{h}_M^{\boldsymbol{\alpha} + (0,1)} \left\|\boldsymbol{D}^{\boldsymbol{\alpha}+(0,1)} u_x\right\|_{L_p(\omega_M)}.
	 \end{gather}
	 Together with \eqref{eq:stabtriangle} and \eqref{eq:stabpir} the assertion of the lemma is proven since the first derivative in $y$-direction can be estimated analogously.
\end{proof}

\begin{thm}\label{thm:tapprox}
	Based on Assumption \ref{assump:ass1} for the quasi-interpolation operator $\tilde \Pi$ the approximation error estimate
	\begin{gather}\label{eq:tapprox}
		|u- \tilde \Pi u|_{W_{1,p}(M)} \le C C_{M,p} \sum_{|\boldsymbol{\alpha}| = 2} \boldsymbol{h}_M^{\boldsymbol{\alpha}} |\boldsymbol{D}^{\boldsymbol{\alpha}} u|_{W_{1,p}(\omega_M)}
	\end{gather}
	holds true provided that $u \in W_{3,p}(\omega_M) \cap C^1(M)$ with $p\in[1,\infty]$. Here $C_{M,p}$ is the constant from Lemma \ref{lem:stab}.
\end{thm}
\begin{proof}
	Let $q \in P_2(\omega_M)$ denote the polynomial of Lemma \ref{lem:lemmapoly} with $\ell = 3$. A triangle inequality gives
	\begin{gather}\label{eq:tapproxtriangle}
		|u- \tilde \Pi u|_{W_{1,p}(M)} \le |u-q|_{W_{1,p}(M)} + |q - \tilde \Pi u|_{W_{1,p}(M)}.
	\end{gather}
	As a polynomial $q \in P_2(\omega_M)$ is preserved by $\tilde \Pi$ on the macro-element $M$ considered, see Lemma \ref{lem:polypreserve}. Hence, we can use the stability of $\tilde \Pi$ shown in Lemma \ref{lem:stab} to get a bound for the second summand
	\begin{gather}\label{eq:tapprox2}
		|q - \tilde \Pi u|_{W_{1,p}(M)} = |\tilde \Pi (q - u)|_{W_{1,p}(M)} \le C C_{M,p} \sum_{|\boldsymbol{\alpha}|\le 2} \boldsymbol{h}_M^{\boldsymbol{\alpha}} |\boldsymbol{D}^{\boldsymbol{\alpha}} (q-u)|_{W_{1,p}(\omega_M)}.
	\end{gather}
	The first summand is estimated as follows:
	\begin{gather}\label{eq:tapprox1}
		|u-q|_{W_{1,p}(M)} \le C \sum_{|\boldsymbol{\alpha}| \le 2} \boldsymbol{h}_M^{\boldsymbol{\alpha}} |\boldsymbol{D}^{\boldsymbol{\alpha}} (u-q)|_{W_{1,p}(M)},
	\end{gather}
	which can be proven to hold true by setting $v\coloneqq \boldsymbol{D}^{\boldsymbol{\gamma}} (u-q)$ with $|\boldsymbol{\gamma}|=1$ in
	\begin{gather*}
		\|v\|_{L_p(M)} \le C \sum_{|\boldsymbol{\alpha}| \le 2} \boldsymbol{h}_M^{\boldsymbol{\alpha}} \|\boldsymbol{D}^{\boldsymbol{\alpha}} v\|_{L_p(M)}.
	\end{gather*}
	This is in turn the embedding $W_{2,p}(M) \hookrightarrow L_p(M)$ on the reference macro-element and appropriate scaling. Collecting \eqref{eq:tapproxtriangle}, \eqref{eq:tapprox2} and \eqref{eq:tapprox1} we arrive at
	\begin{gather*}
		|u- \tilde \Pi u|_{W_{1,p}(M)} \le C C_{M,p} \sum_{|\boldsymbol{\alpha}|\le 2} \boldsymbol{h}_M^{\boldsymbol{\alpha}} |\boldsymbol{D}^{\boldsymbol{\alpha}} (u-q)|_{W_{1,p}(\omega_M)} \le C C_{M,p} \sum_{|\boldsymbol{\alpha}| = 2} \boldsymbol{h}_M^{\boldsymbol{\alpha}} |\boldsymbol{D}^{\boldsymbol{\alpha}} u|_{W_{1,p}(\omega_M)},
	\end{gather*}
	due to the special choice of $q$ and Lemma \ref{lem:lemmapoly}.
\end{proof}

\begin{rem}\label{rem:ass1}
	The absence of abrupt changes in the mesh sizes leads not only to Assumption \ref{assump:ass1} always being satisfied but also to $C_{M,p} \le C$ in \eqref{eq:tapprox}, similar to the results in \cite{A99}. If on the contrary there are abrupt changes in the mesh sizes of arbitrary magnitude then \eqref{eq:tapprox} can become useless for $p < \infty$ --- an observation that was made in \cite{AR08}, as well.
\end{rem}

\begin{rem}\label{rem:tapprox}
	Inspecting the proofs of Lemma \ref{lem:stab} and Theorem \ref{thm:tapprox} one sees that under the same assumptions the approximation error estimate
	\begin{gather}\label{eq:lperr}
		\|u- \tilde \Pi u\|_{L_{p}(M)} \le C C_{M,p} \sum_{|\boldsymbol{\alpha}| = 3} \boldsymbol{h}_M^{\boldsymbol{\alpha}} \|\boldsymbol{D}^{\boldsymbol{\alpha}} u\|_{L_{p}(\omega_M)}
	\end{gather}
	holds true for $p>1$. In fact, the stability estimate
	\begin{align*}
		\| \Pi^r u\|_{L_p(M)} \le C C_{M,p} \sum_{|\boldsymbol{\alpha}|\le 3} \boldsymbol{h}_M^{\boldsymbol{\alpha}} \|\boldsymbol{D}^{\boldsymbol{\alpha}} u\|_{L_p(M)}
	 \end{align*}
	 can be established based on the embedding $W_{3,p}(\Lambda) \hookrightarrow C^1(\Lambda)$ (which holds true for $p\ge 2$ in two dimensions) on the reference macro-element and a scaling argument. Moreover, one can make use of $\|\psi_{i,j}\|_{L_\infty(M)} \le C \boldsymbol{h}_M^{(1,1)}$ for $(i,j) \in I_M$.
	 If one only has $u \in W_{2,\infty}(\omega_M)$ one can still obtain
	 \begin{gather*}
		\|u - \tilde \Pi u\|_{L_\infty(M)} \le C \sum_{|\boldsymbol{\alpha}| = 2} \boldsymbol{h}_M^{\boldsymbol{\alpha}} \|\boldsymbol{D}^{\boldsymbol{\alpha}} u\|_{L_{\infty}(\omega_M)}
	 \end{gather*}
	 by estimating \eqref{eq:aijfinal} directly.
\end{rem}

\begin{rem}
	Similarly to the situation in which the interpolation operator is defined by local functionals it is again important that \emph{polynomials are reproduced on larger entities}. While we demanded this property for macro-elements in the local setting we need it now on patches of macro-elements. This seems to be an underlaying principle.
\end{rem}

We now turn our attention to second order derivatives. Inspecting the arguments in Theorem \ref{thm:tapprox} for the possibility to prove $L_p$-bounds for second order derivatives of the approximation error, we see that stability of $\tilde \Pi$ is crucial.

It is possible to prove
\begin{gather*}
	\big\|\big(\tilde \Pi u\big)_{xy}\big\|_{L_p(M)} \le C C_{M,p} \sum_{|\boldsymbol{\alpha}| \le 1} \boldsymbol{h}_M^{\boldsymbol{\alpha}} \left\|\boldsymbol{D}^{\boldsymbol{\alpha}} u_{xy}\right\|_{L_p(\omega_M)}.
\end{gather*}
However, it is unclear how to obtain a similar estimate for the other second order derivatives. 
We therefore restrict the subsequent study to the case of an \emph{isotropic macro-element patch} $\omega_M$. These results will be useful in Section \ref{sec:Shishkin_macro}. There we want to apply $\tilde \Pi$ in the fine regions of a Shishkin mesh close to the corners of the domain where the mesh is uniform.

\begin{assump}\label{assump:ass2}
	Let $M\in \mathcal{M}_{\boldsymbol{h}}$ denote a macro-element such that the restriction of $\mathcal{M}_{\boldsymbol{h}}$ to the associated macro-element patch $\omega_M$ is locally uniform with mesh size $h_M$.
\end{assump}

\begin{thm}\label{thm:tapproxunif}
	Based on Assumption \ref{assump:ass2} the quasi-interpolation operator $\tilde \Pi$ satisfies the approximation error estimate
	\begin{gather}\label{eq:tildePiuniformthm}
		\big| u-\tilde \Pi u \big|_{W_{k,p}(M)} \le C h_M^{3-k} |u|_{W_{3,p}(\omega_M)},
	\end{gather}
	for $u \in W_{3,p}(\omega_M)\cap C^1(M)$ with $p \in [1,\infty]$ and $k \le 2$.
\end{thm}
\begin{proof}
Under Assumption \ref{assump:ass2} the estimates \eqref{eq:tapprox} and \eqref{eq:lperr} simplify to \eqref{eq:tildePiuniformthm} for $k \le 1$ and it remains to validate this estimate for $k=2$.

Let $v\in C^1(M)$ with $v_{xy}|_{\sigma_{ij}} \in L_1(\sigma_{i,j})$ for all $(i,j) \in I_M$ so that $\Pi v$ is well defined. By Assumption \ref{assump:ass2} and the fact that $\Pi v$ is piecewise biquadratic an inverse estimate yields
\begin{gather}\label{eq:tildepiinverse}
	\|\tilde \Pi v\|_{W_{2,p}(M)} \le C h_M^{-1} \|\tilde \Pi v\|_{W_{1,p}(M)}.
\end{gather}

We proceed as in Theorem \ref{thm:tapprox}. By Lemma \ref{lem:lemmapoly} with $\ell = 3$ there exits a unique polynomial $q \in P_2(\omega_M)$ such that
\begin{subequations}\label{eq:polyprob}
\begin{align}
	\sum_{k=0}^3 h_M^k \left| u - q \right|_{W_{k,p}(\omega_M)} &\le C h_M^3 | u |_{W_{3,p}(\omega_M)},\label{eq:polyprob0}\\
    \sum_{k=0}^2 h_M^k \left| u - q \right|_{W_{k+1,p}(\omega_M)} &\le C h_M^2 | u |_{W_{3,p}(\omega_M)}.\label{eq:polyprob1}
\end{align}
\end{subequations}
A triangle inequality implies
\begin{gather}\label{eq:approx2triangle}
	\big| \tilde \Pi u - u \big|_{W_{2,p}(M)} \le | u - q |_{W_{2,p}(M)} + \big| \tilde \Pi (q - u) \big|_{W_{2,p}(M)}.
\end{gather}
The first summand is easily bounded by \eqref{eq:polyprob0}. For the other one we use the inverse estimate \eqref{eq:tildepiinverse}, the stability estimates for low order derivatives of $\tilde \Pi$, see Lemma \ref{lem:stab} and Remark \ref{rem:tapprox}, and \eqref{eq:polyprob}:
\begin{align}\label{eq:approx2discrete}
	\big| \tilde \Pi (q - u) \big|_{W_{2,p}(M)} &\le C h_M^{-1} \big\| \tilde \Pi (q - u) \big\|_{W_{1,p}(M)} \le C h_M^{-1} \Big( \big| \tilde \Pi (q - u) \big|_{W_{1,p}(M)} + \big\| \tilde \Pi (q - u) \big\|_{L_{p}(M)} \Big)\notag\\
	& \le C h_M^{-1} \bigg( \sum_{k=0}^2 h_M^k |q-u|_{W_{k+1,p}(\omega_M)} + \sum_{k=0}^3 h_M^k |q-u|_{W_{k,p}(\omega_M)} \bigg)\\
	& \le C h_M | u |_{W_{3,p}(\omega_M)}. \notag
\end{align}
Collecting \eqref{eq:approx2triangle}, \eqref{eq:polyprob0} and \eqref{eq:approx2discrete} the result follows.
\end{proof}

\begin{rem}
For $p < \infty$ the constant $C_{M,p}$ in the estimates \eqref{eq:tapprox} and \eqref{eq:lperr} renders them useless on meshes of Shishkin type or any other mesh with abrupt changes in the mesh sizes. In this case $L_\infty$ estimates are desirable. For second order derivatives we were able to prove a result of classical type with Theorem \ref{thm:tapproxunif}. In order to prove anisotropic error estimates it might be necessary to specify additional rules for the choice of the macro-element edges $\sigma_{i,j}$ associated with the macro-element vertices $\boldsymbol{X}_{ij}$, $(i,j) \in I_M$.
Moreover, Theorem \ref{thm:tapprox} shows two things:
\begin{itemize}
\item Firstly, it is possible to design useful quasi-interpolation operators that are defined by a mix of local and non-local functionals. This is particularly true if the element considered is not of Lagrange type. Extending this idea one might use different entities $\sigma_{i,j}$ for every component of a quasi-interpolation operator.
\item Secondly, by using non-local functionals only for the coefficients of basis functions associated with higher order derivatives the resulting quasi-interpolation operators of Scott-Zhang type seem to be very flexible with respect to the choice of the entities $\sigma_{i,j}$. Note that in \cite{A99} derivatives of adaptations of the Scott-Zhang operator were only proven to obey anisotropic interpolation error estimates if the entities $\sigma_{i,j}$ were chosen all parallel.
\end{itemize}
\end{rem}

\subsection[Summary]{\texorpdfstring{Summary: anisotropic $C^1$ (quasi-)interpolation error estimates}{Summary: anisotropic C1 (quasi-)interpolation error estimates}}\label{subsec:summary_macro}

In this Section we want to summarize our results and those of \cite{CYM2009}. To the knowledge of the author these are the only sources of anisotropic (quasi-)interpolation error estimates for $C^1$ Hermite(-type) interpolation.
All estimates are valid on rectangular tensor product meshes such that the edges of an element $K$ are aligned with the coordinate axes.
In all estimates $C$ is a generic constant that does not depend on $u$ or the mesh.

The work \cite{CYM2009} addresses for $N\ge 1$ two $C^{N-1}$ Hermite interpolation operators $I_{12}$ and $I_{22}$ into the piecewise $Q_{2N-1}$ and $Q_{2N}$ functions, respectively. Its main results are the anisotropic error estimates
\begin{align*}
	|u-I_{12} u|_{N,K} &\le C \sum_{|\boldsymbol{\beta}|=N} \boldsymbol{h}_K^{\boldsymbol{\beta}} |\boldsymbol{D}^{\boldsymbol{\beta}} u|_{N,K},\\
	|u-I_{22} u|_{N,K} &\le C \sum_{|\boldsymbol{\beta}|=N+1} \boldsymbol{h}_K^{\boldsymbol{\beta}} |\boldsymbol{D}^{\boldsymbol{\beta}} u|_{N,K},
\end{align*}
for $u \in H^{2N}(K)$ and with $\boldsymbol{h}_K = (h_{1,K},h_{2,K})$ where $h_{i,K}$ is the size of $K$ in $x_i$-direction.

Inspecting their proofs for $N=2$ we see that there is a $C^1$ Hermite interpolation operator $I_{12}$ into the piecewise bicubic functions (more precisely the Bogner-Fox-Schmidt element space) such that
\begin{gather*}
	\left \| \boldsymbol{D}^{\boldsymbol{\gamma}} \left(  u - I_{12} u\right) \right \|_{0, K} \le C \sum_{|\boldsymbol{\alpha} |=4-|\boldsymbol{\gamma}|} \boldsymbol{h}_K^{\boldsymbol{\alpha}} \left\| \boldsymbol{D}^{\boldsymbol{\alpha}} \boldsymbol{D}^{\boldsymbol{\gamma}} u \right\|_{0,K},
\end{gather*}
for $|\boldsymbol{\gamma} |\le 2$ and $u \in H^4(K)$. We want to emphasize that this result originally obtained by \cite{CYM2009} can alternatively be proven using Apel's theory and our key observation that two dimensional divided differences may be used as associated functionals (cf.~Corollary \ref{cor:bicubic}). 

We refer to \cite{CYM2009} for a note on the three dimensional case.

In the case of piecewise biquadratic functions we extended the results of \cite{HHZ2011} to the anisotropic case using new results on macro-interpolation. If the mesh can be generated as a uniform refinement of a macro-element mesh $\mathcal{M}_{\boldsymbol{h}}$, then there is a $C^1$ Hermite interpolation operator $\Pi$ into the piecewise biquadratic functions such that (cf.~Corollary \ref{cor:respi})
	\begin{gather*}
		\|\boldsymbol{D}^{\boldsymbol{\gamma}} \left( u - \Pi u \right)\|_{0,M} \le C \left( \sum_{|\boldsymbol{\alpha}| = 4-|\boldsymbol{\gamma}|} \!\boldsymbol{h}_M^{\boldsymbol{\alpha}} \left| \boldsymbol{D}^{\boldsymbol{\alpha}} \boldsymbol{D}^{\boldsymbol{\gamma}} u \right|_{0,M} + \sum_{|\boldsymbol{\alpha}|=3-|\boldsymbol{\gamma}|} \!\boldsymbol{h}_M^{\boldsymbol{\alpha}} \left| \int_{M} \boldsymbol{D}^{\boldsymbol{\alpha}} \boldsymbol{D}^{\boldsymbol{\gamma}} u(x,y) \,\D x\D y \right| \right)
	\end{gather*}
on a macro-element $M \in \mathcal{M}_{\boldsymbol{h}}$ for a multi-index $\boldsymbol{\gamma}$ with $|\boldsymbol{\gamma}| \le 2$ and $u \in C^2(M)$ such that $\boldsymbol{D}^{\boldsymbol{\gamma}} u \in H^{4-|\boldsymbol{\gamma}|}(M)$.

In order to reduce the regularity required we use non-local information of the interpolant in order to define the coefficient of the basis function associated with the mixed second derivative, creating the quasi-interpolation operator $\tilde \Pi$.
For its analysis we need Assumption \ref{assump:ass1} to be satisfied.
Collecting the results of Theorem \ref{thm:tapprox}, Remark \ref{rem:tapprox} we summarize that for $p \in [2,\infty]$ and $m=0,1$ the error estimate
	\begin{gather}\label{eq:summtildepi}
		|u- \tilde \Pi u|_{W_{m,p}(M)} \le C C_{M,p} \sum_{|\boldsymbol{\alpha}| = 3-m} \boldsymbol{h}_M^{\boldsymbol{\alpha}} |\boldsymbol{D}^{\boldsymbol{\alpha}} u|_{W_{m,p}(\omega_M)}
	\end{gather}
holds true, provided $u \in W_{3,p}(\omega_M)$. Here
	\begin{gather*}
		C_{M,p} \coloneqq \left(\frac{\meas M}{\min_{T \in \mathcal{M},\,T \subset \omega_M} \meas T}\right)^{1/p} \ge 1
	\end{gather*}
	and $\omega_M$ is the associated macro-element patch $\omega_M$ around $M$, cf.~Definition \ref{defn:rm}.

In the case of a more regular mesh (more precisely: under Assumption \ref{assump:ass2}) the operator $\tilde \Pi$ satisfies error estimates of classical type even for second order derivatives, see Theorem \ref{thm:tapproxunif}. Note that the absence of abrupt changes in the mesh sizes implies the validity of Assumption \ref{assump:ass1} and a simplification of the estimates \eqref{eq:summtildepi} due to $C_{M,p} \le C$, cf.~Remark \ref{rem:ass1}.

It would be very interesting to check numerically if there is hope for the Girault-Scott operator of \cite[Section 4]{HHZ2011} to allow anisotropic interpolation error estimates given only some $W_{2,p}$ regularity of the function to be approximated. However, certain details in that paper are unclear --- especially the scaling of the true dual basis functions (given only as a brief note) is questionable.

\section{An anisotropic macro-element of tensor product type}\label{sec:anisomacro_macro}

In Section \ref{sec:univariate} we have seen 1D Hermite interpolation in the space of quadratic $C^1$ splines. The tensor product of this 1D macro-element with itself created a 2D macro-element and the induced interpolation operator $\Pi$ for which we were able to prove certain anisotropic interpolation error estimates. However, the usage of this operator on for instance a Shishkin mesh (where the direction of anisotropy and mesh sizes changes abruptly) does not lead to optimal results. The main reason for this failure is that the $C^1$ operators $\Pi$ or $\tilde \Pi$ do not satisfy certain $L_\infty$-stability estimates. Based on the usage of derivatives one has for instance on some macro-element $M \in \mathcal{M}_{\boldsymbol{h}}$ with sizes $\boldsymbol{h}_M$ that
\begin{gather*}
	\|\Pi v\|_{L_\infty(M)} \le C \left(\sum_{|\boldsymbol{\alpha}| \le 1} \boldsymbol{h}_M^{\boldsymbol{\alpha}} \| \boldsymbol{D}^{\boldsymbol{\alpha}} v\|_{L_\infty(M)} + \boldsymbol{h}_M^{(1,1)} \| \boldsymbol{D}^{(1,1)} v\|_{L_\infty(M)}\right),
\end{gather*}
holds true, i.e.~$L_\infty$ norms of derivatives appear on the right hand side. Hence, if one wants to bound the error in the interior with large elements one can no longer use that the interpolant is small there but has to demand that also derivatives of the interpolant are small. This is however not true on a Shishkin mesh as already mentioned in the introduction. In order to remedy this problem we consider the following anisotropic macro-element.

\begin{figure}
	\centering
	\includegraphics[height=5.5cm]{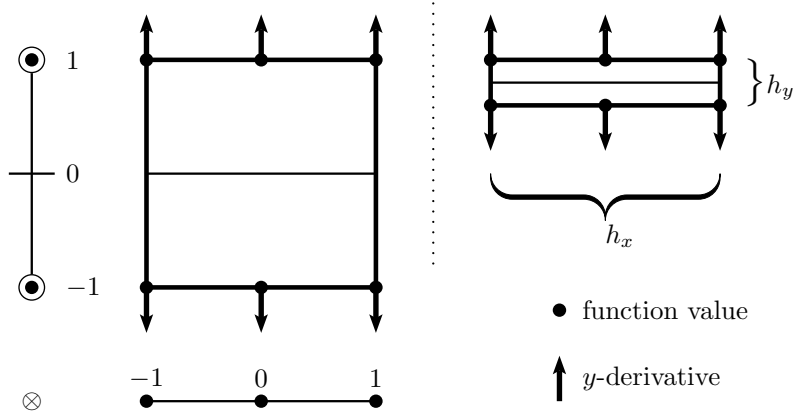}
	\caption{Degrees of freedom of the anisotropic macro-element on the reference macro-element $\hat M$ (left) and on some anisotropic macro in the world domain (right).}
	\label{fig:aniso-macro-element}
\end{figure}

We form a macro of two rectangles and use as degrees of freedom the function value and the value of a certain first derivate in six points along the boundary of the macro (cf.~Figure \ref{fig:aniso-macro-element}). Note that this macro-element can be considered as the tensor product of one dimensional $C^1-P_2$ macro-interpolation and $P_2$ Lagrange interpolation. Hence, we \emph{leave the realm of $C^1$ macro-elements} but preserve the property of a continuous normal derivative across some macro-element edges. This will be vital in the next section.

More precisely, assuming that, as illustrated in Figure \ref{fig:aniso-macro-element}, the reference macro-element $\hat M\coloneqq\{[-1,1]\times[-1,0],[-1,1]\times[0,1]\}$ over the reference domain $\Lambda = [-1,1]^2$ is mapped to an anisotropic one for which the aspect ratio $h_x/h_y$ is very large we use quadratic $C^1$ splines in $y$ direction (small side) and $P_2$ in $x$ direction (large side). This space $S(\hat M)$ is 12 dimensional and from \eqref{eq:newtonSpline} and
\begin{gather*}
	p(x) = p[-1] + p[-1,0](x+1) + p[-1,0,1] (x+1) x\quad \forall p \in P_2([-1,1]),
\end{gather*}
we can obtain the representation
\begin{gather}
\begin{aligned}\label{eq:sanisomacro}
	s(x,y) &= \sum_{j=1}^3 \left(F_{1j}(s)(y+1)^{j-1} + F_{2j}(s) (x+1)(y+1)^{j-1} + F_{3j}(s) (x+1)x(y+1)^{j-1}\right)\\
	&\quad+ 4 \big( F_{14}(s) + F_{24}(s) (x+1) + F_{34}(s) (x+1)x \big) \hat \psi_1(y)\quad \forall s \in S(\hat M).
\end{aligned}
\end{gather}
By $\Pi^x$ we denote the macro-element interpolation operator such that the roles of the sizes $h_x$ and $h_y$ of a macro-element $M$ are interchanged, i.e.~$h_x \gg h_y$.

The functionals $F_{ij}$ are again defined as two dimensional divided differences:
\begin{gather*}
	F_{ij}(s) \coloneqq s[m_i;n_j]\quad \text{with}\quad m_i = \left \{
	\begin{alignedat}{2}
	&{-1}&\quad &\text{for $i=1$,}\\
	&{-1},{0}&\quad &\text{for $i=2$,}\\
	&{-1},{0},1&\quad &\text{for $i=3$,}\\
	\end{alignedat}\right.
	\quad \text{and} \quad n_j = \left \{
	\begin{alignedat}{2}
	&{-1}&\quad &\text{for $j=1$,}\\
	&{-1},{-1}&\quad &\text{for $j=2$,}\\
	&{-1},{-1},1&\quad &\text{for $j=3$,}\\
	&{-1},{-1},1,1&\quad &\text{for $j=4$.}\\
	\end{alignedat}\right.
\end{gather*}

It is easy to establish the $H^1$-conformity of this macro-element. Moreover, we find that the $y$-derivative along the edge $y=\pm 1$ of $\Lambda$ can be expressed by
\begin{align*}
	\frac{\partial s}{\partial y}(x,\pm 1) &= \frac{\partial s}{\partial y}(0,\pm1) + \frac{1}{2}\left( \frac{\partial s}{\partial y}(1,\pm1) - \frac{\partial s}{\partial y}(-1,\pm1) \right) x \\
	&\quad + \frac{1}{2}\left( \frac{\partial s}{\partial y}(-1,\pm1) - 2 \frac{\partial s}{\partial y}(0,\pm1) + \frac{\partial s}{\partial y}(1,\pm1) \right) x^2.
\end{align*}
Hence, if two such macro-elements are combined in $y$-direction the normal derivative along the common edge parallel to the $x$-axis (long side) is continuous. Clearly, this macro-element induces another interpolation operator $\hat\Pi^y : C^1(\Lambda) \rightarrow S(\hat M)$:
\begin{gather}\label{eq:represpiy}
	\hat\Pi^y u(x,y) \coloneqq \sum_{i\in\{-1,0,1\}} \sum_{j\in\{-1,1\}} \left( u(i,j) \hat \ell_i(x) \hat \varphi_j(y) + \frac{\partial u}{\partial y}(i,j) \hat \ell_i(x) \hat \psi_j(y) \right).
\end{gather}
Here $\hat \ell_i \in P_2[-1,1]$ denotes the quadratic Lagrange basis function that corresponds to the node $i \in \{-1,0,1\}$, i.e.
\begin{gather*}
	\hat\ell_{-1} \coloneqq x(x-1)/2,\quad \hat\ell_{0} \coloneqq -(x+1)(x-1),\quad \hat\ell_{+1} \coloneqq (x+1)x/2.
\end{gather*}

Let $M=[x_0-h_x/2, x_0+h_x/2]\times[y_0-h_y/2, y_0+h_y/2$ denote a macro-element.
From the representation \eqref{eq:represpiy} and the affine reference mapping $F_M:[-1,1] \to M$:
\begin{gather}\label{eq:affinmacro}
	x = x_0 + h_x \hat x,\qquad y = y_0 + h_y \hat y,
\end{gather}
it is easy to deduce for the interpolation operator $\Pi^y u \coloneqq \hat \Pi^y \hat u \circ F_M^{-1}$ with $\hat u \coloneqq u \circ F_M$ on the macro-element $M$ in the world domain the stability property
\begin{align}\label{eq:Piainfstab}
	\| \Pi^y u \|_{L_\infty(M)} \le C \left( \| u \|_{L_\infty(M)} + h_y \left\| \frac{\partial u}{\partial y} \right\|_{L_\infty(M)} \right).
\end{align}

\begin{rem}
	Note that by construction $h_y$ is the length of the small side of the macro-element $M$ in the world domain. Hence, the first derivative in \eqref{eq:Piainfstab} is combined with a small multiplier.
\end{rem}

Next we study the approximation properties of this interpolation operator.

\begin{thm}
	For $u \in H^3(\Lambda)$ and a multi-index $\boldsymbol{\gamma}$ with $|\boldsymbol{\gamma}|\le 2$ we have the estimates
	\begin{subequations}
	\begin{align}
		\| \boldsymbol{D}^{\boldsymbol{\gamma}} (u - \Pi^y u) \|_0 &\le C \left| \boldsymbol{D}^{\boldsymbol{\gamma}} u \right|_{3-|\boldsymbol{\gamma}|} \qquad \text{for $\boldsymbol{\gamma} \neq (2,0)$} \label{eq:Pianxx}\\
		\| (u - \Pi^y u)_{xx} \|_0 &\le C \big( |u_{xx}|_1 + |u_x|_2 \big) \label{eq:Piaxx}	
	\end{align}	
	\end{subequations}
\end{thm}
\begin{proof}
	We shall apply Lemma \ref{lem:lemma3} in order to prove \eqref{eq:Pianxx}. Thus, we set $\boldsymbol{P} \coloneqq \boldsymbol{Q} \coloneqq P_{2-|\boldsymbol{\gamma}|}$. By a direct calculation similarly to \eqref{eq:invarprop} we observe that the additional error component involving the polynomial $q \in P_2(\Lambda)$ vanishes since $\Pi^y v = v$ holds true for any function $v \in Q_2(\Lambda) \supset P_2(\Lambda)$. It remains to specify the associate functionals $F_{ij}^{\boldsymbol{\gamma}}$ according to \eqref{eq:functionals} for a given differential operator $\boldsymbol{D}^{\boldsymbol{\gamma}}$ with $|\boldsymbol{\gamma}|\le 2$. We use the same techniques as in Theorem \ref{thm:approxPi}. Firstly, it can be seen by applying the differential operator $\boldsymbol{D}^{\boldsymbol{\gamma}}$ to the representation \eqref{eq:sanisomacro} of an element $s \in S(\hat M)$ that $\boldsymbol{D}^{\boldsymbol{\gamma}} S(\hat M)$ can be normed by
	\begin{gather*}
		\sum_{(i,j) \in J_{\boldsymbol{\gamma}}} |F_{ij}(\cdot)|\qquad \text{with}\quad  J_{\boldsymbol{\gamma}}\coloneqq \{(i,j)\,:\, i=\gamma_1+1,\dots,3,\; j=\gamma_2+1,\dots,4\}.
	\end{gather*}
	Clearly, $F_{ij}(u) = F_{ij}(\Pi^y u)$ for all $i \in \{1,2,3\}$ and $j \in \{1,2,3,4\}$ because the divided differences are linear combinations of the interpolation data $\{u(k,\ell),u_y(k,\ell)\}_{k\in \{-1,0,1\},\ell \in \{-1,1\}}$.
	The associated functionals $F_{ij}^{\boldsymbol{\gamma}}$ for $(i,j) \in J_{\boldsymbol{\gamma}}$ are listed in Table \ref{tab:assfuncPia}. Using Sobolev embeddings like in the proof of Theorem \ref{thm:approxPi} it is easy to check that $F_{ij}^{\boldsymbol{\gamma}} \in \big(H_{3-|\boldsymbol{\gamma}|}(\Lambda)\big)'$. Moreover, 
	\begin{gather*}
		F_{ij}(u) = F_{ij}^{\boldsymbol{\gamma}} (\boldsymbol{D}^{\boldsymbol{\gamma}} u)\quad \text{and}\quad F_{ij}(\Pi^y u) = F_{ij}^{\boldsymbol{\gamma}} (\boldsymbol{D}^{\boldsymbol{\gamma}} \Pi^y u)
	\end{gather*}	
	for $(i,j) \in J_{\boldsymbol{\gamma}}$. The first identity follows from the techniques in the proof of Theorem \ref{thm:approxPi}, especially \eqref{eq:peanoform}. A simple computation for each basis function in $S(\hat M)$ shows the second identity, due to the linearity of $F_{ij}$ and $F_{ij}^{\boldsymbol{\gamma}}$. Hence, indeed $F_{ij}^{\boldsymbol{\gamma}} (\boldsymbol{D}^{\boldsymbol{\gamma}} \Pi^y u) = F_{ij}^{\boldsymbol{\gamma}} (\boldsymbol{D}^{\boldsymbol{\gamma}} u)$. 
We shall demonstrate this procedure for $F_{33}$. A calculation gives
\begin{align}
	F_{33}(u) &= u[-1,0,1;-1,-1,1] = \frac{1}{2} \big( u(-1,\cdot)[-1,-1,1] - 2 u(0,\cdot)[-1,-1,1] + u(1,\cdot)[-1,-1,1]\big)\notag\\
	&= \frac{1}{2} \int_{-1}^1 s_1(y) \big( u_{yy}(-1,y) - 2 u_{yy}(0,y) + u_{yy}(1,y)\big)\,\D y, \label{eq:F33}
\end{align}
where we used \eqref{eq:peanoform} with $s_1(y) = (1-y)/4$ and from which $F_{33}^{(0,1)}$ and $F_{33}^{(0,2)}$ can be deduced. Moreover, we may rewrite this identity to obtain
\begin{gather*}
	F_{33}(u) = \frac{1}{2} \int_{-1}^1 s_1(y) \left( \int_0^1 u_{xyy}(x,y) \,\D x - \int_{-1}^0 u_{xyy}(x,y) \,\D x\right)\,\D y.
\end{gather*}
A reinterpretation of this equation according to $F_{33}(u) = F_{33}^{\boldsymbol{\gamma}} (\boldsymbol{D}^{\boldsymbol{\gamma}} u)$ gives $F_{33}^{(1,0)}$ and $F_{33}^{(1,1)}$. A computation shows $F_{33} (s) = F_{33}^{\boldsymbol{\gamma}} (\boldsymbol{D}^{\boldsymbol{\gamma}} s)$ for all $s \in S(\hat M)$ and $|\boldsymbol{\gamma}| \le 2$, $\boldsymbol{\gamma} \neq (2,0)$. Hence, the estimate \eqref{eq:Pianxx} is proven.

For $\boldsymbol{\gamma} = (2,0)$ it appears impossible to provide the associated functionals by the above technique. Consider for instance the divided difference $F_{33}$.  Using Taylor expansion it is possible to rewrite the equation \eqref{eq:F33} to
\begin{gather*}
	F_{33}(u) = \frac{1}{2} \int_{-1}^1 s_1(y) \left( \int_{-1}^0 (1+x) u_{xxyy}(x,y)\,\D x + \int_{0}^1 (1-x) u_{xxyy}(x,y)\,\D x\right)\,\D y.
\end{gather*}
This however comes at the price of demanding higher regularity. Clearly, we have to approach this problem differently.
Let $\boldsymbol{P}\coloneqq\{(2,0),(1,1),(1,0)\}$ and $q \in \boldsymbol{P}(\Lambda)$ denote the polynomial with
\begin{gather*}
	\int_{\Lambda} \boldsymbol{D}^{\boldsymbol{\alpha}} (u-q)\,\D \boldsymbol{x} = 0\qquad \forall \boldsymbol{\alpha} \in \boldsymbol{P}.
\end{gather*}
By Lemma \ref{lem:lemma1} the polynomial $q$ exits and is unique. From Lemma \ref{lem:lemma2} we can deduce by setting $v \coloneqq (u-q)_{xx}$ that
\begin{gather}\label{eq:dxxh1part}
	\|v\|_1 = \|(u-q)_{xx}\|_1 \le C |(u-q)_{xx}|_1 = C |u_{xx}|_1,
\end{gather} 
since $\int_{\Lambda} v \,\D \boldsymbol{x} = \int_{\Lambda} (u-q)_{xx} \,\D \boldsymbol{x} = 0$. 
Similarly, Lemma \ref{lem:lemma2} implies that for $v \coloneqq (u-q)_x$ we find
\begin{gather}\label{eq:dxh2part}
	\|v\|_2 = \|(u-q)_x\|_2 \le C |(u-q)_x|_2 = C |u_x|_2,
\end{gather} 
based on 
\begin{gather*}
	\int_{\Lambda} \boldsymbol{D}^{\boldsymbol{\alpha}} v \,\D \boldsymbol{x}= 0\qquad \forall\, |\boldsymbol{\alpha}| \le 1
	\quad \Leftrightarrow \quad \int_{\Lambda} \boldsymbol{D}^{\boldsymbol{\alpha}} (u-q)_x \,\D x= 0\qquad \forall\, |\boldsymbol{\alpha}| \le 1.
\end{gather*}
Next from \eqref{eq:Pianxx} for $\boldsymbol{\gamma} = (1,0)$ we obtain the following stability estimate
\begin{gather}\label{eq:piastab}
	\big\|\big(\Pi^y v\big)_x\big\|_0 \le \|v_x\|_0 + \big\|\big(v - \Pi^y v\big)_x\big\|_0 \le C \|v_x\|_2.
\end{gather}
A triangle inequality implies due to $q = \Pi^y q$ that
\begin{gather}\label{eq:piaxxtri}
	\big\|\big(u-\Pi^y u\big)_{xx}\big\|_0 \le \|(u-q)_{xx}\|_0 + \big\|\big(\Pi^y(q-u)\big)_{xx}\big\|_0.
\end{gather}
The first summand on the right hand side of \eqref{eq:piaxxtri} is estimated using \eqref{eq:dxxh1part}, while for the other one we use the inverse estimate
\begin{gather*}
	\|s_{xx}\|_0 \le C \|s_x\|_0\quad \forall s \in S,
\end{gather*}
which is easily verified in the four dimensional space $\boldsymbol{D}^{(2,0)} S(\hat M)$ over the reference macro-element. In fact, the optimal constant in this estimate is given by $C = \sqrt{3}$. Hence, by \eqref{eq:piaxxtri},
\begin{gather*}
	\big\|\big(u-\Pi^y u\big)_{xx}\big\|_0 \le C \big( |u_{xx}|_1 + \big\| \big(\Pi^y(q-u)\big)_x\big\|_0 \big).
\end{gather*}
We finish the proof of \eqref{eq:Piaxx} by using \eqref{eq:piastab} for $v=q-u$ and \eqref{eq:dxh2part}.
\end{proof}

\begin{table}\label{tab:assfuncPia}
	\centering
	{ 	
	\begin{tabular}{ccl}
		\toprule
		$\boldsymbol{\gamma}$ & $\dim \boldsymbol{D}^{\boldsymbol{\gamma}} S(\hat M)$ & associate functionals\\
		\midrule
		   $(0,0)$ & 12 & $F_{ij}^{(0,0)} \coloneqq F_{ij} \quad  i \in \{1,2,3\} \text{ and } j \in \{1,2,3,4\}$\\
		\midrule
		   $(1,0)$ & 8  & 
		   $\begin{aligned}
		   F_{2\ell}^{(1,0)}(v) &\coloneqq \int_{-1}^0 \frac{\partial^{\ell-1} v}{\partial y^{\ell-1}}(x,-1)\,\D x,\quad \ell\in\{1,2\}\\
		   F_{2k}^{(1,0)}(v) &\coloneqq \int_{-1}^1 \int_{-1}^0 s_{k-2}(y) v_{yy}(x,y)\,\D x\D y,\quad k\in\{3,4\}\\
		   F_{3\ell}^{(1,0)}(v) &\coloneqq \frac{1}{2} \int_{0}^1 \frac{\partial^{\ell-1} v}{\partial y^{\ell-1}}(x,-1)\,\D x - \frac{1}{2} \int_{-1}^0 \frac{\partial^{\ell-1} v}{\partial y^{\ell-1}}(x,-1)\,\D x\\
		   F_{3k}^{(1,0)}(v) &\coloneqq \frac{1}{2} \int_{-1}^1 s_{k-2}(y) \left( \int_{0}^1 v_{yy}(x,y)\,\D x - \int_{-1}^0 v_{yy}(x,y)\,\D x\right)\D y\\
		   \end{aligned}$\\
		 \midrule
		   $(0,1)$ & 9  & 
		   $\begin{aligned}
		   F_{12}^{(0,1)}(v) &\coloneqq v(-1,-1)\\
		   F_{1k}^{(0,1)}(v) &\coloneqq \int_{-1}^1 s_{k-2}(y) v_y(-1,y) \,\D y,\quad k\in\{3,4\}\\
		   F_{22}^{(0,1)}(v) &\coloneqq v(0,-1) - v(-1,-1)\\
		   F_{2k}^{(0,1)}(v) &\coloneqq \int_{-1}^1 \int_{-1}^0 s_{k-2}(y) v_{xy}(x,y) \,\D x \D y,\quad k\in\{3,4\}\\
		   F_{32}^{(0,1)}(v) &\coloneqq \frac{1}{2} \big( v(-1,-1) -2 v(0,-1) + v(1,-1) \big)\\
		   F_{3k}^{(0,1)}(v) &\coloneqq \frac{1}{2} \int_{-1}^1 s_{k-2}(y) \big( v_{y}(-1,y) - 2 v_{y}(0,y) + v_{y}(1,y) \big)\,\D y\\
		   \end{aligned}$\\
		 \midrule
		   $(1,1)$ & 6  &
		   $\begin{aligned}
		   F_{22}^{(1,1)}(v) &\coloneqq \int_{-1}^0 v(x,-1)\,\D x\\
		   F_{2k}^{(1,1)}(v) &\coloneqq \int_{-1}^1 s_{k-2}(y) \int_{-1}^0 v_y(x,y)\,\D x\D y,\quad k\in\{3,4\}\\\
		   F_{32}^{(1,1)}(v) &\coloneqq \frac{1}{2} \int_{0}^1 v(x,-1)\,\D x - \frac{1}{2} \int_{-1}^0 v(x,-1)\,\D x\\
		   F_{3k}^{(1,1)}(v) &\coloneqq \frac{1}{2} \int_{-1}^1 s_{k-2}(y) \left(\int_{0}^1 v_y(x,y)\,\D x - \int_{-1}^0 v_y(x,y)\,\D x \right)\D y\\
		   \end{aligned}$\\
		 \midrule
		   $(0,2)$ & 6  &
		   $\begin{aligned}
		   F_{1k}^{(0,2)}(v) &\coloneqq \int_{-1}^1 s_{k-2}(y) v(-1,y)\,\D y,\quad k\in\{3,4\}\\
		   F_{2k}^{(0,2)}(v) &\coloneqq \int_{-1}^1 s_{k-2}(y) \int_{-1}^0 v_x(x,y)\,\D x\D y,\quad k\in\{3,4\}\\
		   F_{3k}^{(0,2)}(v) &\coloneqq \frac{1}{2} \int_{-1}^1 s_{k-2}(y) \big(v(-1,y) - 2 v(0,y) + v(1,y) \big)\D y\\
		   \end{aligned}$\\
		\bottomrule
	\end{tabular}
	}
	\caption{Associated functionals $F_{i,j}^{\boldsymbol{\gamma}}$ for the operator $\Pi^y$ over $\hat M$ with respect to $\boldsymbol{D}^{\boldsymbol{\gamma}}$ and $s_1(y)=(1-y)/4$, $s_2(y) = y/4$.}
\end{table}

Using affine equivalence (cf.~\eqref{eq:affinmacro} and the proof of Theorem \ref{thm:Pired}) we obtain on a macro-element $M$ in the world domain the following result.

\begin{cor}
	For $u \in H^3(M)$ and a multi-index $\boldsymbol{\gamma}$ with $|\boldsymbol{\gamma}|\le 2$ we have the estimates
	\begin{subequations}
	\begin{gather}
		\| \boldsymbol{D}^{\boldsymbol{\gamma}} (u - \Pi^y u) \|_{0,M} \le C \sum_{|\boldsymbol{\alpha}|=3-|\boldsymbol{\gamma}|} \boldsymbol{h}_M^{\boldsymbol{\alpha}} \left\| \boldsymbol{D}^{\boldsymbol{\alpha}+\boldsymbol{\gamma}} u \right\|_{0,M} \qquad \text{for $\boldsymbol{\gamma} \neq (2,0)$} \label{eq:PianxxM}\\
		\big\| \big(u - \Pi^y u\big)_{xx} \big\|_{0,M} \le C \bigg( \sum_{|\boldsymbol{\alpha}|=1} \boldsymbol{h}_M^{\boldsymbol{\alpha}} \|\boldsymbol{D}^{\boldsymbol{\alpha}} u_{xx}\|_{0,M} + \sum_{|\boldsymbol{\alpha}|=2} \frac{\boldsymbol{h}_M^{\boldsymbol{\alpha}}}{h_x} \|\boldsymbol{D}^{\boldsymbol{\alpha}} u_x\|_{0,M} \bigg) \label{eq:PiaxxM}	
	\end{gather}	
	\end{subequations}
\end{cor}

\begin{rem}
	By construction $h_x$ denotes the length of the long side of $M$. Hence, the estimate \eqref{eq:PiaxxM} is useful even in the anisotropic case.
\end{rem}

Before we end this section we prove a suboptimal but useful error estimate for $\boldsymbol{\gamma} = (0,0)$.
\begin{lem}\label{lem:anisoL2suboptimal}
	Let $u \in H^3(M)$ then
	\begin{gather*}
	\|u-\Pi^y u\|_{0,M} \le C \sum_{|\boldsymbol{\alpha}|=2} \big( \boldsymbol{h}_M^{\boldsymbol{\alpha}} \|\boldsymbol{D}^{\boldsymbol{\alpha}} u\|_{0,M} + \boldsymbol{h}_M^{\boldsymbol{\alpha}} h_y \|\boldsymbol{D}^{\boldsymbol{\alpha}} u_y \|_{0,M} \big).
	\end{gather*}
	\begin{proof}
		Let $q \in \boldsymbol{P}_1(\Lambda)$ denote the linear polynomial such that
		\begin{gather*}
			\int_{\Lambda} \boldsymbol{D}^{\boldsymbol{\alpha}} (u-q) \D \boldsymbol{x} =0\quad \forall \boldsymbol{\alpha} \in \boldsymbol{P}_1\coloneqq\{(0,0),(1,0),(0,1)\}.
		\end{gather*}
		Then from Lemma \ref{lem:lemma2} it follows that $\|u-q\|_2 \le C |u-q|_2 = C |u|_2$.
		Using this and $\Pi^y q = q$ we see that
		\begin{align*}
			\|u - \Pi^y u\|_0 &\le \|u-q\|_0 + \|\Pi^y (q-u)\|_0\\
			&\le \|u-q\|_2 + C \sum_{i \in \{-1,0,1\}} \sum_{j \in \{-1,1\}} \bigg( |\Pi^y(q-u) (i,j)| + \bigg| \frac{\partial \Pi^y(q-u)}{\partial y} (i,j) \bigg| \bigg)\\
			&\le |u|_2 + C \sum_{i \in \{-1,0,1\}} \sum_{j \in \{-1,1\}} \bigg( | (q-u) (i,j)| + \bigg| \frac{\partial (q-u)}{\partial y} (i,j) \bigg| \bigg)\\
			&\le |u|_2 + C \bigg( \|u-q\|_2 + \bigg\|\frac{\partial (q-u)}{\partial y}\bigg\|_2 \bigg).
		\end{align*}
		From 
		\begin{gather*}
			\|(q-u)_y\|_2 \le \|q-u\|_2 + |(q-u)_y|_2 \le C |u|_2 + |u_y|_2,
		\end{gather*}
		the estimate follows on the reference macro $\hat M$. The assertion of the lemma is again easily obtained by affine transformation.
	\end{proof}
\end{lem}

\section[Application on a Shishkin mesh]{Application of macro-element interpolation on a tensor product Shish\-kin mesh}\label{sec:Shishkin_macro}

As an application of the anisotropic quasi-interpolation error estimates obtained we want to examine the approximation error of the solution of a reaction-diffusion problem on an anisotropic mesh. Let $u$ denote the solution of the singularly perturbed linear reaction-diffusion problem
\begin{gather}\label{eq:reacdiff}
	-\eps \Delta u + c u = f \quad \text{in $\Omega$},\qquad u=0 \quad \text{on $\partial \Omega$},
\end{gather}
where $0 < \eps \ll 1$, $0 < 2 (c^\star)^2 \le c$ and $c$ and $f$ are smooth functions on some bounded two dimensional domain $\Omega$ with Lipschitz-continuous boundary $\partial \Omega$. 
We consider the unit square $\Omega \coloneqq (0,1)^2$ with the four edges
\begin{alignat*}{2}
	\Gamma_1 &= \{(x,0)\,:\, 0\le x\le 1\},&\qquad \Gamma_2 &= \{(0,y)\,:\, 0\le y\le 1\},\\
	\Gamma_3 &= \{(x,1)\,:\, 0\le x\le 1\},&\qquad \Gamma_4 &= \{(1,y)\,:\, 0\le y\le 1\}.
\end{alignat*}
In the corners of the domain $\Omega$ derivatives of $u$ are unbounded, in general. One refers to the solution components that cause this phenomenon as corner singularities. If we however assume the corner compatibility conditions
\begin{gather}\label{eq:cornercomp}
	f(0,0) = f(1,0) = f(0,1) = f(1,1) = 0,
\end{gather}
then third derivatives of $u$ are smooth up to the boundary, $u \in C^3(\overline \Omega)$, see, e.g.~\cite{HK90}.

The following solution decomposition is taken from \cite[Lemma 1.1 and Lemma 1.2]{LMSZ09}
\begin{lem}\label{lem:soldec}
	The solution $u \in C^3(\overline{\Omega})$ of \eqref{eq:reacdiff} can be decomposed as
	\begin{subequations}	
	\label{eq:soldecestim}		
	\begin{gather}\label{eq:soldec}
		u = S + \sum_{i=1}^4 E_i + E_{12} + E_{23} + E_{34} + E_{41}.
	\end{gather}
	Here $E_i$ is a boundary layer associated with the edge $\Gamma_i$. Similarly, $E_{ij}$ is a corner layer associated with the corner that is formed by the edges $\Gamma_i$ and $\Gamma_j$. Moreover, there are positive constants $C > 0$ such that for all $(x,y) \in \overline{\Omega}$ and $0\le i+j\le 3$ we have
	\begin{align}
		\left| \frac{\partial^{i+j} S(x,y)}{\partial x^i \partial y^j} \right| &\le C \big(1+\eps^{1-(i+j)/2}\big) \label{eq:soldecestimS}\\
		\left| \frac{\partial^{i+j} E_1(x,y)}{\partial x^i \partial y^j} \right| &\le C \big(1 + \eps^{1-i/2}\big) \eps^{-j/2} \E^{-c^\star y /\sqrt{\eps}}\label{eq:soldecestimE1}\\
		\left| \frac{\partial^{i+j} E_{12}(x,y)}{\partial x^i \partial y^j} \right| &\le C \eps^{-(i+j)/2} \E^{-c^\star (x+y) /\sqrt{\eps}}\label{eq:soldecestimE12}
	\end{align}
	\end{subequations}
	and analogous bounds for the other boundary and corner layers.
\end{lem}

Next we introduce a standard domain decomposition. Let $N$ denote a multiple of eight --- $N$ will later denote the number of mesh intervals in each coordinate direction --- and define the transition point
\begin{gather}\label{eq:trapo}
	\lambda \coloneqq \min \left \{ \frac{1}{4}, \frac{\lambda_0 \sqrt{\eps}}{c^\star} {\ln N} \right \}\qquad \text{with $\lambda_0 \ge 3$.}
\end{gather}
For our subsequent error analysis we shall make the practical and standard assumption
\begin{gather*}
	\sqrt{\varepsilon} \le C N^{-1},
\end{gather*}
from which $\lambda < 1/4$ follows.

\begin{figure}
	\centering
	\includegraphics[height=6.15cm]{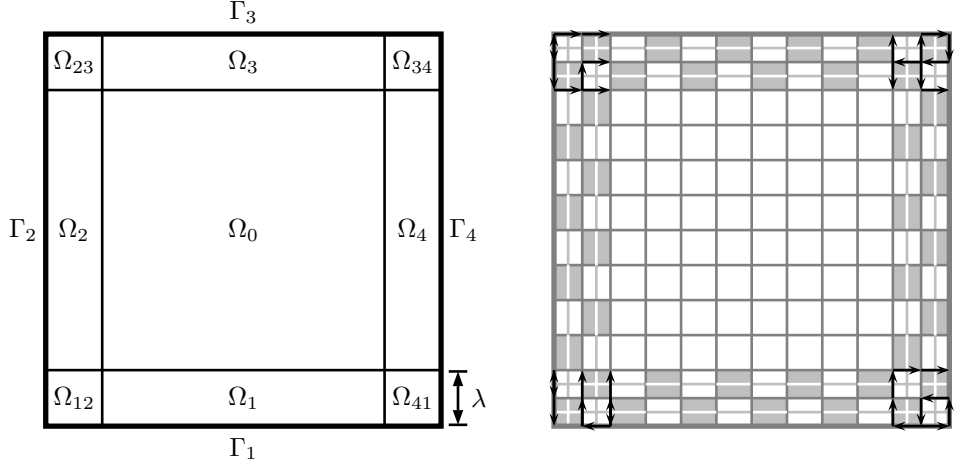}
	\caption{Domain decomposition (left) and anisotropic mesh $\Omega^N$ (right) for $N=16$, corresponding macro-element triangulation $\mathcal{M}^{16}$ of $\Omega\setminus\Omega_0$ as checkerboard and possible choice for $\sigma_{i,j}$ symbolized by black arrows pointing to the corresponding mesh node $\boldsymbol{X}_{ij}$.}
	\label{fig:domain}
\end{figure}

For our approximation error analysis we use a standard approach and split the domain into several subdomains
	\begin{alignat*}{2}
		\Omega_0 &\coloneqq (\lambda,1-\lambda)^2, &\qquad \Omega_{12} &\coloneqq (0,\lambda)^2,\\
		\Omega_1 &\coloneqq (\lambda,1-\lambda)\times(0,\lambda),&\qquad \Omega_{23} &\coloneqq (0,\lambda)\times(1-\lambda,1),\\
		\Omega_2 &\coloneqq (0,\lambda)\times(\lambda,1-\lambda),&\qquad \Omega_{34} &\coloneqq (1-\lambda,\lambda)^2,\\
		\Omega_3 &\coloneqq (\lambda,1-\lambda)\times(1-\lambda,1),&\qquad \Omega_{41} &\coloneqq (1-\lambda,\lambda)\times (0,\lambda),\\
		\Omega_4 &\coloneqq (1-\lambda,1)\times(\lambda,1-\lambda),&\qquad \Omega_f &\coloneqq \Omega_{12} \cup \Omega_{23} \cup \Omega_{34} \cup \Omega_{41},
	\end{alignat*}
as shown in the left of Figure \ref{fig:domain}. 

We use $\lambda$ to construct a 1D Shishkin mesh as follows: subdivide each of the intervals $[0,\lambda]$, $[1-\lambda,1]$ into $N/4$ subintervals, equidistantly. Giving the small grid size $h = \lambda/(N/4-2)$. Next, divide the third subinterval $[\lambda,1-\lambda]$ into $N/2$ subintervals of same size $H$. Hence, the mesh is uniform in each of the subintervals $[0,\lambda]$, $[\lambda,1-\lambda]$ and $[1-\lambda,1]$ but it changes from fine to coarse at the transition points $\lambda$ and $1-\lambda$. Remark that since $N$ is a multiple of eight the number of subintervals within $[0,\lambda]$, $[1-\lambda,1]$ and $[\lambda,1-\lambda]$ is even. Finally, form the tensor product of this one-dimensional mesh with itself to obtain our anisotropic Shishkin mesh $\Omega^N$ with the mesh nodes $\{(x_i,y_j)\}_{i,j=0,\dots,N}$. 

Note that by the definition of $\lambda$ in the inner subdomain $\Omega_0$ all the layers have declined such that they can be bounded pointwise by a constant times $N^{-\lambda_0}$. This is however not true for their derivatives.
Consequently, it is very challenging to define a $C^1$ (quasi-)interpolant of $u \in C^3(\overline\Omega)$ in the function space of piecewise biquadratics over $\Omega^N$ featuring anisotropic error estimates.
We relax this too ambitious objective by defining a quasi-interpolant $u^\star$ of $u$, such that the normal derivative of $u^\star$ is continuous only across certain edges of $\Omega^N$. For this purpose we shall use the results of the previous sections on macro-element quasi-interpolation.

In $\Omega_f$, i.e. close to corners of the domain, we combine four neighbouring elements of equal shape to form a macro-element $M = [x_{i-1},x_{i+1}] \times [y_{j-1},y_{j+1}]$ and in $\Omega_1 \cup \Omega_3$ we combine two neighbouring elements to get $M = [x_{i},x_{i+1}] \times [y_{j-1},y_{j+1}]$ as shown in the right of Figure \ref{fig:domain}. In $\Omega_2 \cup \Omega_4$ we proceed likewise. We denote the obtained macro-element triangulation by $\mathcal{M}^N$. Note that the mesh $\Omega^N$ can also be understood as the result of a refinement routine of the macro-element mesh $\mathcal{M}^N$. 

The elements of our Shishkin mesh $\Omega^N$ are axis-parallel rectangles with side lengths
\begin{gather}
	h \coloneqq \frac{4 \lambda}{N-8} = \Landau(\sqrt{\eps} N^{-1} {\ln N})\qquad \text{or} \qquad  H \coloneqq \frac{2\big(1-2\lambda\big)}{N} \sim N^{-1}.
\end{gather}
The sizes of a macro-element are equivalent to the sizes of the containing mesh elements.

Close to the corners of the domain, i.e.~in $\Omega_f$ we want to approximate $u$ by the quasi-interpolant $\tilde \Pi u$, see Subsection \ref{subsec:SZ_macro}. Hence, we have to specify how the macro-element edges $\sigma_{i,j}$ associated with the macro-element vertices $\boldsymbol{X}_{ij} \in \overline{\Omega_f}$ are chosen. If we want to satisfy Assumption \ref{assump:ass1} on our anisotropic mesh we have to choose carefully whenever $\boldsymbol{X}_{ij}$ lies on one of the lines $x=x_{N/4}=\lambda$, $x=x_{3N/4}=1-\lambda$ or $y=y_{N/4}=\lambda$, $y=y_{3N/4}=1-\lambda$ where the mesh sizes change abruptly. Restricted to $\Omega_f$ our Shishkin mesh $\Omega^N$ is (quasi-)uniform, hence any choice that satisfies
\begin{gather}\label{eq:choice_sigma}
	\sigma_{i,j} \subset \overline{\Omega_f}
\end{gather}
is possible. One may fulfill \eqref{eq:choice_sigma} as demonstrated in the right of Figure \ref{fig:domain}. In that Figure a macro-element edge $\sigma_{i,j}$ is symbolized by an arrow pointing to $\boldsymbol{X}_{ij}$. 

Let us recall the functions $\varphi_i, \psi_i \in C^1[0,1]$, supported within $[x_{1-2},x_{i+2}]$, defined by\\
\begin{minipage}{0.75\textwidth}
\begin{gather*}
\begin{alignedat}{2}
	\varphi_{i} (x) &\coloneqq \left\{ \begin{alignedat}{2} &\frac{1}{2} + \frac{x-x_{i-1}}{h_{i-1}}+\frac{(x-x_{i-1})^2}{2 h_{i-1}^2}&&\quad \text{in $[x_{i-2},x_{i-1}]$,} \\ &\frac{1}{2} + \frac{x-x_{i-1}}{h_{i-1}}-\frac{(x-x_{i-1})^2}{2 h_{i-1}^2}&&\quad \text{in $[x_{i-1},x_{i}]$,} \\ &\frac{1}{2} - \frac{x-x_{i+1}}{h_{i+1}}-\frac{(x-x_{i+1})^2}{2 h_{i+1}^2}&&\quad \text{in $[x_{i},x_{i+1}]$,} \\ &\frac{1}{2} - \frac{x-x_{i+1}}{h_{i+1}}+\frac{(x-x_{i+1})^2}{2 h_{i+1}^2}&&\quad \text{in $[x_{i+1},x_{i+2}]$,} \end{alignedat} \right.\\
	\psi_{i} (x) &\coloneqq \left\{ \begin{alignedat}{2} &{-\frac{h_{i-1}}{4}} - \frac{x-x_{i-1}}{2}-\frac{(x-x_{i-1})^2}{4 h_{i-1}}&&\quad \text{in $[x_{i-2},x_{i-1}]$,} \\ &{-\frac{h_{i-1}}{4}} - \frac{x-x_{i-1}}{2}+\frac{3(x-x_{i-1})^2}{4 h_{i-1}}&&\quad \text{in $[x_{i-1},x_{i}]$,} \\ &\frac{h_{i+1}}{4} - \frac{x-x_{i+1}}{2}-\frac{3(x-x_{i+1})^2}{4 h_{i+1}}&&\quad \text{in $[x_{i},x_{i+1}]$,} \\ &\frac{h_{i+1}}{4} - \frac{x-x_{i+1}}{2}+\frac{(x-x_{i+1})^2}{4 h_{i+1}}&&\quad \text{in $[x_{i+1},x_{i+2}]$,} \end{alignedat} \right.
\end{alignedat}
\end{gather*}
\end{minipage}%
\begin{minipage}{0.25\textwidth}
\centering
\begin{tabular}{rcc}
$x$ & $\varphi_i(x)$ & $\varphi_i'(x)$\\
\hline
$x_{i-2}$ & $0$ & $0$\\	
$x_{i-1}$ & $0.5$ & $>0$\\
$x_{i}$ & $1$ & $0$\\	
$x_{i+1}$ & $0.5$ & $<0$\\
$x_{i}$ & $0$ & $0$\\	
\end{tabular}

\vspace*{5mm}

\begin{tabular}{rcc}
$x$ & $\psi_i(x)$ & $\psi_i'(x)$\\
\hline
$x_{i-2}$ & $0$ & $0$\\	
$x_{i-1}$ & $<0$ & ${-0.5}$\\
$x_{i}$ & $0$ & $1$\\	
$x_{i+1}$ & $>0$ & ${-0.5}$\\
$x_{i}$ & $0$ & $0$\\	
\end{tabular}
\end{minipage}\\[2mm]
with $h_{i-1} \coloneqq x_{i-1} - x_{i-2} = x_i - x_{i-1}$ and $h_{i+1} \coloneqq x_{i+1} - x_i = x_{i+2} - x_{i+1}$, i.e.~$h_i = h$ for $i<N/4$ or $i>3N/4$ and $h_i = H$ else. Based on these one-dimensional functions one can define the global basis functions in the world domain
\begin{gather}\label{eq:basis2dworld}
	\begin{alignedat}{2}
	\varphi_{i,j}(x,y) &\coloneqq \varphi_i(x) \varphi_j(y),&\quad \phi_{i,j}(x,y) &\coloneqq \psi_i(x) \varphi_j(y),\\
	\chi_{i,j}(x,y) &\coloneqq \varphi_i(x) \psi_j(y),&\quad \psi_{i,j}(x,y) &\coloneqq \psi_i(x) \psi_j(y),
	\end{alignedat}\quad i,j=0,\dots,N.
\end{gather}

Now we are able to define our quasi-interpolation operator into the finite element space
\begin{gather}
	V^N \coloneqq \{ v \in H^{1}(\Omega)\;{:} \left. v\right|_T \in Q_2(T)\;\; \forall T \in \Omega^N\}.
\end{gather}

As already mentioned, for $M = [x_{i-1},x_{i+1}] \times [y_{j-1},y_{j+1}] \subset \overline{\Omega_f}$, $M \in \mathcal{M}^N$ close to the corners of the domain we use the quasi-interpolation operator $\tilde \Pi$ from Subsection \ref{subsec:SZ_macro}, i.e.
\begin{gather*}
	u^\star\big|_M = (\tilde \Pi u)|_M = \sum_{\substack{k=i-1,i+1\\ \ell=j-1,j+1}} u(x_k,y_\ell) \varphi_{k,\ell} + u_x(x_k,y_\ell) \phi_{k,\ell} + u_y(x_k,y_\ell) \chi_{k,\ell} + a_{k,\ell} \psi_{k,\ell}.
\end{gather*}
The coefficients $a_{k,\ell}$ depend on the direction of $\sigma_{k,\ell}$ given by \eqref{eq:aijfinal}:
\begin{gather*}
	 a_{k,\ell} = \left\{\begin{alignedat}{2}
		&\int_{\sigma_{k,\ell}} \frac{\partial^2 u (x,y_\ell)}{\partial x\partial y} \psi_{k}^d(x) \,\D x &\quad &\text{if $\sigma_{k,\ell}$ is horizontal,}\\
		&\int_{\sigma_{k,\ell}} \frac{\partial^2 u (x_k,y)}{\partial x\partial y} \psi_{\ell}^d(y) \,\D y &\quad &\text{if $\sigma_{k,\ell}$ is vertical,}\\
		\end{alignedat}\right.
\end{gather*}
with the dual basis functions $\psi_{k}^d$ obtained in \eqref{eq:duallocal}:
\begin{gather*}
\psi_{k}^d (x) \coloneqq \left\{
\begin{alignedat}{2}
	&{-\frac{h_{k-1}^2 + 12 h_{k-1} (x-x_{k-1})}{2 h_{k-1}^3}} \left\{ \begin{alignedat}{2} & - \frac{3 (x-x_{k-1})^2}{h_{k-1}^3},&\quad x_{k-2} &\le x \le x_{k-1}, \\ & + \frac{9 (x-x_{k-1})^2}{h_{k-1}^3},&\quad x_{k-1} &\le x \le x_{k}, \end{alignedat} \right.\qquad &\\
	&{-\frac{h_{k+1}^2 -12 h_{k+1} (x-x_{k+1})}{2 h_{k+1}^3}} \left\{ \begin{alignedat}{2} & +\frac {9 (x-x_{k+1})^2}{h_{k+1}^3},&\quad x_k &\le x \le x_{k+1}, \\ & -\frac{3 (x-x_{k+1})^2}{h_{k+1}^3},&\quad x_{k+1} &\le x \le x_{k+2}. \end{alignedat} \right.
\end{alignedat}\right.
\end{gather*}

In $\Omega_0$ we use on the element level the standard biquadratic nodal interpolant $u^I$ of $u$. Set $\mathcal{I} \coloneqq \{\frac{N}{4},\frac{N}{4}+\frac{1}{2},\frac{N}{4}+1,\frac{N}{4}+\frac{3}{2},\dots,\frac{3}{4}N\}$. Let $\ell_i$ denote the \emph{1D quadratic Lagrange basis functions}, $i\in \mathcal{I}$ with
\begin{align*}
	\ell_i(x) &= \left \{ \begin{alignedat}{2}
	& \frac{2}{h_i^2}(x-x_{i-1})(x-x_{i-1/2}),&\quad x_{i-1} &\le x \le x_{i}\\
	& \frac{2}{h_{i+1}^2}(x_{i+1}-x)(x_{i+1/2}-x),&\quad x_{i} &\le x \le x_{i+1}
	\end{alignedat} \right. \quad \text{for }i \in \mathcal{I}\cap \N,\\
	\ell_{i+1/2}(x) &= \frac{4}{h_{i+1}^2}(x_{i+1} - x)(x - x_{i}),\quad \text{for }i \in\mathcal{I}\cap \N, i \neq N,
\end{align*}
where $x_{i+1/2} \coloneqq (x_{i}+x_{i+1})/2$, $i \in\mathcal{I}\cap \N$ with $i \neq N$, denotes the midpoint of the interval $[x_{i},x_{i+1}]$. Now for $T \subset \overline{\Omega_0}$ we set
\begin{gather*}
	u^\star|_T(x,y) \coloneqq u^I|_T(x,y) = \sum_{i,j \in  \mathcal{I}} u(x_i,y_j) \ell_i(x) \ell_j(y),\quad (x,y) \in T.
\end{gather*}

Finally, we need some modified anisotropic macro-interpolation operator in $\bigcup_{i=1}^4 \Omega_i$ to glue these interpolants together. Let us consider a macro-element $M = [x_{i},x_{i+1}]\times[y_{j-1},y_{j+1}] \subset \overline{\Omega_1}$. The two elements contained in this macro-element have a long side of length $H$ in $x$-direction and a short one in $y$-direction (with length $h$). On all of these macro-elements $M \subset \overline{\Omega_1}$ that are not adjacent to $\partial\Omega_0$ we use the anisotropic macro-interpolation $\Pi^y$ as introduced and analyzed in Section \ref{sec:anisomacro_macro}, c.p.~\eqref{eq:represpiy}:
\begin{gather*}
	u^\star|_M(x,y) = \Pi^y u(x,y) \coloneqq \sum_{\substack{k\in\{i,i+1/2,i+1\}\\m\in\{j-1,j+1\}}} \left( u(k,m)  \ell_k(x) \varphi_m(y) + \frac{\partial u}{\partial y}(k,m) \ell_k(x) \psi_m(y) \right).
\end{gather*}
We use the same interpolation operator for $M \subset \overline{\Omega_3}$. In $\Omega_2\cup \Omega_4$ we use $\Pi^x$ instead. Hence, the roles of $x$ and $y$ are interchanged, there.

On macro-elements that are adjacent to $\partial \Omega_0$ we modify the anisotropic macro-interpolation operator in order to archive continuity of the normal derivative $\partial_n u^\star$ across $\partial \Omega_0$. Let for instance $M=[x_{i},x_{i+1}]\times[y_{N/4-2},y_{N/4}] \subset \overline{\Omega_1}$ denote such a macro-element. Then on $M$ the interpolant $u^\star$ is of the form:
\begin{align*}
	u^\star|_M(x,y) &= \sum_{k\in\{i,i+1/2,i+1\}} \bigg( \sum_{j\in\{N/4-2,N/4\}} u(x_k,y_j) \ell_k(x) \varphi_j(y) \\	&\quad+ \frac{\partial u}{\partial y}(x_k,y_{N/4-2}) \ell_k(x) \psi_{N/4-2}(y)
	+ \frac{\partial (u^I|_{\Omega_0})}{\partial y}(x_k,y_{N/4}) \ell_k(x) \psi_{N/4}(y) \bigg).
\end{align*}
In the other subdomains we proceed likewise. Since $\left.\frac{\partial (u^I|_{\Omega_0})}{\partial y}\right|_{M \cap \Omega_0}$ and $\left.\frac{\partial (u^\star|_{M})}{\partial y}\right|_{M \cap \Omega_0}$ are quadratic polynomials they are indeed uniquely determined by the values in three distinct points along the edge where they coincide. Note further that $\frac{\partial (u^I|_{\Omega_0})}{\partial y}(x_k,y_{N/4})$ is simply a linear combination of the nodal values $u(x_k,y_{N/4})$, $u(x_k,y_{N/4+1/2})$ and $u(x_k,y_{N/4+1})$. Hence, this coefficient is well defined along element interfaces due to the continuity of $u^I$.

Summarizing,
\begin{gather*}
	u^\star(x,y) = \left \{ \begin{alignedat}{2}
		&(\tilde \Pi u)|_M &\qquad &\text{$(x,y) \in M \subset \overline{\Omega_f}$,}\\
		& (\Pi^y u)|_M + \!\!\!\!\!\sum_{\substack{i=N/2 \\ {j \in \{N/4, 3N/4\}}}}^{3 N/2}\!\!\!\!\! \left.\frac{\partial(u^I-u)}{\partial y} \right|_{\Omega_0}\!\!\!\!(x_{i/2},y_j)\, \ell_{i/2}(x) \psi_j(y) &\qquad &\text{$(x,y) \in M \subset \overline{\Omega_1} \cup \overline{\Omega_3}$,}\\
		& (\Pi^x u)|_M + \!\!\!\!\!\sum_{\substack{j=N/2 \\ {i \in \{N/4, 3N/4\}}}}^{3 N/2}\!\!\!\!\! \left.\frac{\partial(u^I-u)}{\partial x} \right|_{\Omega_0}\!\!\!\!(x_i,y_{j/2})\, \psi_i(x) \ell_{j/2}(y)  &\qquad &\text{$(x,y) \in M \subset \overline{\Omega_2} \cup \overline{\Omega_4}$,}\\
		& u^I|_T &\qquad &\text{$(x,y) \in T \subset \overline{\Omega_0}$.}
		\end{alignedat} \right.
\end{gather*}

By construction the normal derivative of $u^\star$ is only discontinuous along short edges of anisotropic elements (type-III edges) and interior edges of $\Omega_0$ (type I edges). For some illustration see Figure \ref{fig:cont-dofs}.

\begin{figure}
	\centering
	\includegraphics{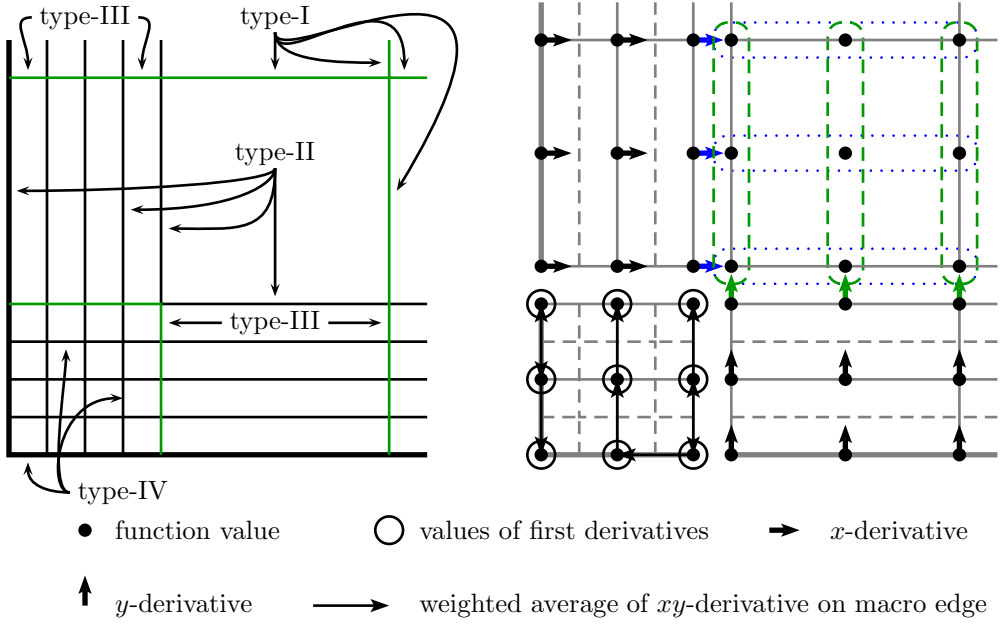}
	\caption{The normal derivative of $u^\star$ is discontinuous along the edges of type I and III highlighted in green (left) and linear functionals of $u$ that enter in the definition of $u^\star$ in the various subdomains (right).}
	\label{fig:cont-dofs}
\end{figure}

Before we analyze $u^\star$ on the Shishkin mesh $\omega^N$ let us assign a type to each element edge as shown in the left of Figure \ref{fig:cont-dofs}: 

\begin{defn}\label{def:edgetype}
A type-I edge $e \subset \overline{\Omega_0}$ is a long edge given as the intersection of two isotropic elements. An edge that belongs to at least one anisotropic element is of type II if it is a long one. Otherwise it is short and of type III. A remaining type-IV edge $e \subset \overline{\Omega_f}$ belongs to two small and square shaped elements and is close to a corner of $\Omega$. Let $\mathcal{E}(I)$ be the set of interior edges of type I and introduce similar symbols for $\mathcal{E}(II)$, $\mathcal{E}(III)$ and $\mathcal{E}(IV)$.
\end{defn}

First we show that the modification is small in various $L_2$-based norms. By the solution decomposition \eqref{eq:soldecestim}, standard interpolation error estimates and the choice of $\lambda$ wee find that
\begin{align*}
	|u-u^I|_{W_{1,\infty}(\Omega_0)} &\le |S-S^I|_{W_{1,\infty}(\Omega_0)} + |(u-S) - (u-S)^I|_{W_{1,\infty}(\Omega_0)}\\
	& \le C \big( H^2 |S|_{W_{3,\infty}} + |u-S|_{W_{1,\infty}(\Omega_0)} + |(u-S)^I|_{W_{1,\infty}(\Omega_0)}\big)\\
	& \le C \big( H^2 \eps^{-1/2} + \eps^{-1/2} N^{-\lambda_0} + H^{-1} N^{-\lambda_0}\big) \le C \eps^{-1/2} N^{-2}.
\end{align*}
Here we also used an inverse estimate. Let $\omega_1$ denote the strip of macro-elements in $\Omega_1$ that are adjacent to $\Omega_0$ then for $|\boldsymbol{\alpha}| \le 2$ it holds
\begin{gather}
	\bigg\|\boldsymbol{D}^{\boldsymbol{\alpha}} \sum_{i=N/2}^{3 N/2} \left.\frac{\partial(u^I-u)}{\partial y} \right|_{\Omega_0}(x_{i/2},y_{N/4})\, \ell_{i/2} \psi_{N/4} \bigg\|_{0,\omega_1} \le |u-u^I|_{W_{1,\infty}(\Omega_0)} \bigg\|\sum_{i=N/2}^{3 N/2} \ell_{i/2}^{(\alpha_1)} \psi_{N/4}^{(\alpha_2)} \bigg\|_{0,\omega_1} \notag\\
	\le C \eps^{-1/2} N^{-2} \meas(\Omega_1)^{1/2} \bigg\|\sum_{i=N/2}^{3 N/2} \ell_{i/2}^{(\alpha_1)} \bigg\|_{L_\infty([x_{N/4},x_{3 N/4}])} \big\|\psi_{N/4}^{(\alpha_2)} \big\|_{L_\infty([y_{N/4-2},y_{N/4}])} \label{eq:modestim}\\
	\le C \eps^{-1/4} N^{-5/2} (\ln N)^{1/2} H^{-\alpha_1} h^{1-\alpha_2} \le C \eps^{1/4-\alpha_2/2} N^{-7/2+\alpha_1+\alpha_2} (\ln N)^{1/2 - \alpha_2}.\;\;\notag
\end{gather}
For $|\boldsymbol{\alpha}| = 2$ the $L_2$ norms have to be read as norms in the broken Sobolev space over $\mathcal{M}^N$. Bounds for the other three strips $\omega_i$ in $\Omega_i$ for $i=2,3,4$ that are adjacent to $\Omega_0$ follow similarly.

Since the Shishkin mesh is (quasi-)uniform in $\Omega_f$ and by the choice of the macro-element edges according to \eqref{eq:choice_sigma} the interpolation error estimates for $\tilde \Pi$ simplify to (c.p.~Theorem \ref{thm:tapproxunif})
\begin{gather}\label{eq:tildePiuniform}
	\big| v-\tilde \Pi v \big|_{k,M} \le C h^{3-k} |v|_{3,\omega_M}\quad \text{for $v \in H^3(\omega_M)$ and $k \le 2$.}
\end{gather}

Next we estimate the approximation error of $u-u^\star$: 
\begin{lem}\label{lem:l2estimates}
	There exists a constant $C>0$ such that
	\begin{subequations}\label{eq:l2estimates}
	\begin{align}
		\|u - u^\star\|_0 &\le C \big( N^{-2} + \eps^{1/4} N^{-2} (\ln N)^2 \big), \label{eq:l2estimatea}\\
		\eps^{1/4} |u-u^\star|_1 &\le C \big( \eps^{1/4} N^{-1} + N^{-2} (\ln N)^2 \big), \label{eq:h1estimatea}\\
		\eps^{3/4} \bigg(\sum_{M \in \mathcal{M}^N} |u-u^\star|_{2,M}^2\bigg)^{1/2} &\le C \big( \eps N^{-1} (\ln N)^2 + N^{-1} \ln N \big). \label{eq:h2estimatea}
		\shortintertext{If $\eps^{1/4} \le (\ln N)^{-2}$, then}
		\|u - u^\star\|_0 &\le C N^{-2},\label{eq:l2estimatesb}\\
		\eps^{3/4} \bigg(\sum_{M \in \mathcal{M}^N} |u-u^\star|_{2,M}^2\bigg)^{1/2} &\le C N^{-1} \ln N. \label{eq:h2estimateb}
		\shortintertext{If $|S|_{3} \le C \eps^{-1/4}$, then}
		\eps^{1/4} |u-u^\star|_1 &\le C  N^{-2} (\ln N)^2. \label{eq:h1estimateb}
		\shortintertext{Suppose $\eps^{1/4} \le (\ln N)^{-3}$ and $|S|_3 + \sum_{i\in{1,3}} \|\boldsymbol{D}^{(3,0)} E_i\|_{0,\Omega_i} + \sum_{j\in{2,4}} \|\boldsymbol{D}^{(0,3)} E_j\|_{0,\Omega_j} \le C$, then}
		\|u - u^\star\|_0 &\le C N^{-3} (\ln N)^3.\label{eq:l2estimatesc}
	\end{align}
	\end{subequations}
\end{lem}
\begin{proof}
We use the solution decomposition \eqref{eq:soldecestim} several times without mentioning it explicitly and different techniques in each subdomain.

In $\Omega_f$ the approximation error is small because the mesh is very fine. We use \eqref{eq:tildePiuniform}:
\begin{gather}\label{eq:estimOmegah}
\begin{aligned}
	|u-u^\star|_{k,\Omega_f} = \big|u- \tilde \Pi u\big|_{k,\Omega_f} \le C h^{3-k} |u|_{3,\Omega_f} &\le C h^{3-k} \eps^{-3/2} \meas(\Omega_f)^{1/2}\\
	&= C \eps^{(1-k)/2} N^{k-3} (\ln N)^{2-k}.
\end{aligned}
\end{gather}

In $\Omega_0$ the Shishkin mesh is coarse but all layer components have declined sufficiently. With the $L_\infty$- stability of the nodal interpolant we get
\begin{gather*}
	\|(u-S)^I\|_{L_\infty(\Omega_0)} \le C \| u-S \|_{L_\infty(\Omega_0)} \le C N^{-\lambda_0} \le C N^{-3}.
\end{gather*}
Hence, we obtain for the layer components of $u$ and $k \le 2$ with an inverse estimate
\begin{gather}\label{eq:layerOmega0}
\begin{gathered}
	|(u-S)-(u-S)^\star|_{k,\Omega_0} \le |(u-S)-(u-S)^I|_{k,\Omega_0} \le |(u-S)|_{k,\Omega_0} + |(u-S)^I|_{k,\Omega_0}\\
	\le C\big( \eps^{1/4-k/2} N^{-\lambda_0} + H^{-k} \|(u-S)^I\|_{0,\Omega_0} \big)
	\le C\big( \eps^{1/4-k/2} N^{-3} + N^{k-3} \big).
\end{gathered}
\end{gather}
For the smooth solution component $S$ we estimate
\begin{subequations}
\label{eq:smoothOmega0}
\begin{gather}
	|S-S^\star|_{k,\Omega_0} = |S-S^I|_{k,\Omega_0} \le C H^{2-k} |S|_{2,\Omega_0} \le C N^{k-2}\quad \text{for $k=0,1$}
\end{gather}
and
\begin{gather}
	|S-S^\star|_{2,\Omega_0} = |S-S^I|_{2,\Omega_0} \le C H |S|_{3,\Omega_0} \le C \eps^{-1/2} N^{-1}.
\end{gather}
\end{subequations}
Obviously these bounds can be improved to $|S-S^\star|_{k,\Omega_0} \le C N^{k-3}$ if $|S|_{3} < C$.

In the remainder of the domain the elements of the Shishkin mesh are anisotropic. For the smooth part $S$ we use Lemma \ref{lem:anisoL2suboptimal}, for instance in $\Omega_1$:
\begin{gather}
\label{eq:smoothOmega1}
\begin{aligned}
	\|S-\Pi^y S\|_{0,\Omega_1} &\le C \sum_{|\boldsymbol{\alpha}| = 2} \big( \boldsymbol{h}_M^{\boldsymbol{\alpha}} \|\boldsymbol{D}^{\boldsymbol{\alpha}} S\|_{0,\Omega_1} + \boldsymbol{h}_M^{\boldsymbol{\alpha} + (0,1)} \|\boldsymbol{D}^{\boldsymbol{\alpha}} S_y\|_{0,\Omega_1} \big) \\
	&\le C (H^{2} + \eps^{-1/2} H^2 h) \le C N^{-2}.
\end{aligned}
\end{gather}
If $|S|_{3,\Omega_1} < C$ we could improve the estimate to $\|S-\Pi^y S\|_{0,\Omega_1} \le C N^{-3}$ using \eqref{eq:PianxxM}. In the other subdomains $\Omega_i$ for $i=2,3,4$ the smooth part is estimated similarly.
For the layer term $E_1$ Lemma \ref{lem:anisoL2suboptimal} yields 
\begin{align}
	\|E_1 - \Pi^y E_1\|_{0,\Omega_1} &\le C \sum_{|\boldsymbol{\alpha}| = 2} \big( \boldsymbol{h}_M^{\boldsymbol{\alpha}} \|\boldsymbol{D}^{\boldsymbol{\alpha}} E_1\|_{0,\Omega_1} + \boldsymbol{h}_M^{\boldsymbol{\alpha} + (0,1)} \|\boldsymbol{D}^{\boldsymbol{\alpha} + (0,1)} E_1\|_{0,\Omega_1} \big)\notag\\
	& \le H^2 \|\boldsymbol{D}^{(2,0)} E_1\|_{0,\Omega_1} + H h \|\boldsymbol{D}^{(1,1)} E_1\|_{0,\Omega_1} + h^2 \|\boldsymbol{D}^{(0,2)} E_1\|_{0,\Omega_1}\notag\\
	&\quad + H^2 h \|\boldsymbol{D}^{(2,1)} E_1\|_{0,\Omega_1} + H h^2 \|\boldsymbol{D}^{(1,2)} E_1\|_{0,\Omega_1} + h^3 \|\boldsymbol{D}^{(0,3)} E_1\|_{0,\Omega_1}\notag\\
	& \le C \big( N^{-2} \eps^{1/4} + \eps^{1/2} N^{-2} \ln N \eps^{-1/4} + \eps N^{-2} (\ln N)^2 \eps^{-3/4}\notag\\
	&\quad + \eps^{1/2} N^{-3} \ln N \eps^{-1/4} + \eps N^{-3} (\ln N)^2 \eps^{-3/4} + \eps^{3/2} N^{-3} (\ln N)^3 \eps^{-5/4}\notag\\
	& \le C \eps^{1/4} N^{-2} (\ln N)^2.\label{eq:EiOmegai}
\end{align}
If $\|\boldsymbol{D}^{(3,0)} E_1\|_{0,\Omega_1} \le C$ this bound can be improved to $\|E_1 - \Pi^y E_1\|_{0,\Omega_1} \le C N^{-3} (\ln N)^3$ with \eqref{eq:PianxxM}. With the same technique one can estimate the layer component $E_i$ on $\Omega_i$, $i=2,3,4$. The other layer components are small on $\Omega_1$, for instance for the corner layer $E_{12}$ it holds
\begin{gather}
\label{eq:EiOmegaj}
\begin{aligned}
	\|E_{12} - \Pi^y E_{12}\|_{0,\Omega_1} &\le \|E_{12}\|_{0,\Omega_1} + (\meas{\Omega_1})^{1/2} \|\Pi^y E_{12}\|_{L_\infty(\Omega_1)}\\
	& \le C (\meas{\Omega_1})^{1/2} \big( \|E_{12}\|_{L_\infty(\Omega_1)} + h \|\boldsymbol{D}^{(0,1)} E_{12}\|_{L_\infty(\Omega_1)} \big)\\
	& \le C \eps^{1/4} (\ln N)^{1/2} \big( N^{-\lambda_0} + \eps^{1/2} N^{-1} \ln N \eps^{-1/2} N^{-\lambda_0} \big)\\
	& \le C \eps^{1/4} N^{-\lambda_0} (\ln N)^{1/2}.
\end{aligned}
\end{gather}
Here we used the stability estimate \eqref{eq:Piainfstab}. Proceed similarly for all layer components $E_j$ on $\Omega_i$ with $j \in \{1,\dots,4,12,23,34,41\}$ and $i=1,\dots,4$ with $i\neq j$. Now collect \eqref{eq:modestim} for $\boldsymbol{\alpha}=(0,0)$, \eqref{eq:estimOmegah}, \eqref{eq:layerOmega0}, \eqref{eq:smoothOmega0} with $k=0$, \eqref{eq:smoothOmega1}, \eqref{eq:EiOmegai} and \eqref{eq:EiOmegaj} to obtain \eqref{eq:l2estimatea}.

Next if we want to estimate $\eps^{1/4} |u-u^\star|_1$ it remains to estimate the error on the anisotropic elements, for instance on $\Omega_1$. There the smooth solution component can be bounded with \eqref{eq:PianxxM}. Let $|\boldsymbol{\gamma}| = 1$, then
\begin{gather}\label{eq:H1smoothOmega1}
\begin{aligned}
	\|\boldsymbol{D}^{\boldsymbol{\gamma}} (S- \Pi^y S)\|_{0,\Omega_1} &\le C \sum_{|\boldsymbol{\alpha}|=2} \boldsymbol{h}_M^{\boldsymbol{\alpha}} \|\boldsymbol{D}^{\boldsymbol{\alpha} + \boldsymbol{\gamma}} S \|_{0, \Omega_1}
	\le C N^{-2} \|S\|_{W_{3,\infty}(\Omega_1)} (\meas \Omega_1)^{1/2}\\
	&\le C N^{-2} \eps^{-1/2} \eps^{1/4} (\ln N)^{1/2} \le C \eps^{-1/4} N^{-2} (\ln N)^{1/2}.
\end{aligned}
\end{gather}
The other domains $\Omega_i$, $i=2,3,4$ are treated similarly. From \eqref{eq:PianxxM} we deduce for the boundary layer component $E_1$ that
\begin{gather}\label{eq:H1EiOmegaiy}
\begin{multlined}
	\|\boldsymbol{D}^{(0,1)} (E_1 - \Pi^y E_1)\|_{0,\Omega_1} \le C \sum_{|\boldsymbol{\alpha}|=2} \boldsymbol{h}_M^{\boldsymbol{\alpha}} \|\boldsymbol{D}^{\boldsymbol{\alpha} + (0,1)} E_1 \|_{0, \Omega_1}\\
	\le C \big( H^2 \|\boldsymbol{D}^{(2,1)} E_1\|_{0, \Omega_1} + H h \|\boldsymbol{D}^{(1,2)} E_1\|_{0, \Omega_1} + h^2 \|\boldsymbol{D}^{(0,3)} E_1\|_{0, \Omega_1} \big)\\
	\le C \big( N^{-2} \eps^{-1/4} + \eps^{1/2} N^{-2} \ln N \eps^{-3/4} + \eps N^{-2} (\ln N)^2 \eps^{-5/4} \le C \eps^{-1/4} N^{-2} (\ln N)^2.
\end{multlined}
\end{gather}
The derivative with respect to $x$ is better behaved and the same bound holds true: 
\begin{gather}\label{eq:H1EiOmegaix}
\begin{multlined}
	\|\boldsymbol{D}^{(1,0)} (E_1 - \Pi^y E_1)\|_{0,\Omega_1} \le C \sum_{|\boldsymbol{\alpha}|=2} \boldsymbol{h}_M^{\boldsymbol{\alpha}} \|\boldsymbol{D}^{\boldsymbol{\boldsymbol{\alpha}} + (1,0)} E_1 \|_{0, \Omega_1}\\
	\le C \big( H^2 \|\boldsymbol{D}^{(3,0)} E_1\|_{0, \Omega_1} + H h \|\boldsymbol{D}^{(2,1)} E_1\|_{0, \Omega_1} + h^2 \|\boldsymbol{D}^{(1,2)} E_1\|_{0, \Omega_1} \big)\\
	\le C \big( N^{-2} \eps^{-1/4} + \eps^{1/2} N^{-2} \ln N \eps^{-1/4} + \eps N^{-2} (\ln N)^2 \eps^{-3/4} \le C \eps^{-1/4} N^{-2} (\ln N)^2.
\end{multlined}
\end{gather}
Obviously this bound holds also on $\Omega_3$ where the anisotropy of the elements is in the same direction compared to $\Omega_1$. In $\Omega_2$ (or $\Omega_4$) we use inverse estimates and the stability of $\Pi^x$:
\begin{align}
	&| E_1 - \Pi^x E_1 |_{1,\Omega_2} \le | E_1 |_{1,\Omega_2} + C h^{-1} \| \Pi^x E_1\|_{0,\Omega_2}\notag\\
	&\qquad\le C (\meas \Omega_2)^{1/2} \big( | E_1 |_{W_{1,\infty}(\Omega_2)} + h^{-1} ( \| E_1\|_{L_\infty(\Omega_2)} + h \| \boldsymbol{D}^{(1,0)} E_1\|_{L_\infty(\Omega_2)} ) \big)\label{eq:H1EiOmegaj}\\
	&\qquad\le C \eps^{1/4} (\ln N)^{1/2} \big( \eps^{-1/2} N^{-\lambda_0} + \eps^{-1/2} N (\ln N)^{-1} N^{-\lambda_0} + N^{-\lambda_0}\big) \le C \eps^{-1/4} N^{-2} (\ln N)^{-1/2}\!\!.\notag
\end{align}
Clearly, this technique can also be applied to estimate $E_i$, $i=2,3,4$. The corner layer components are bounded in exactly the same way. Consider for instance $E_{12}$ on $\Omega_1$:
\begin{align}
	&| E_{12} - \Pi^y E_{12} |_{1,\Omega_1} \le | E_{12} |_{1,\Omega_1} + C h^{-1} \| \Pi^y E_{12}\|_{0,\Omega_1}\notag\\
	&\qquad\le C (\meas \Omega_1)^{1/2} \big( | E_{12} |_{W_{1,\infty}(\Omega_1)} + h^{-1} ( \| E_{12} \|_{L_\infty(\Omega_1)} + h \| \boldsymbol{D}^{(0,1)} E_{12}\|_{L_\infty(\Omega_1)} ) \big)\label{eq:H1EikOmegaj}\\
	&\qquad\le C \eps^{1/4} (\ln N)^{1/2} \big( \eps^{-1/2} N^{-\lambda_0} + \eps^{-1/2} N (\ln N)^{-1} N^{-\lambda_0} \big) \le C \eps^{-1/4} N^{-2} (\ln N)^{-1/2}.\notag
\end{align}
Collecting \eqref{eq:modestim} for $|\boldsymbol{\alpha}|=1$, \eqref{eq:estimOmegah}, \eqref{eq:layerOmega0}, \eqref{eq:smoothOmega0} with $k=1$, \eqref{eq:H1smoothOmega1}, \eqref{eq:H1EiOmegaiy}, \eqref{eq:H1EiOmegaix}, \eqref{eq:H1EiOmegaj} and \eqref{eq:H1EikOmegaj} yields \eqref{eq:h1estimatea}.

Finally, we consider second order derivatives. Unfortunately $u^\star \not \in H^2(\Omega)$. However, $u^\star \in H^2(T)$ for all $T \in \Omega^N$ and even $u^\star \in H^2(M)$ for all $M \in \mathcal{M}^N$. Hence, we introduce the abbreviation  $\|v\|_{0,\mathcal{M}(V)} \coloneqq (\sum_{M \in \mathcal{M}, M \subset V} \|v\|_{0,M}^2)^{1/2}$. Now Let $|\boldsymbol{\gamma}| = 2$, then by \eqref{eq:PianxxM} and \eqref{eq:PiaxxM} we find for instance in $\Omega_1$ that
\begin{align}
	\|\boldsymbol{D}^{\boldsymbol{\gamma}} (S- \Pi^y S)\|_{0,\mathcal{M}(\Omega_1)} &\le C \bigg(\sum_{|\boldsymbol{\alpha}|=1} \boldsymbol{h}_M^{\boldsymbol{\alpha}} \|\boldsymbol{D}^{\boldsymbol{\alpha} + \boldsymbol{\gamma}} S \|_{0, \Omega_1} + \sum_{|\boldsymbol{\alpha}|=2} \frac{\boldsymbol{h}_M^{\boldsymbol{\alpha}}}{H} \|\boldsymbol{D}^{\boldsymbol{\alpha}} \boldsymbol{D}^{(1,0)} S\|_{0,\Omega_1} \bigg)\notag\\
	&\le C N^{-1} \|S\|_{W_{3,\infty}(\Omega_1)} (\meas \Omega_1)^{1/2}\label{eq:H2smoothOmega1}\\
	&\le C N^{-1} \eps^{-1/2} \eps^{1/4} (\ln N)^{1/2} \le C \eps^{-1/4} N^{-1} (\ln N)^{1/2}.\notag
\end{align}
Similar bounds hold on $\Omega_i$ for $i=2,3,4$. In oder to obtain bounds for the layer components $E_i$ on $\Omega_i$ ($i=1,\dots,4$) we use \eqref{eq:PianxxM} and \eqref{eq:PiaxxM} more careful.
\begin{align}
	\|\boldsymbol{D}^{(0,2)} (E_1- \Pi^y E_1)\|_{0,\mathcal{M}(\Omega_1)} &\le C \sum_{|\boldsymbol{\alpha}|=1} \boldsymbol{h}_M^{\boldsymbol{\alpha}} \|\boldsymbol{D}^{\boldsymbol{\alpha} + (0,2)} E_1 \|_{0, \Omega_1}\notag\\
& \le C \big( H \|\boldsymbol{D}^{(1,2)} E_1 \|_{0, \Omega_1} + h \|\boldsymbol{D}^{(0,3)} E_1 \|_{0, \Omega_1} \big)\label{eq:H2EiOmegaiyy}\\
& \le C \big( N^{-1} \eps^{-3/4} + \eps^{1/2} N^{-1} \ln N \eps^{-5/4} \big) \le \eps^{-3/4} N^{-1} \ln N, \allowdisplaybreaks[1]\notag\\
	\|\boldsymbol{D}^{(1,1)} (E_1- \Pi^y E_1)\|_{0,\mathcal{M}(\Omega_1)} &\le C \sum_{|\boldsymbol{\alpha}|=1} \boldsymbol{h}_M^{\boldsymbol{\alpha}} \|\boldsymbol{D}^{\boldsymbol{\alpha} + (1,1)} E_1 \|_{0, \Omega_1}\notag\\
& \le C \big( H \|\boldsymbol{D}^{(2,1)} E_1 \|_{0, \Omega_1} + h \|\boldsymbol{D}^{(1,2)} E_1 \|_{0, \Omega_1} \big)\label{eq:H2EiOmegaixy}\\
& \le C \big( N^{-1} \eps^{-1/4} + \eps^{1/2} N^{-1} \ln N \eps^{-3/4} \big) \le \eps^{-1/4} N^{-1} \ln N,\allowdisplaybreaks[1]\notag
\end{align}

\vspace*{-8mm}
\begin{align}
	\|\boldsymbol{D}^{(2,0)} (E_1- \Pi^y E_1)\|_{0,\mathcal{M}(\Omega_1)} &\le C \bigg( \sum_{|\boldsymbol{\alpha}|=1} \!\boldsymbol{h}_M^{\boldsymbol{\alpha}} \|\boldsymbol{D}^{\boldsymbol{\alpha} + (2,0)} E_1 \|_{0, \Omega_1} + \sum_{|\boldsymbol{\alpha}|=2} \!\frac{\boldsymbol{h}_M^{\boldsymbol{\alpha}}}{H} \|\boldsymbol{D}^{\boldsymbol{\alpha}} \boldsymbol{D}^{(1,0)} E_1\|_{0,\Omega_1} \bigg)\notag\\
&\hspace*{-2cm} \le C \big( H \|\boldsymbol{D}^{(3,0)} E_1 \|_{0, \Omega_1} + h \|\boldsymbol{D}^{(2,1)} E_1 \|_{0, \Omega_1} + h^2 H^{-1} \|\boldsymbol{D}^{(1,2)} E_1 \|_{0, \Omega_1} \big)\notag\\
&\hspace*{-2cm} \le C \big( N^{-1} \eps^{-1/4} + \eps^{1/2} N^{-1} \ln N \eps^{-1/4} + \eps N^{-1} (\ln N)^2 \eps^{-3/4} \big)\label{eq:H2EiOmegaixx}\\
&\hspace*{-2cm} \le C \big( \eps^{-1/4} N^{-1} + \eps^{1/4} N^{-1} (\ln N)^2 \big).\notag
\end{align}
The same technique can be used to bound the error of $E_i$ on the anisotropic part of the Shishkin mesh along the opposite edge. In $\Omega_2$ (or $\Omega_4$) inverse estimates and the stability of $\Pi^x$ yield again:
\begin{gather}
\label{eq:H2EiOmegaj}
\begin{aligned}
	&| E_1 - \Pi^x E_1 |_{2,\mathcal{M}(\Omega_1)} \le | E_1 |_{2,\Omega_2} + C h^{-2} \| \Pi^x E_1\|_{0,\Omega_2}\\
	&\qquad\le C (\meas \Omega_2)^{1/2} \big( | E_1 |_{W_{2,\infty}(\Omega_2)} + h^{-2} ( \| E_1\|_{L_\infty(\Omega_2)} + h \| \boldsymbol{D}^{(1,0)} E_1\|_{L_\infty(\Omega_2)} ) \big)\\
	&\qquad\le C \eps^{1/4} (\ln N)^{1/2} \big( \eps^{-1} N^{-\lambda_0} + \eps^{-1} N^{2} (\ln N)^{-2} N^{-\lambda_0} + \eps^{-1/2} N (\ln N)^{-1} N^{-\lambda_0}\big)\\
	&\qquad \le C \eps^{-3/4} N^{-1} (\ln N)^{-3/2}.
\end{aligned}
\end{gather}
The corner layers are handled similarly, for instance $E_{12}$ on $\Omega_1$:
\begin{align}
	&| E_{12} - \Pi^y E_{12} |_{2,\mathcal{M}(\Omega_1)} \le | E_{12} |_{2,\Omega_1} + C h^{-2} \| \Pi^y E_{12}\|_{0,\Omega_1}\notag\\
	&\qquad\le C (\meas \Omega_1)^{1/2} \big( | E_{12} |_{W_{2,\infty}(\Omega_1)} + h^{-2} ( \| E_{12} \|_{L_\infty(\Omega_1)} + h \| \boldsymbol{D}^{(0,1)} E_{12}\|_{L_\infty(\Omega_1)} ) \big)\notag\\
	&\qquad\le C \eps^{1/4} (\ln N)^{1/2} \big( \eps^{-1} N^{-\lambda_0} + \eps^{-1} N^{2} (\ln N)^{-2} N^{-\lambda_0} + \eps^{-1/2} N (\ln N)^{-1} \eps^{-1/2} N^{-\lambda_0} \big)\notag\\
	&\qquad\le C \eps^{-3/4} N^{-2} (\ln N)^{-1/2}.\label{eq:H2EikOmegaj}
\end{align}
Collect \eqref{eq:modestim} for $|\boldsymbol{\alpha}|=2$, \eqref{eq:estimOmegah}, \eqref{eq:layerOmega0}, \eqref{eq:smoothOmega0} with $k=2$, \eqref{eq:H2smoothOmega1}, \eqref{eq:H2EiOmegaiyy}, \eqref{eq:H2EiOmegaixy}, \eqref{eq:H2EiOmegaixx}, \eqref{eq:H2EiOmegaj} and \eqref{eq:H2EikOmegaj} to obtain \eqref{eq:h2estimatea}. The other assertions of the Lemma follow easily.
\end{proof}

After quantifying the approximation properties of $u^\star$ we want to study certain traces of $u-u^\star$ along interior edges. 

Since $\Omega^N$ is an admissible triangulation two elements $T_1,T_2 \in \Omega^N$ define traces of a function $v\in H^1(T_1 \cup T_2) \cap H^2(T_1) \cap H^2(T_2)$ along an interior edge $e$. We associate a unit normal vector $n$ with each edge. If $e\subset \partial \Omega$ is an edge along the boundary we define $n$ as the unit outer normal to $\partial \Omega$. In a similar manner there are two traces of the normal derivative $\frac{\partial v}{\partial n} \in L_2(e)$. Assuming $n$ is oriented from $T_1$ to $T_2$ we obtain jumps $\jump{\frac{\partial v}{\partial n}}$ of these traces as follows:
\begin{gather*}
	\bigjump{\frac{\partial v}{\partial n}} \coloneqq \left.\frac{\partial v}{\partial n}\right|_{T_1} - \left.\frac{\partial v}{\partial n}\right|_{T_2} \in L_2(e).
\end{gather*}

\begin{lem}\label{lem:tracel2estim}
	Suppose $\eps^{1/4} \le (\ln N)^{-2}$. Then there is a positive constant $C$ such that
	\begin{align}
		\|u-u^\star\|_{0,e}^2 &\le C N^{-5}\qquad \text{on a long edge $e$, i.e.~of type I or II},\label{eq:tracelong}\\
		\sum_{e \in \mathcal{E}(III)} \|u - u^\star\|_{0,e}^2 &\le C \eps^{-1/2} N^{-5} (\ln N)^2,\label{eq:traceIII}\\
		\sum_{e \in \mathcal{E}(IV)} \|u - u^\star\|_{0,e}^2 &\le C \eps^{1/2} N^{-5} (\ln N)^3\label{eq:traceIV}. 
	\end{align}
\end{lem}
\begin{proof}
	Let $e \subset \overline\Omega \setminus \Omega_f$ denote a long type-I or type-II edge of a possibly anisotropic element.
	For instance, on a long edge $e \subset \Omega_1$ the interpolant $\Pi^y v$ of $v$ is a quadratic polynomial which is uniquely described by its values in the endpoints and the midpoint of $e$. Hence, on long edges $\Pi^y$ coincides with the 1D Lagrange interpolation and we find that
	\begin{gather}\label{eq:tracesmoothEiOmegai}
	\begin{aligned}
		\|(S + E_1)- \Pi^y (S + E_1)\|_{0,e}^2 &\le \meas(e) \|(S + E_1) - \Pi^y (S + E_1)\|_{L_\infty(e)}^2\\
		&\le C H\,H^{-4} \|(S + E_1)_{yy}\|_{L_\infty(e)}^2 \le C N^{-5}.
	\end{aligned}
	\end{gather}
	Any other layer component $E \coloneqq u - S - E_1$ is estimated using a stability argument of the interpolation operator involved on a macro-element $M$ that is adjacent to $e \subset M$:
	\begin{gather}\label{eq:tracelayerOmegai}
	\begin{aligned}
		\|E - \Pi^y E\|_{0,e}^2 &\le \meas(e) \|E - \Pi^y E\|_{L_\infty(e)}^2 \le C H \big( \|E\|_{L_\infty(M)}^2 + \| \Pi^y E\|_{L_\infty(M)}^2 \big) \\
		&\le C H \big( \|E\|_{L_\infty(M)}^2 + h \|E_y\|_{L_\infty(M)}^2 \big) \le C N^{-\lambda_0 - 1} \le C N^{-7}.
	\end{aligned}	
	\end{gather}
	Similarly to \eqref{eq:tracesmoothEiOmegai}, we estimate the smooth part $S$ on any edge $e \subset \overline \Omega_0$ in the interior subdomain:
	\begin{gather}\label{eq:tracesmoothOmega0}
		\|S - S^I\|_{0,e}^2 \le \meas(e) \|S - S^I\|_{L_\infty(e)}^2 \le C H\,H^{-4} |S|_{W_{2,\infty}(e)}^2 \le C N^{-5}.
	\end{gather}
	Next we use that all the layer components $E \coloneqq u - S$ have declined sufficiently. Let $T \subset \overline{\Omega_0}$ denote an element that has the edge $e$, then
	\begin{gather}\label{eq:tracelayerOmega0}
	\begin{aligned}
		\|E - E^I\|_{0,e}^2 &\le \meas(e) \|E - E^I\|_{L_\infty(e)}^2 \le C H \big( \|E\|_{L_\infty(T)}^2 + \| E^I\|_{L_\infty(T)}^2 \big) \\
		&\le C H \|E\|_{L_\infty(T)}^2 \le C N^{-\lambda_0 - 1} \le C N^{-7}.
	\end{aligned}
	\end{gather}
	Collecting \eqref{eq:tracesmoothEiOmegai}, \eqref{eq:tracelayerOmegai}, \eqref{eq:tracesmoothOmega0} and \eqref{eq:tracelayerOmega0} gives \eqref{eq:tracelong}. 

	Now we consider the short type-III edge $e$ of an anisotropic element $T$ for instance in $\Omega_1$. We use the trace Lemma \ref{lem:tracelem} and \eqref{eq:PianxxM}:
	\begin{gather}\label{eq:smoothtraceIII}
		\begin{aligned}
		\|S - \Pi^y S\|_{0,e}^2 &\le C \|S - \Pi^y S\|_{0,T} \|(S - \Pi^y S)_x\|_{0,T} + \frac{1}{H} \|S - \Pi^y S\|_{0,T}^2\le C H^5 |S|_{3,M}^2\\ &\le C \meas(M) H^5 |S|_{W_{3,\infty}(M)}^2 \le C \eps^{1/2} N^{-2} \ln N\, N^{-5} \eps^{-1} \le C \eps^{-1/2} N^{-7} \ln N
		\end{aligned}
	\end{gather}
	Here $M$ denotes the macro-element such that $T \subset M$. Similarly, we obtain for the layer $E_1$
	\begin{gather*}
		\|E_1 - \Pi^y E_1\|_{0,e}^2 \le C \|E_1 - \Pi^y E_1\|_{0,T} \|(E_1 - \Pi^y E_1)_x\|_{0,T} + \frac{1}{H} \|E_1 - \Pi^y E_1\|_{0,T}^2.
	\end{gather*}
	Hence, a summation over all type-III edges gives with Young's inequality
	\begin{gather}\label{eq:E1traceIII}
		\sum_{e \in \mathcal{E}(III)} \|E_1 - \Pi^y E_1\|_{0,e}^2 \le C \eps^{-1/2} N^{-5} (\ln N)^2, 
	\end{gather}
	due to \eqref{eq:H1EiOmegaix} and a similar estimate with \eqref{eq:PianxxM} and $\eps^{1/4} \le (\ln N)^{-2}$ for $\|E_1 - \Pi^y E_1\|_{0,\Omega_1}$, namely
	\begin{align*}
			\| E_1 - \Pi^y E_1 \|_{0,\Omega_1} &\le C \sum_{|\boldsymbol{\alpha}|=3} \boldsymbol{h}_M^{\boldsymbol{\alpha}}\|\boldsymbol{D}^{\boldsymbol{\alpha}} E_1 \|_{0, \Omega_1}
			\le C \big( H^3 \|\boldsymbol{D}^{(3,0)} E_1\|_{0, \Omega_1} \\
			&\quad + H^2 h \|\boldsymbol{D}^{(2,1)} E_1\|_{0, \Omega_1} + H h^2 \|\boldsymbol{D}^{(1,2)} E_1\|_{0, \Omega_1} + h^3 \|\boldsymbol{D}^{(0,3)} E_1\|_{0, \Omega_1} \big)\\
			&\le C \big( \eps^{-1/4} N^{-3} + \eps^{1/4} N^{-3} (\ln N)^3 \big).
	\end{align*}
	The other layer components can be estimated like in \eqref{eq:tracelayerOmegai}. With \eqref{eq:smoothtraceIII} and \eqref{eq:E1traceIII} we arrive at \eqref{eq:traceIII}.
	
	For the short type-IV edges of $\Omega_f$ close to the corners of the domain we again use the a trace Lemma and \eqref{eq:estimOmegah} to obtain
	\begin{gather*}
		\sum_{e \in \mathcal{E}(IV)} \|u - \tilde \Pi u\|_{0,e}^2 \le C \eps^{1/2} N^{-5} (\ln N)^3,
	\end{gather*}
	which is \eqref{eq:traceIV}.
\end{proof}

\begin{lem}[Anisotropic multiplicative trace inequality]\label{lem:amti}
	Let $T$ be a rectangle with sides parallel to the coordinate axes and a width in $x$-direction of $h_x$. Let $\partial T_y$ denote the union of the two edges parallel to the $y$-axis. Then for $v \in W_{1,p}(T)$ we have the estimate
	\begin{gather}
		\|v\|^p_{L_p(\partial T_y)} \le p \|v\|_{L_p(T)}^{p-1} \|v_x\|_{L_p(T)} + \frac{2}{h_x} \|v\|_{L_p(T)}^p\quad \text{for $p\in[1,\infty)$},\\
		\|v\|_{L_\infty(\partial T_y)} \le \|v\|_{L_\infty(T)}.
	\end{gather}
\end{lem}
\begin{proof}The proof follows its isotropic version in \cite[Theorem 1.5.1.10]{G85} (or \cite[Lemma 3.1]{DFS02} in the $L_2$ setting):
	Without loss of generality we assume that the origin of the coordinate system is given by the midpoint of the rectangle $T$. The divergence theorem yields for $v \in C^1(\overline{T})$:
	\begin{gather}
		\begin{split}
			\int \limits_T \frac{\partial}{\partial x} \left(|v|^p x \right) \D x\D y &= \int \limits_T \nabla \cdot \begin{pmatrix} |v|^p x \\ 0 \end{pmatrix} \D x\D y = \int \limits_{\partial T} n \cdot \begin{pmatrix} |v|^p x \\ 0 \end{pmatrix} \D s \\
			&= \int \limits_{\partial T_y} |v|^p |x| \D s = \frac{h_x}{2} \int \limits_{\partial T_y} |v|^p \D s = \frac{h_x}{2} \|v\|^p_{L_p(\partial T_y)}.
		\end{split}
	\end{gather}
	Moreover since $|x| \le h_x/2$ on $T$ an application of the product rule and Hölder's inequality with $\frac{1}{p}+\frac{1}{q}=1$ imply
	\begin{gather*}
		\int \limits_T \frac{\partial}{\partial x} \left(|v|^p x \right) \D x\D y = \int \limits_T \frac{\partial}{\partial x} \left(|v|^p\right) x \,\D x\D y + \int \limits_T |v|^p \D x\D y = p \int \limits_T |v|^{p-2} v \frac{\partial v}{\partial x} x\,\D x\D y + \|v\|_{L_p(T)}^p\\
		\le \frac{p h_x}{2} \int \limits_T |v|^{p-1} \left|\frac{\partial v}{\partial x} \right| \D x \D y + \|v\|_{L_p(T)}^p \le \frac{p h_x}{2} \Biggl( \int \limits_T |v|^p \D x\D y \Biggr)^{1/q} \left\| \frac{\partial v}{\partial x} \right \|_{L_p(T)} + \|v\|_{L_p(T)}^p \\
		\le \frac{p h_x}{2} \|v\|_{L_p(T)}^{p-1} \left\| \frac{\partial v}{\partial x} \right \|_{L_p(T)} + \|v\|_{L_p(T)}^p.
	\end{gather*}
	The assertion follows from a standard density argument. The case $p=\infty$ is trivial.
\end{proof}

\begin{lem}\label{lem:tracelem} Let $T$ be a rectangle with sides parallel to the coordinate axes and a width in $x$-direction of $h_x$. Let $\partial T_y$ denote the union of the two edges parallel to the $y$-axis having length $h_y$. Denote by $v^I \in Q_2(T)$ the nodal interpolant of $v \in C(\bar T)$. Then for $v \in H^3(T)$ it holds
	\begin{gather}
		\left\|\left(v - v^I \right)_x \right\|_{0,\partial T_y} \le C \big( h_x^{3/2} \|v_{xxx}\|_{0,T} + \sqrt{h_x}h_y \|v_{xxy}\|_{0,T} + \frac{h_y^2}{\sqrt{h_x}} \|v_{xyy}\|_{0,T} \big).
	\end{gather}
\end{lem}
	\begin{proof} Lemma \ref{lem:amti} and Young's inequality yield
	\begin{gather*}
		\left\|\left(v - v^I \right)_x \right\|^2_{0,\partial T_y} \le C \big( \frac{1}{h_x} \left\|\left(v - v^I \right)_x \right\|^2_{0,T} + h_x \left\|\left(v - v^I \right)_{xx} \right\|^2_{0,T} \big).
	\end{gather*}
	With the well known anisotropic nodal interpolation error estimates for $v \in H^{3}(T)$:
	\begin{align*}		
		\left\|\left(v - v^I \right)_x \right\|_{0,T} &\le C \big( h_x^2 \|v_{xxx}\|_{0,T} + h_x h_y \|v_{xxy}\|_{0,T} + h_y^2 \|v_{xyy}\|_{0,T} \big),\\
		\left\|\left(v - v^I \right)_{xx} \right\|_{0,T} &\le C \big( h_x \|v_{xxx}\|_{0,T} + h_y \|v_{xxy}\|_{0,T} \big),
	\end{align*}
	we complete the proof.
	\end{proof}

\begin{lem}\label{lem:jumpomega}
	Assume $|S|_3 \le C$ and $\eps^{1/2} \le (\ln N)^{-2}$ then there is a positive constant $C$ such that
	\begin{align}
			\sum_{e \in \mathcal{E}(I)} \bigg\| \bigjump{\frac{\partial(u-u^\star)}{\partial n}} \bigg\|_{0,e}^2 &\le C N^{-3},\label{eq:jumpomega0}\\
			\sum_{e \in \mathcal{E}(III)} \bigg\| \bigjump{\frac{\partial(u-u^\star)}{\partial n}} \bigg\|_{0,e}^2 &\le C \eps^{-1/2} N^{-3} (\ln N)^4.\label{eq:jumpomega}
	\end{align}
\end{lem}
\begin{proof}
	Recall that by construction the normal derivative of $u^\star$ is continuous across type-II and type-IV edges, i.e.~across long edges of anisotropic elements and within the subdomains close to the four corners of $\Omega$.	Let $e \subset \overline{\Omega_0}$ be a type-I edge.
	Since $u^\star$ is defined by nodal interpolation on the the two elements $T_1$ and $T_2$ that share the edge $e$ we find with Lemma \ref{lem:tracelem} that
	\begin{gather*}
		\bigg\| \bigjump{\frac{\partial(S-S^I)}{\partial n}} \bigg\|_{0,e} = \bigg\| \left.\frac{\partial(S-S^I)}{\partial n}\right|_{T_1} \bigg\|_{0,e} + \bigg\| \left.\frac{\partial(S-S^I)}{\partial n}\right|_{T_2} \bigg\|_{0,e} \le C H^{3/2} \big( |S|_{3,T_1} + |S|_{3,T_2} \big).
	\end{gather*}
Hence, 
	\begin{gather}\label{eq:smoothNormalJumpOmega0}
		\sum_{e \in \mathcal{E}^N_{int}(\Omega_0)} \bigg\| \bigjump{\frac{\partial(S-S^I)}{\partial n}} \bigg\|_{0,e}^2 \le C H^3 |S|_{3,\Omega_0}^2 \le C N^{-3}.
	\end{gather}
	Next we abbreviate $E = u-S$. In $\Omega_0$ the layer components $E$ are pointwise small and smooth, hence on a type-I  edge $e \in \mathcal{E}(I)$ we use inverse estimates to obtain
	\begin{align*}
		\bigg\| \bigjump{\frac{\partial(E-E^I)}{\partial n}} \bigg\|_{0,e} &= \bigg\| \bigjump{\frac{\partial(E^I)}{\partial n}} \bigg\|_{0,e} = \bigg\| \left.\frac{\partial(E^I)}{\partial n}\right|_{T_1} \bigg\|_{0,e} + \bigg\| \left.\frac{\partial(E^I)}{\partial n}\right|_{T_2}\bigg\|_{0,e}\\
		&\hspace*{-1cm}\le C H^{-1/2} \bigg( \bigg\| \frac{\partial(E^I)}{\partial n} \bigg\|_{0,T_1} + \bigg\| \frac{\partial(E^I)}{\partial n} \bigg\|_{0,T_2} \bigg) \le C H^{-3/2} \big(\|E^I\|_{0,T_1} + \|E^I\|_{0,T_2}\big).
	\end{align*}
	A summation over all type-I edges then yields
	\begin{gather}\label{eq:layerNormalJumpOmega0}
		\sum_{e \in \mathcal{E}(I)} \bigg\| \bigjump{\frac{\partial(E-E^I)}{\partial n}} \bigg\|_{0,e}^2 \le C H^{-3} \|E^I\|_{0,\Omega_0}^2 \le C H^{-3} \|E\|_{\infty,\Omega_0}^2 \le C H^{-3} N^{-2\lambda_0} \le C N^{-3}.
	\end{gather}
	Combining \eqref{eq:smoothNormalJumpOmega0} and \eqref{eq:layerNormalJumpOmega0} we arrive at \eqref{eq:jumpomega0}.
	
	It remains to estimate the jump of the normal derivative across short edges of anisotropic elements which are of type III. Let $e = T_1 \cap T_2 \subset \overline{\Omega_1}$ denote such an edge. We shall first deal with the case that $T_1$ and $T_2$ are anisotropic elements. Again, we split $u$ into smooth and layer components and estimate
	\begin{gather*}
		\big\|\jump{\boldsymbol{D}^{(1,0)}(S - \Pi^y S)} \big\|_{0,e} \le \big\| \boldsymbol{D}^{(1,0)}(S - \Pi^y S)|_{T_1} \big\|_{0,e} + \big\| \boldsymbol{D}^{(1,0)}(S - \Pi^y S)|_{T_2} \big\|_{0,e}.
	\end{gather*}
	Lemma \ref{lem:amti} gives for the smooth part
	\begin{align*}
		\big\| \boldsymbol{D}^{(1,0)}(S - \Pi^y S)|_{T} \big\|_{0,e}^2 &\le C \bigg( \big\| \boldsymbol{D}^{(1,0)}(S - \Pi^y S) \big\|_{0,T} \big\| \boldsymbol{D}^{(2,0)}(S - \Pi^y S) \big\|_{0,T}\\
		&\quad + \frac{1}{H} \big\| \boldsymbol{D}^{(1,0)}(S - \Pi^y S) \big\|_{0,T}^2 \bigg)\\
		& \le C ( H^2 H + H^{-1} H^4 ) |S|_{3,T}^2 \le C N^{-3} |S|_{3,T}^2.
	\end{align*}
	A summation of all type-III edges then yields
	\begin{gather}\label{eq:smoothNormalJumpOmegai}
		\sum_{e\in \mathcal{E}(III)} \Big\| \bigjump{\frac{\partial}{\partial n}(S - S^\star)} \Big\|_{0,e}^2 \le C N^{-3} |S|_{3,\bigcup_{i=1}^4 \Omega_i}^2 \le C N^{-3}.
	\end{gather}	
	With the layer component $E_1$ we proceed in a similar manner
	\begin{align*}
		\big\| \boldsymbol{D}^{(1,0)}(E_1 - \Pi^y E_1)|_{T} \big\|_{0,e}^2 &\le C \bigg( \big\| \boldsymbol{D}^{(1,0)}(E_1 - \Pi^y E_1) \big\|_{0,T} \big\| \boldsymbol{D}^{(2,0)}(E_1 - \Pi^y E_1) \big\|_{0,T}\\
		&\quad + \frac{1}{H} \big\| \boldsymbol{D}^{(1,0)}(E_1 - \Pi^y E_1) \big\|_{0,T}^2 \bigg).
	\end{align*}
	A summation gives with \eqref{eq:H1EiOmegaix} and \eqref{eq:H2EiOmegaixx}
	\begin{gather}\label{eq:E1NormalJumpOmegai}
		\sum_{e\in \mathcal{E}(III)} \Big\| \bigjump{\frac{\partial}{\partial n}(E_1 - E_1^\star)} \Big\|_{0,e}^2 \le C \eps^{-1/2} N^{-3} (\ln N)^4.
	\end{gather}	
	Any other layer component $E \neq E_1$ can handled similarly as in the interior subdomain $\Omega_0$:
	\begin{gather*}
		\big\| \jump{\boldsymbol{D}^{(1,0)}(E- \Pi^y E)} \big\|_{0,e} \le C H^{-3/2} \big(\| \Pi^y E\|_{0,T_1} + \|\Pi^y E\|_{0,T_2}\big).
	\end{gather*}
	Hence,
	\begin{gather}\label{eq:layerNormalJumpOmegai}
		\sum_{e\in \mathcal{E}(III)} \Big\| \bigjump{\frac{\partial}{\partial n}(E- E^\star)} \Big\|_{0,e}^2 \le C H^{-3} \big( \| \Pi^y E\|_{0,\Omega_1\cup\Omega_3}^2 + \| \Pi^x E\|_{0,\Omega_2\cup\Omega_4}^2 \big) \le C \eps^{1/2} N^{-3} \ln N,
	\end{gather}
	as shown in \eqref{eq:EiOmegaj}.
	In order to estimate the jump of the normal derivative of $u - u^\star$ across short interior edges of for instance $\Omega_1$ it remains to estimate the jump of the $x$-derivative of the term
	\begin{gather*}
		\sum_{\substack{i=N/2 \\ {j \in \{N/4, 3N/4\}}}}^{3 N/2}\!\!\!\!\! \left.\frac{\partial(u^I-u)}{\partial y} \right|_{\Omega_0}\!\!\!\!(x_{i/2},y_j)\, \ell_{i/2}(x) \psi_j(y)	
	\end{gather*}
	across these edges. With Lemma \ref{lem:amti} and \eqref{eq:modestim} one easily sees that this term is better behaved than $\jump{\boldsymbol{D}^{(1,0)}(u- \Pi^y u)}$.
	
	Finally, we consider type-III edges that are shared by an anisotropic element and a small square shaped one in the subdomains close to the corners of $\Omega$. The common edge is then a subset of $\partial \Omega_f \setminus \partial \Omega$. Let for instance $T_1 \in \overline{\Omega_1}$ and $T_2 \in \overline{\Omega_{12}}$ denote such elements. Then the normal derivative of $u^\star$ jumps across the common edge at $x=\lambda$.
	Since
	\begin{gather*}
		\bigg\| \bigjump{\frac{\partial(u - u^\star)}{\partial n}} \bigg\|_{0,e} = \bigg\| \left.\frac{\partial(u- u^\star)}{\partial n}\right|_{T_1} \bigg\|_{0,e} + \bigg\| \left.\frac{\partial(u-\tilde \Pi u)}{\partial n}\right|_{T_2} \bigg\|_{0,e},
	\end{gather*}
	we can estimate the first summand like before and it remains to estimate the second one. We start off with a trace inequality
	\begin{gather*}
			\big\| \boldsymbol{D}^{(1,0)}(u-\tilde \Pi u)|_{T_2} \|_{0,e}^2 \le C \big( \frac{1}{h} | u - \tilde \Pi u |^2_{1,T_2} + h | u - \tilde \Pi u |^2_{2,T} \big).
	\end{gather*}
	Hence, with \eqref{eq:tildePiuniform}:
	\begin{gather}\label{eq:NormalJumpOmegaij}
		\begin{aligned}
			\big\| \boldsymbol{D}^{(1,0)}(u-\tilde \Pi u)|_{\Omega_{12}} \|_{0,x=\lambda}^2 &\le C \big( \frac{1}{h} | u - \tilde \Pi u |^2_{1,\Omega_{12}} + h | u - \tilde \Pi u |^2_{2,\Omega_{12}} \big)\\
			&\le C h^3 |u|_{3,\Omega_{12}}^2 \le C h^3 \meas(\Omega_{12}) |u|_{W_{3,\infty}(\Omega_{12})}^2\\
			&\le C \eps^{3/2} N^{-3} (\ln N)^3 \eps \ln N \eps^{-3} = C \eps^{-1/2} N^{-3} (\ln N)^4.
		\end{aligned}
	\end{gather}
	Collecting \eqref{eq:jumpomega0}, \eqref{eq:smoothNormalJumpOmegai}, \eqref{eq:E1NormalJumpOmegai}, \eqref{eq:layerNormalJumpOmegai} and \eqref{eq:NormalJumpOmegaij} we arrive at \eqref{eq:jumpomega} and finish the proof.
\end{proof}
	
	\begin{rem}\label{rem:S3}
		Under additional compatibility conditions on the right hand side $f$ it should be possible to remove the dependency of the third-order derivatives of the smooth part $S$ on $\eps$ in \eqref{eq:soldecestimS}, giving $\|S\|_3 \le C$. However, assuming $|S|_3 \le C$ is of course weaker than requiring that all third-order derivatives of $u$ are pointwise bounded uniformly with respect to $\eps$.
	\end{rem}
	
	\begin{rem}
		Let $e$ denote a horizontal long edge of an anisotropic macro-element. The interpolation operator $\Pi^y$ features a stability of the form
		\begin{gather*}
			\|(\Pi^y v)_y\|_{\infty,e} \le C \|v_y\|_{\infty,e}.
		\end{gather*}
		However, this seems to lead only to the estimate $\|(\Pi^y E_1)_y\|_{0,e}^2 = \Landau(\eps^{-1})$ which is not good enough for our purposes. That is why we use a modification of $\tilde \Pi$ in the definition of $u^\star$ in order to match the normal derivatives on both sides of $\partial \Omega_0$.
	\end{rem}

\FloatBarrier

\end{document}